\documentclass[a4paper]{amsart} 
\usepackage{graphicx,epsfig,amscd,graphics,xypic}
\usepackage{amssymb}
\usepackage{amscd}
\usepackage{times}
\usepackage{enumitem}
\usepackage{mathrsfs}
\usepackage[latin1]{inputenc}
\usepackage{psfrag}
\usepackage[active]{srcltx}
\usepackage{color}
\usepackage{wrapfig}

\usepackage{url}

\usepackage{float}

\usepackage{pstricks,pstricks-add,pst-math,pst-xkey}
\usepackage{subfigure}

\usepackage{graphicx,epsfig,amscd,graphics,xypic}

\usepackage[dvips,all]{xy}

\usepackage{graphics}

\usepackage[margin=1in]{geometry}

\newtheorem{thm}{Theorem}[section]
  
\newtheorem{coro}[thm]{Corollary}  
\newtheorem*{thm*}{Theorem}   
\newtheorem*{lemma*}{Lemma}

\newtheorem{prop}[thm]{Proposition}

\newtheorem{lemma}[thm]{Lemma}
\newtheorem{rem}[thm]{Remark}
\newtheorem{ex}[thm]{Example}
\newtheorem*{ex*}{Example \ref{exJ2} (continued)}

\newtheorem*{XJC}{$XJC$-Principle}

\newtheorem*{Ex:Ce/e3}{Example \ref{Ex:Ce/e3} (continued)}

\newcommand{\Sing}{\operatorname{Sing}}

\newcommand{\p}{{\mathbb P}}

\newcommand{\Sec}{\operatorname{Sec}}

\newcommand{\rk}{\operatorname{rk}}

\newcommand{\map}{\dasharrow}

\newcommand{\sat}{\operatorname{sat}}

\renewcommand{\th}{\operatorname{th}}
\newcommand{\reg}{\operatorname{reg}}

\def\S{Section~}
\newcommand{\End}{\operatorname{End}}
\newcommand{\Rad}{\operatorname{Rad}}

\newcommand{\Hess}{\operatorname{Hess}}
\renewcommand{\ln}{\operatorname{ln}}

\renewcommand{\ss}{\operatorname{ {\rm ss}}}

\def\og{\leavevmode\raise.3ex\hbox{$\scriptscriptstyle\langle\!\langle$~}}
\def\fg{\leavevmode\raise.3ex\hbox{~$\!\scriptscriptstyle\,\rangle\!\rangle$}}

\newcommand{\incl}[1][r]
  {\ar@<-0.2pc>@{^(-}[#1] \ar@<+0.2pc>@{-}[#1]}

\newcommand{\eq}[1][r]
   {\ar@<-3pt>@{-}[#1]
    \ar@<-1pt>@{}[#1]|<{}="gauche"
    \ar@<+0pt>@{}[#1]|-{}="milieu"
    \ar@<+1pt>@{}[#1]|>{}="droite"
    \ar@/^2pt/@{-}"gauche";"milieu"
    \ar@/_2pt/@{-}"milieu";"droite"}

\begin{document}

\title{Extremal varieties 3-rationally connected by cubics,\\ quadro-quadric Cremona transformations \\  and rank 3 Jordan algebras}  

\author{Luc PIRIO  AND Francesco RUSSO} 
\address{Luc PIRIO, IRMAR, UMR 6625 du CNRS, Universit\' e Rennes1, Campus de beaulieu, 
35000 Rennes, France}
\email{luc.pirio@univ-rennes1.fr} 
\address{Francesco RUSSO, Dipartimento di Matematica e Informatica, Universit\`a degli Studi di Catania, Viale A. Doria 6,  95125  Catania, Italy}
\email{frusso@dmi.unict.it}

\begin{abstract} 
For any $n\geq 3$, we prove that there are  equivalences between \smallskip 
\begin{itemize}
\item  irreducible $n$-dimensional non degenerate projective varieties $X\subset \mathbb P^{2n+1}$ different from rational normal scrolls and  3-covered  by rational cubic curves, up to projective equivalence;\smallskip 
\item quadro-quadric Cremona transformations of $ \mathbb P^{n-1}$, up to linear equivalence;\smallskip 
\item  $n$-dimensional complex Jordan algebras of rank three, up to isotopy.\medskip 
\end{itemize}

We also provide some applications to the classification of particular
classes of varieties in the class defined above and of quadro-quadric Cremona transformations,
proving also  a structure theorem for these
birational maps and for varieties 3-covered by twisted cubics by reinterpreting 
for these objects the solvability of the radical of a Jordan algebra.
\end{abstract}

\maketitle

\section*{Introduction}
In this paper we continue the study began in \cite{PR} of  the unexpected relations between the following three sets: $n$-dimensional complex Jordan algebras of rank three modulo isotopy;  irreducible $n$-dimensional projective varieties $X\subset\p^{2n+1}$ such
that through three general points there passes a twisted cubic contained in it modulo projective equivalence; quadro-quadric Cremona transformations in $\p^{n-1}$ modulo linear equivalence.\smallskip

Jordan algebras have been introduced  by physicists around 1930 in the attempt of  discovering 
a non-associative algebraic setting for quantum mechanics. These algebras found later applications  in many different areas of mathematics, spanning from  Lie algebras and group theory to real and complex differential geometry, see for example \cite[Part I]{McCrimmon-book} for a general panorama. 
In algebraic geometry, complex Jordan algebras of rank three were used to construct projective varieties with notable geometric properties
either by considering some determinantal varieties associated to the simple finite dimensional ones
such as Severi varieties, see \cite[IV.4.8]{Zak},  
or by defining the so called {\it twisted cubic over a rank three Jordan algebra}, see 
\cite{Freudenthal}, \cite{mukai} \cite[\S 4]{PR},  
\cite{landsbergmanivel2},\cite{landsbergmanivel3} and \S 3 below. 
These last objects are examples of projective varieties such that through three general points 
there passes a twisted cubic contained in it
and they also appear as the first exceptional 
examples to the classification of  extremal varieties $m$--covered by rational 
 curves of fixed degree, see \cite{PT} and \cite{PR} for definitions and examples and also \S 1 and
\S 3. Moreover, twisted cubic over rank three Jordan algebras 
are also  examples of varieties with one apparent double point,
 see \cite{mukai}, \cite[Corollary 5.4]{PR} and \cite{CR}, and the smooth ones 
are also {\it Legendrian} varieties, see \cite{mukai} and \cite{landsbergmanivel3}.

Quadro-quadric Cremona transformations can be considered as the simplest examples of  birational maps of a projective space different from  linear automorphisms.
In the plane these transformations are completely classified and together with projective automorphisms generate the group of birational maps of $\p^2$. In low dimension they were studied classically by the Italian school, see for example  \cite{Conforto} and the references therein, and soon later by Semple \cite{semple}.
These results were reconsidered recently in \cite{PRV}, where the classification in $\p^3$ originally outlined in \cite{Conforto} is completed, see also  \cite{brunoverra}.
In \cite{EinShB} it is proved the surprising and nice  result that there are only four examples of quadro-quadric Cremona transformations with smooth irreducible base locus.
 These four examples  are related to  the so called  Severi varieties
and are linked to the four simple complex Jordan algebras of  hermitian $3\times 3$ matrices 
with coefficients in the complexification of
the four real division algebras $\mathbb R$, $\mathbb C$, $\mathbb H$ and $\mathbb O$, see 
\cite{EinShB}, \cite{Zak}, and also \cite{EKP}, \cite{CS} and  Corollary \ref{C:chaputsab} here. \medskip

The main results of the paper, collected in Theorem  \ref{T:main} and in the related diagram,  assert that the three sets described  above are in bijection and that the composition of two of these bijections is the identity map. This correspondence, which we call ``{\it $XJC$-correspondence}'',  was conjectured in the final remarks of \cite{PR} and it is based on
the following results: every quadro--quadric Cremona transformation of $\p^{n-1}$ is linearly equivalent to an involution which is the adjoint
of a rank 3 Jordan algebra of dimension $n$ (Theorem \ref{P:jordanXf}); every irreducible $n$-dimensional variety $X^n\subset\p^{2n+1}$ which is 3--covered by twisted
cubics and different from a rational normal scroll is projectively equivalent to a twisted cubic over a rank three complex Jordan algebra (Theorem  \ref{T:XJ}). Some particular versions of the $XJC$-correspondence are the following: cartesian products of varieties $3$--covered by twisted cubics
correspond to direct product Jordan algebras of rank three 
and to the so called elementary quadratic transformations (Proposition \ref{P:product}); 
smooth varieties 3--covered by twisted cubics, modulo projective equivalence,
 are in bijection with semi-simple rank three  Jordan algebras, modulo isotopy, 
and with {\it semi--special}
 quadro-quadric Cremona transformations, modulo linear equivalence (Theorem \ref{T:simple}).\smallskip 

The $XJC$--correspondence is extended in \S \ref{S:generalizedXJC} to cover some degenerated cases: rational normal scrolls, Jordan algebras with a cubic norm
 and  `fake' quadro-quadric Cremona transformations, respectively. Moreover, the $XJC$--correspondence leads us to some new constructions and definitions. The theory of the radical  and 
the semi--simple part of a Jordan algebra suggested  the definitions of  semi--simple part,
semi--simple rank and semi-simple dimension  of
a quadro-quadric Cremona transformation, or of an extremal variety 3--covered by twisted cubics, providing for instance a general Structure Theorem for these maps,
see Theorem \ref{T:structureBir22}. As an application we prove in Corollary \ref{C:chaputsab} that every homaloidal
polynomial $f$  of degree 3 defining a quadro-quadric Cremona transformation 
whose  ramification locus scheme is cut out by $f$ is,  modulo linear equivalence,
the norm of a rank 3 semi-simple Jordan algebra,
 providing a new short proof of \cite[Theorem 3.10]{EKP} 
and of  \cite[Theorem 2, Corollary 4]{CS}. 
\medskip 

The paper is organized as follows. 
In \S 1 we introduce some notation which is not standard. In \S 2 we define precisely the objects studied giving some examples: the $X$-world consisting of extremal varieties $3$--covered by twisted cubics; the $C$--world consisting of quadro-quadric Cremona transformations and the $J$--world consisting of rank three complex Jordan algebras. Moreover the natural equivalence relations: projective equivalence, linear equivalence, respectively isotopy are introduced as well the notion of cubic Jordan pair. In \S 3 we define the correspondences between the three sets modulo equivalences. First from the $J$--world to the $X$ and $C$ worlds. Then from the $C$--world to the $X$--world. We prove the  equivalence between the $C$ and $J$ worlds in Theorem \ref{P:jordanXf} while the equivalence between the $X$ and $J$ worlds is proved in Theorem \ref{T:XJ}. The $XJC$-correspondence and its particular 
forms recalled above are stated in Section 4 while \S 5 is devoted to the Structure Theorem of quadro-quadric 
Cremona transformation, Theorem \ref{T:structureBir22}, and to the reinterpretation of the theory of the solvability of the radical of a Jordan algebra in the $C$--world and in the $X$--world.

\section{Notation}

If $V$ is a  complex vector space of finite dimension and if $A\subset V$ is a subset,
then $\langle A\rangle$  denotes the smallest linear subspace of $V$ containing $A$, analogous
notions being defined in $\p(V)$.
  The  projective equivalence class  of  $x\in V\setminus \{0\}$ is the element $[x]\in\p(V)$.
  Let $P_1,P_2$ be two projective subspaces in $\mathbb P^N$. 
When $P_1\cap P_2=\emptyset$, we define their {\it direct sum}  as 
$P_1\oplus P_2=\langle P_1,P_2\rangle \subset \p^N$.
\smallskip 
 
 We shall consider
(irreducible) algebraic varieties defined over the complex field.
If $X$ is an  irreducible algebraic variety and if $n=\dim(X)$, we shall write 
 $X=X^n$ or simply $X^n$.   We denote by $[X]$ the projective equivalence class of an irreducible projective variety $X\subset\p^N$. We shall indicate by 
 {$(X)^m$} the $m$-times cartesian product $X\times\cdots\times X$. We denote by
 $T_x X$ the embedded projective tangent space to $X\subset\p^N$ at a smooth point $x$ o f $ X$ while
$T_{X,x}$ indicates   the abstract tangent space to $X$ at $x$.

The irreducible quadric hypersurface in $\p^{r+1}$ is denoted by $Q^r$ while $v_3(\mathbb P^1)\subset\p^3$
is the twisted cubic curve. 
\smallskip

\section{The objects}

\subsection{The $X$-world:   varieties  $X^n(3,3)$  }
An irreducible  projective variety $X=X^n\subset \mathbb P^N$ is said to be  {\it $3$-rationally connected by cubic curves} ({\it 3-RC by cubics} for short) if for a general  3-uplet of points  $x=(x_i)_{i=1}^3 \in (X)^3$, there exists an irreducible rational cubic curve included in $X$ that passes through $x_1,x_2 $ and $x_3$.

If $X\subset \p^N$ is 3--RC by cubics, then projecting $X$ from a general projective tangent space $T_xX$ we get an irreducible variety $Y^{n-\delta}\subset\p^{N-n-1}$, $\delta\geq 0$, such
that through two general points there passes a line contained in $Y^{n-\delta}$. This immediately implies $Y=\p^{n-\delta}$ so that:
\begin{equation}
\dim \langle X\rangle \leq 2n+1-\delta\leq 2n+1,
\end{equation}
see also  \cite[\S 1.2]{PT} for more general results and formulations.

We will say that a variety $X\subset\p^N$ 3--RC by cubics  is {\it extremal} if $N=
\dim \langle X\rangle= 2n+1$. In what follows, we shall use the notation $X=X^n({3,3})$  when $X\subset\p^{2n+1}$  is an extremal  variety which is 3--RC by cubics.

Thus for  $X=X^n(3,3)\subset\p^{2n+1}$ and for two general points $x_1,x_2\in X$, we have
$$\mathbb P^{2n+1}=\langle X \rangle = T_{x_1}X\oplus T_{x_2}X,$$  
see also \cite[Lemme 1.3]{PT}.\medskip

{\begin{ex} 
\rm{
\label{EX:X(3,3)}
\begin{enumerate}
\item There exists a unique 3-RC curve $X^1(3,3)$: 
the twisted cubic cubic curve $v_3(\mathbb P^1)\subset \mathbb P^3$; \medskip 
\item Let $Q$ be an irreducible hyperquadric in $\mathbb P^n$.
It is well-known that $Q$ is 3-RC by conics and since $\mathbb P^1$ 
is 3-covered by lines(!),
it immediately follows that the Segre product ${\rm Seg}(\mathbb P^1\times Q)\subset\p^{2n+1}$ 
is 3-RC by cubics so that  ${\rm Seg}(\mathbb P^1\times Q)=X^n(3,3)$. These examples produce a family of $X^n(3,3)$ for every $n\geq 2$; 
 \medskip 
\item  Let $(\Pi_i)_{i=1}^3$ be a  3-uple of elements of the grassmannian variety  $\mathbb G(2,5)=G_3(\mathbb C^6)\subset \mathbb P^{19}$,   Pl\"ucker embedded.  If  the $\Pi_i$'s are general, one can find a basis $(u_i)_{i=1}^6$ of $\mathbb C^6$ such that $\Pi_1=u_1\wedge u_2\wedge u_3$,  $\Pi_2=u_4\wedge u_5\wedge u_6$
and $\Pi_3=(u_1+u_4)\wedge (u_2+u_5)\wedge ( u_3+u_6)$.  Then $ s\mapsto (u_1+s u_4)\wedge (u_2+s u_5)\wedge ( u_3+ s u_6)$ extends to a morphism $\varphi: \mathbb P^1\rightarrow G_3(\mathbb C^6)$ such that $\varphi(0)=\Pi_1, \varphi(\infty)=\Pi_2$ and $\varphi(1)=\Pi_3$.  The curve $\varphi(\mathbb P^1)\subset\mathbb G(2,5)\subset\mathbb P^{19}$ in  the 
Pl\"ucker embedding is a twisted cubic, showing that $\mathbb G(2,5)=X^{9}(3,3)$; 
\medskip 
 \item 
 The $n$-dimensional rational normal scrolls  $S_{1\ldots 13}\subset\p^{2n+1}$, $n\geq 1$, and $
 S_{1\ldots 122}\subset\p^{2n+1}$, $n\geq 2$, are classical examples of $X^n(3,3)$, which we shall call {\it degenerated examples},
 see \cite{PT} for the explanation of the terminology and also
 Section   \ref{S:fromXworld} below.
 \end{enumerate}}
\end{ex}}

We shall denote  by $\boldsymbol{X}^n(3,3)$ the set of irreducible non-degenerate varieties $X^n\subset \mathbb P^{2n+1}$ which are 3-RC by twisted cubics and which are not  degenerated  in the above sense, {\it i.e.} that are different from  $S_{1\ldots 13}$ or $
 S_{1\ldots 122}$.  The description of  the projective equivalence classes of 
  elements in $\boldsymbol{X}^n(3,3)$   is a natural geometrical problem already considered in \cite{PR}, see also \cite{PT} for general
  classification results of this kind. Indeed, this problem naturally appears when trying to solve the 
 question on which  maximal rank webs are algebraic,  a central problem in web geometry, see \cite{PTweb}. \smallskip 

Remind that if $X\in \boldsymbol{X}^n(3,3)$, we shall denote by $[X]$ its projective equivalence class.

\subsection{The $C$-world:  Cremona transformations of bidegree $(2,2)$}
\label{S:S-C}

Let $f: \mathbb P^{n-1}\dashrightarrow \mathbb P^{n-1}$ be a rational map. 
There exist  a unique integer $d\geq 1$ and $f_i \in |\mathcal O_{\mathbb P^{n-1}}(d) |$, $i=1,\ldots, n,$ with ${\rm gcd}(f_1,\ldots,f_n)=1$ such that 
$$
f(x)=\big[f_1(x):\cdots: f_n (x)    \big]
$$
for $x\in \mathbb P^{n-1}$ outside the {\it base locus scheme} $\mathcal B=V(f_1,\ldots, f_n)\subset\p^{n-1}$ of $f$. 
By definition, the {\it degree} of $f$ is $\deg(f)=d$.
 We will denote by $F: \mathbb C^{n}\rightarrow \mathbb C^{n}$  the
 homogeneous affine {polynomial} map 
defined by $ F(x)=(f_1(x),\ldots, f_n (x))$   
for $x\in \mathbb C^{n}$. Note that the projectivization of $F$ is  of course
 the rational map $f$ and that  {$F$} depends on $f$ only up to multiplication by a nonzero constant.

A rational  map $f: \mathbb P^{n-1}\dashrightarrow \mathbb P^{n-1}$ is  {\it birational} (or is a {\it Cremona transformation}) if it admits a rational inverse $f^{-1}: \mathbb P^{n-1}\dashrightarrow \mathbb P^{n-1}$. In this case, one defines  
the {\it bidegree} of $f$ as ${\rm bideg}(f)=(\deg f, \deg f^{-1})$.  In this paper we will mainly consider  {\it quadro-quadric} Cremona transformations, 
that is   Cremona transformations of bidegree $(2,2)$.   The set of such birational maps of $\p^{n-1}$ will be indicated by ${\bf Bir}_{2,2}(\mathbb P^{n-1})$.

\begin{ex} \rm{
\label{EX:bir}
\begin{enumerate}
\item The {\it standard  involution} of $\mathbb P^{n-1}$ is the birational map
$$
[x_1:x_2:\ldots: x_n]\longmapsto \big[x_2x_3\ldots x_n : x_1x_3\ldots x_n:\ldots:x_1x_2\ldots x_{n-1}  \big].
$$
It has bidegree $(n-1,n-1)$ and it  is an {\it involution}, that is $f=f^{-1}$ or equivalently  $f\circ f$ is equal to  the identity  of $\mathbb P^{n-1}$ as a rational map; \medskip 
\item 
\label{EX:bir-fake}
Assume that $ x\mapsto (\ell_0(x),\ldots,\ell_n(x))$ is a linear automorphism of $\mathbb C^n$. Then for any nonzero linear form $\ell:\mathbb C^n\rightarrow \mathbb C$,  the map $ x\mapsto [\ell(x)\ell_0(x):\cdots : 
\ell(x)\ell_n(x)]$ is  a birational map. With the previous definitions it is a birational map of bidegree $(1,1)$ but we shall consider such a map  as a {\it  fake quadro-quadric Cremona transformation}, see Section \ref{S:generalizedXJC}; 
\medskip 

\item  Let $Q^{n-1}\subset\p^n$ be an irreducible hyperquadric. Given $p\in Q_{\reg}$, the projection 
 from $p$ induces a birational map $\pi_p : Q\dashrightarrow \mathbb P^{n-1}$.   For $p,p'\in Q_{\reg}$
with 
$p'\not\in T_pQ$, the composition  $\pi_{p'}\circ \pi_p^{-1}:\p^{n-1}\map\p^{n-1}$
 is a birational map of bidegree $(2,2)$, called  an  {\it elementary quadratic transformation};
\medskip

\item\label{ex:rank}  Let $J$ be a  finite dimensional power-associative algebra. The {\it inversion}  $x\dashrightarrow x^{-1}$ induces a birational involution  
$j:\p(J)\map\p(J)$.  If $J$ has rank $r$, see \ref{S:Jworld} for the definitions, then $j$ is of bidegree  $(r-1,r-1)$. 
\end{enumerate}
}
\end{ex}

For simplicity, 
we denote  by  $V$ the vector space $\mathbb C^n$ in the lines below.\smallskip

Let $f_1,\ldots, f_n $ and $g_1,\ldots, g_n $ be quadratic forms on $V$ defining  the affine polynomial maps 
$F=(f_1,\ldots, f_n):V\to V $, respectively  $G=(g_1,\ldots, g_n):V\to V $. Let $f:\p^{n-1}\map\p^{n-1}$, respectively  $g:\p^{n-1}\map\p^{n-1}$, be the induced rational maps.
Then $g=f^{-1}$ as rational maps if and only if  there exist homogeneous cubic forms $N,M\in {\rm Sym}^3(V^*)$ such that, for every $x,y\in V$: 
\begin{equation}
\label{E:bir22}
G\big( F(x)   \big)=N(x)\, x \qquad \mbox{ and } \qquad 
F\big( G(y)   \big)=M(y)\, y.
\end{equation}
 
In the previous case, one easily verifies that  for every $x,y\in V$, we also have 
\begin{equation}
\label{E:bir22-2}
M\big( F(x)   \big)=N(x)^2 \qquad \mbox{ and } \qquad 
N\big( G(y)   \big)=M(y)^2.
\end{equation}

 Two Cremona transformations  $f, \tilde f: \mathbb P^{n-1}\dashrightarrow \mathbb P^{n-1}$ are said to be {\it linearly equivalent} (or just {\it equivalent} for short) if there exist  projective transformations  $\ell_1,\ell_2: \mathbb P^{n-1}\rightarrow \mathbb P^{n-1}$ such that $ \tilde f= \ell_1 \circ f\circ  \ell_2$.  
This is an equivalence relation on 
  ${\bf Bir}_{2,2}(\mathbb P^{n-1})$ and in the sequel we  shall investigate the quotient space ${\bf Bir}_{2,2}(\mathbb P^{n-1}){\!\big/ \!\!\!
{\tiny{\begin{tabular}{l}
$linear$ \vspace{-0.08cm}  \\
$equivalence$
\end{tabular}}
}}$  and its various incarnations.\smallskip 

If $f\in {\bf Bir}_{2,2}(\mathbb P^{n-1})$ we will denote by $[f]$ its linear equivalence class.

%%%%%%%%%%%%%%%%%%%%%%%%%%%%%%%%%%%%%%%%%%%%%
%%%%%%%%%%%%%%%%%%%%%%%%%%%%%%%%%%%%%%%%%%%%%
%%%%%%%%%%%%%%%%%%%%%%%%%%%%%%%%%%%%%%%%%%%%%
%%%%%%%%%%%%%%%%%%%%%%%%%%%%%%%%%%%%%%%%%%%%%
\subsection{The $J$-world: Jordan algebras  and Jordan pairs of degree 3}
${}^{}$ \medskip 
\label{S:Jworld}

By definition, a {\it Jordan algebra} is a commutative complex algebra $\mathbb J$  with a unity $e$ such that the  {\it Jordan identity}
\begin{equation}
\label{E:jordan}
x^2(xy)=x(x^2y)
\end{equation}
holds for every $x,y\in \mathbb J$ (see \cite{jacobson, McCrimmon-book}).  Here we shall also assume that $\mathbb J$ is
finite dimensional.
It is well known that a Jordan algebra is power-associative.
 By definition, the {\it rank} $\rk(\mathbb J)$  of  $\mathbb J$ 
is the complex dimension of the (associative) subalgebra  
$\langle x \rangle $ of $ \mathbb J$ spanned by the unity $e$ and by
 a general element $x\in  \mathbb J$. A  general element
 $x\in \mathbb J$ is invertible, {\it i.e.} for $x$ in an open nonempty subset 
of $\mathbb J$, there exists a unique 
$x^{-1}\in \langle x\rangle$ such that $xx^{-1}=e=x^{-1}x$.

\begin{ex} \rm{
\label{EX:jordan}
\begin{enumerate}
\item 
Let $A$ be a non-necessarily commutative associative algebra with a unity.  Denote by $A^+$ the vector space $A$  with the symmetrized product $a\cdot a'=\frac{1}{2}(aa'+a'a)$.  Then $A^+$ is a Jordan algebra. Note that   $A^+=A$ if  $A$ is commutative.\medskip 
\item  
\label{Ex:rank2}
Let $q:W\to\mathbb C$ be a quadratic form on the vector space $W$. For $(\lambda,w),(\lambda',w')\in \mathbb C\oplus W$, the product $(\lambda,w)*(\lambda',w')=(\lambda \lambda'-q(w,w'),\lambda w'+\lambda'w)$  induces a structure of rank  2 Jordan algebra on 
$\mathbb C\oplus W$ with unity $e=(1,0)$.
 \medskip 
\item
\label{L:Herm3}
 Let $\mathbb A$ be the complexification of one of the four Hurwitz's algebras $\mathbb R,\mathbb C,\mathbb H$ or $\mathbb O$ and denote by  ${\rm Herm}_3(\mathbb A)$ the algebra of Hermitian $3\times 3$ matrices with coefficients in $\mathbb A$: 
\begin{equation*}
 {\rm Herm}_3(\mathbb A)= \bigg\{    
\begin{tiny} \begin{pmatrix}
  r_1 & \overline{x_3} &  \overline{x_2} \\x_3 & r_2 &   \overline{x_1} 
\\ x_2& x_1& r_3
 \end{pmatrix}
 \end{tiny}
 \; \Big| 
 \begin{small}
 \begin{tabular}{c}
 $ x_1,x_2,x_3 \in \mathbb A $\\
$r_1,r_2,r_3 \in \mathbb C$
\end{tabular}
\end{small}
\bigg\} .
\end{equation*}
Then the symmetrized product $M\bullet N=\frac{1}{2}(MN+NM)$ induces on   ${\rm Herm}_3(\mathbb A)$ a structure of rank 3  Jordan algebra.
\end{enumerate}}
\end{ex}

A Jordan algebra of rank 1 is isomorphic to $\mathbb C$ (with the standard multiplicative product). It is a classical result that any rank 2 Jordan algebra is isomorphic to an algebra as in Example \ref{EX:jordan}.(\ref{Ex:rank2}). 
In this paper, we will mainly consider Jordan algebras of rank 3.  
These are the simplest Jordan algebras which have not been yet classified in arbitrary dimension. 
\smallskip

Let $\mathbb J$ be a rank 3 Jordan algebra. The general theory specializes in this case and ensures the existence of a linear form $T:\mathbb J\rightarrow \mathbb C$ (the {\it generic trace}), of a quadratic form $S\in {\rm Sym}^2(\mathbb J^*)$ and of a cubic form 
$N\in {\rm Sym}^3(\mathbb J^*)$  (the {\it generic norm}) such that 
\begin{equation}
\label{E:GenPoly}
x^3-T(x)x^2+S(x)x-N(x)e=0
\end{equation}
for every $x\in\mathbb  J$.  Moreover, 
$x$ is invertible in $\mathbb J$ if and only if $N(x)\neq 0$ and in this case  
$x^{-1}={N(x)}^{-1}{x^\#}$, where $x^\#$ stands for the {\it adjoint} of $x$ defined by 
$x^\#=x^2-T(x)x+S(x)e$. The adjoint  satisfies the identity:
$$
\big(x^\#\big)^\#=N(x)x.
$$
\begin{ex}\rm{
\begin{enumerate}
\item   
The algebra $M_3(\mathbb C)$ of $3\times 3$ matrices with complex entries is associative. 
Then  $M_3(\mathbb C)^+$ is a rank 3 Jordan algebra. If $M\in M_3(\mathbb C)$, the generic trace of $M$ is the usual trace,  the norm is the determinant of $M$ and the adjoint is the classical one, that is the transpose of the cofactor matrix of $M$.\smallskip 
\item  Let $\mathbb C\oplus W$ be a rank 2 Jordan algebra as defined in Example \ref{EX:jordan}.(\ref{Ex:rank2}). For $x=(\lambda,w)\in
\mathbb C\oplus W$, one has a trace  $T(x)=3\lambda$  and a quadric norm $ N(x)=\lambda^2+q(w)$ such that $x^2-T(x)x+N(x)e=0$. Then one defines the adjoint by $x^\#=(\lambda,-w)$. In the rank 2 case, one has $(x^\#)^\#=x$.  
\smallskip 
 \item  Let $\mathbb A$ be as in Example \ref{EX:jordan}.(\ref{L:Herm3}).  Since $\mathbb A$ is the complexification of a Hurwitz's algebra, it comes with a  non-degenerate quadratic form $ \Vert\cdot \Vert^2 : \mathbb A\rightarrow \mathbb C$ that is multiplicative.  If $\langle \cdot, \cdot \rangle$ stands for its polarization, then  the generic norm on ${\rm Herm}_3(\mathbb A)$ can be computed obtaining:  
\begin{equation}
\label{E:CubicNorm}
 N\bigg(
  \begin{tiny} \begin{matrix}
  r_1 & \overline{x_3} &  \overline{x_2} \\x_3 & r_2 &   \overline{x_1} 
\\ x_2& x_1& r_3
 \end{matrix}
 \end{tiny}\bigg)=r_1r_2r_3+2\langle x_1x_2,x_3\rangle-r_1 \, \Vert x_1 \Vert^2-r_2\,  \Vert x_2 \Vert^2-r_3 \, \Vert x_3\Vert^2
\end{equation}
 for every $x_1,x_2,x_3\in \mathbb A$,  $r_1,r_2,r_3 \in \mathbb C$.\smallskip 
 \item \label{ex:Spampinato} {\rm Let $J$ be a power-associative algebra. Also in this case one can define the notions of rank, adjoint $x^\#$,
norm $N(x)$ and trace and the theory is completely analogous to the previous one. Let $r=\rk(J)\geq 2$.
The adjoint satisfies the identity $(x^\#\big)^\#=N(x)^{r-2}x$  thus its projectivization 
is a birational involution  of bidegree $(r-1,r-1)$ of $\mathbb P(J)$, see also Example \ref{EX:bir}, \ref{ex:rank}.

 The inverse map
$x\mapsto x^{-1}={N(x)}^{-1}{x^{\#}}$  on $J$   naturally induces a birational involution 
$\widetilde \jmath:\p(J\times\mathbb C)\map
\p(J\times \mathbb C)$ of bidegree $(r,r)$, defined by $\widetilde \jmath([x,r])=[rx^{\#},N(x)]$.
Such  maps were classically investigated by  N. Spampinato and C. Carbonaro Marletta,
see \cite{Spampinato,Carbonaro1,Carbonaro2}, producing examples of interesting Cremona involutions
in higher dimensional projective spaces. It is easy to see that letting $\widetilde J=J\times\mathbb C$, then for $(x,r)\in\widetilde J$
one has $(x,r)^{\#}=(rx^{\#},N(x))$ so that the map $\widetilde \jmath$  is the adjoint map
of the algebra $\widetilde J$. A Cremona transformation 
 of bidegree  $(r,r)$ 
will be called {\it of Spampinato type} if it is linearly equivalent to the adjoint of a 
direct product  $  J\times\mathbb C$ where $J$ is a power-associative algebra of rank $r$.

The previous  construction will be used in \cite{PR3} to produce
  some interesting Cremona involutions and 
to describe differently some  known examples. In \S \ref{S:fromJ} and in \S \ref{S:XtoC}  maps of this type will naturally
appear in relation to tangential projections of twisted cubics over rank three Jordan algebras, respectively 
extremal varieties $3$--covered by twisted cubics.}
 \end{enumerate}\smallskip
 }
\end{ex}

The set of complex Jordan algebras of dimension $n$ will be denoted by $ \boldsymbol{ {\mathcal J }\!ordan}^n$ while
$
\boldsymbol{ {\mathcal J }\!ordan}^n_3$ will indicate  the subset formed by the elements having rank equal to 3.
Here we will focus on   the description of $
\boldsymbol{ {\mathcal J }\!ordan}^n_3$ up to a certain equivalence relation that we now introduce.

\subsubsection{\bf Isotopy}  
Let  $\mathbb J$  be a Jordan algebra.  By definition, the {\it quadratic operator}  associated to an element $x\in   \mathbb J$ is the endomorphism $U_x=2\, L_x\circ L_x-L_{x^2}$ of $\mathbb J$ 
where $L_x$ stands for the multiplication by $x$ in $\mathbb J$.
If $u\in \mathbb J$ is invertible, one defines the {\it $u$-isotope} $\mathbb J^{(u)}$ of $\mathbb J$ as the algebra structure on $\mathbb J$ induced by  the product $\bullet^{(u)}$  defined by 
$$
x \bullet^{(u)} y= \frac{1}{2} U_{x,y}(u), 
$$
where as usual $U_{x,y}=   
U_{x+y}-U_x-U_y$ is  the linearization of the {\it quadratic representation}
$P:V\to\End(V),\, x\mapsto P(x)=U_x$ of $\mathbb J$
(the name is justified by
 the fact that $P$ is a homogenous polynomial map of degree 2).
Then $u^{-1}$ is a unity for  the new product $\bullet^{(u)}$
 and moreover $\mathbb J^{(u)}$
 is a Jordan algebra,  
the {\it $u$-isotope} of $\mathbb J$.  Let us recall that $x\in\mathbb J$
is invertible if and only if $U_x$ is invertible; moreover  $x^{-1}=U_x^{-1}(x)$
and $U_{x^{-1}}=U_x^{-1}$ in this case.
 \smallskip

Two Jordan algebras $\mathbb J$ and ${\mathbb J}'$ are called {\it isotopic} if $\mathbb J'$ is isomorphic to an isotope $\mathbb J^{(u)}$ of $\mathbb J$.  One immediately proves that the rank is invariant by isotopy.
 The norm  $N^{(u)}(x)$ and the adjoint $x^{\#(u)}$ of an element $x\in\mathbb J^{(u)}$ 
are expressed in terms of  the norm $N(x)$ and the adjoint $x^\#$ in the algebra $\mathbb J$ in the following way, see \cite[II.7.4]{McCrimmon-book}:
 \begin{equation}
 \label{E:NormAdjointUisotop}
 N^{(u)}(x)=N(u)N(x) \qquad \mbox{and} \qquad x^{\#(u)}= N(u)^{-1} U_{u^\#}(x^\#). \medskip 
 \end{equation}
 
 If $\mathbb J$ is a Jordan algebra, then we shall denote by  $[\mathbb J]$  its isotopy class.  \smallskip
 
  Of course, {\it isotopy} defines a equivalence relation on $
\boldsymbol{ {\mathcal J }\!ordan}^n$ and  hence on  $
\boldsymbol{ {\mathcal J }\!ordan}^n_3$ since the rank is isotopy-invariant. In this paper, we are interested in the description of the quotient space $
\boldsymbol{ {\mathcal J }\!ordan}^n_3/_{isotopy}$.\smallskip 

 The concept of `Jordan pair' is  a useful  notion to deal with Jordan algebras up to isotopy. We   introduce it in the next section. This notion will be used later in section \ref{S:XtoJ}. 

\subsubsection{\bf Jordan pairs}  By definition  (see \cite{loos}), a {\it Jordan pair}  is a pair  $V=(V^+,V^-)$ of complex vector spaces together with quadratic maps 
(for $\sigma=\pm $)
$$Q^\sigma: V^\sigma\rightarrow {\rm Hom}(V^{-\sigma},V^\sigma)\, , 
$$ 
 satisfying  the following relations  for every $(x,y)\in V^\sigma \times V^{-\sigma}$: 
$$ 
D^\sigma_{x,y}Q^\sigma_x=   Q^\sigma_x D^{-\sigma}_{y,x} \, , \qquad 
D^\sigma_{Q^\sigma_x(y),y}=    D^{\sigma}_{x, Q^{-\sigma}_x(y)}   
  \qquad 
\mbox{and} \qquad 
Q^{\sigma}_{Q^\sigma_x(y)} =   Q^{\sigma}_x Q^{-\sigma}_y Q^{\sigma}_x, 
$$
 where $D^{\sigma}_{x,y}\in {\rm End}(V^{\sigma} )$ is  
defined by $
D^{\sigma}_{x,y}(z)=Q^\sigma_{x+z}(y)-Q^\sigma_{x}(y)-Q^\sigma_{z}(y)$ for every $z\in V^\sigma$.

\begin{ex} {\rm 
\begin{enumerate}
\item 
Let $\mathbb J$ be a Jordan algebra. Then $V=(\mathbb J,\mathbb J)$  with quadratic operators $Q^\pm_x=U_x$ for  every $x\in \mathbb J$,   is a Jordan pair. By definition, it is the {\it Jordan pair associated to $\mathbb J$};
\medskip 
\item  Given  integers $p,q>0$, the pair  $V=(M_{p,q}(\mathbb C),M_{q,p}(\mathbb C))$  together with the quadratic operators defined by $Q^\sigma_x(y)=x \cdot  y\cdot x$ (where $\cdot$ designates the usual matrix product) is a Jordan pair. 
 \end{enumerate}
}
\end{ex}

By definition, a {\it morphism} between two Jordan pairs $(V^+,V^-)$ and $(\overline V^+,\overline V^-)$ with respective associated quadratic operators $Q^\sigma $ and ${\overline Q}^{\, \sigma} $, is a pair $h=(h^+,h^-)$ of linear maps 
$h^\sigma: V^\sigma \rightarrow \overline V^\sigma$ such that, for all  $\sigma=\pm$ and every $(x,y)\in V^\sigma\times V^{-\sigma}$, one has 
$$
h^\sigma\big(Q^\sigma_x(y)\big)={\overline Q}^{\, \sigma}_{h^\sigma(x)}\big(h^{-\sigma}(y)\big).
$$ 
{\it Isomorphisms} and {\it automorphisms} of Jordan pairs are defined in the obvious way.

An element $u\in V^{-\sigma}$  is said to  be {\it invertible} if  $Q^{-\sigma}_u$ is invertible (as a linear map from $V^{\sigma}$ into $V^{-\sigma}$).  In this case,  one 
 verifies that  the product 
 $$
x\bullet x':=\frac{1}{2}Q^\sigma_{x,x'}(u) =\frac{1}{2} \Big(Q^\sigma_{x+x'}(u)-Q^\sigma_{x}(u)-Q^\sigma_{x'}(u)
\Big)
$$
induces on $V^{\sigma}$ a Jordan algebra structure with unit $(Q^{-\sigma}_u)^{-1}(u)\in V^{\sigma}$. This Jordan algebra is noted by $V_u^{\sigma}$. 
Then it can be proved that $V$ is isomorphic  to the Jordan pair associated to  $V_u^\sigma$. This gives an equivalence between Jordan algebras up to isotopies and  Jordan pairs admitting  invertible elements up to isomorphisms, see \cite{loos}.

\section{Equivalences}

In this section, we establish some equivalences between the three mathematical worlds introduced above.

\subsection{Starting from the $J$-world}\label{S:fromJ}

Let $\mathbb J$ be a Jordan algebra of dimension $n$ and of rank 3. Following Freudenthal in \cite{Freudenthal}, one defines the {\it twisted cubic} over $\mathbb J$ as the Zariski closure $X_\mathbb J$ of the image of the affine embedding
\begin{align*}
\mu_{\mathbb J}: \mathbb J & \longrightarrow     \mathbb P\big(\mathbb C\oplus    \mathbb J \oplus \mathbb J \oplus \mathbb C\big)\\
x  & \longmapsto    \big[ 1: x : x^\#:N(x) \big].
\end{align*}
It is  known  that $X_{\mathbb J}\subset\p^{2n+1}$ belongs to the class $\boldsymbol{X}^n(3,3)$, see for example \cite[\S 4.3]{PR}. We shall provide below other  proofs of this fact,
see Proposition \ref{P:Xn33f}. \smallskip   

Let $\mathbb J^{(u)}$ be the $u$-isotope of $\mathbb J$ relatively to an invertible element   $u\in \mathbb J$.  Let $\ell_u$ be the linear automorphism of  $ \mathbb P(\mathbb C\oplus    \mathbb J \oplus \mathbb J \oplus \mathbb C)=\mathbb P^{2n+1}$  defined by 
$$
 \ell_u\big(  [s: X:Y:t  ]\big)=\big[ s: X: N(u)^{-1}U_{u^\#}(Y): N(u)t\big].
 $$  
 It follows from (\ref{E:NormAdjointUisotop}) that, as affine maps from $\mathbb J=\mathbb J^{(u)}$ to $\mathbb P^{2n+1}$, one has 
$\mu_{\mathbb J^{(u)}}=\ell_u\circ \mu_{\mathbb J}$.  Hence the projective varieties $X_{\mathbb J}$ and 
$X_{\mathbb J^{(u)}}$ are projectively equivalent.  Therefore the association $\mathbb J\rightarrow X_{\mathbb J}$ factorizes and induces a  well defined application 
\begin{align*}
{\boldsymbol {\mathcal J \!ordan}_{3}^n}_{\! \!\big/  isotopy}  & \longrightarrow  \;\,   \boldsymbol { X}^n(3,3)_{\!\big/ \!\!\!
{\tiny{\begin{tabular}{l}
$projective$\\
$equivalence$
\end{tabular}}
}} \\
\big[   \mathbb J  \big] \quad & \longmapsto \quad \big[  X_{\mathbb J}  \big]. 
\end{align*}
\smallskip

Similarly, since $x^{\#(u)}=N(u)^{-1}U_{u^\#}(x^\#)$, the linear equivalence class of the birational map ${ \#_{\mathbb J}}: [x]\dashrightarrow [x^\#]$ of $\mathbb P^{n-1}$  does not depend on the isotopy class of $\mathbb J$. Hence we also  get a well-defined map
 \begin{align*}
{\boldsymbol {\mathcal J \!ordan}_{3}^n}_{\! \!\big/  isotopy}  & \longrightarrow    \;\, 
{\bf Bir}_{2,2}(\mathbb P^{n-1})_{\!\big/ \!\!\!
{\tiny{\begin{tabular}{l}
$linear$\\
$equivalence$
\end{tabular}}
}}
\\
\big[   \mathbb J  \big] \quad & \longmapsto \quad \big[  { \#_{\mathbb J}}  \big].  %\medskip 
\end{align*}
 \smallskip

 \begin{rem}{\rm 
 The tools used above  to construct the `twisted cubics over Jordan algebras of rank three'  
are the adjoint $x^\#$ and the norm $N(x)$. These notions has been introduced 
for every unital power-associative algebra so that one can
ask  if it were  possible to define  `twisted cubics over  commutative power-associative
 algebras  of rank three with unity.'  
Since 
   a commutative power-associative algebra of rank  three with unity 
is necessarily a Jordan algebra of the same rank, 
according to  \cite[Corollary 13]{ElduqueOkubo}, this generalization
 would not produce new examples.

In the same vein, one could  define  a map associating 
to a rank three power-associative
algebra with unity the  quadro-quadro Cremona transformation given by the linear
equivalence class of its adjoint. As we shall in Theorem \ref{P:jordanXf} below this
generalization is useless since the restriction of this map to Jordan
algebras of rank three will be surjective. Moreover, by applying Theorem \ref{P:jordanXf}
to the adjoint of a commutative power-associative algebra of rank  three with unity one could
deduce a new proof of \cite[Corollary 13]{ElduqueOkubo} mentioned above.}
  \end{rem}

  One verifies easily that $\infty_\mathbb J=[0:{\boldsymbol{ 0}}: {\boldsymbol{ 0}}:1]$ is a smooth point of $X_{\mathbb J}$ and that the homogeneization of
 $\mu_{\mathbb J}$ is the inverse of the birational map $\pi_{\infty_\mathbb J}:  X_{\mathbb J}  \map \mathbb P^n$ given by the restriction to $X $ of the projection from $T_{\infty_\mathbb J }X_{\mathbb J}$, see for example \cite[\S 4]{PR}.
 It is immediate to verify that $0_\mathbb J=\mu_{\mathbb J}(\mathbf 0)=[1:\mathbf 0:\mathbf 0:0]\in X_{\mathbb J}$ is also a  smooth point and that 
 $T_{0_\mathbb J}X_{\mathbb J}$ is the closure of the locus of  points of the form  $ [1:{x} :\mathbf 0:0]$ with   $x\in  \mathbb J$.  Thus the birational map
 $\psi:\p(\mathbb J\times\mathbb C)\map\p(\mathbb J\times\mathbb C)$ given by $\psi([x_0:x])=[x_0x^{\#}:N(x)]$
 is a birational involution of type $(3,3)$ of Spampinato type (see Example \ref{EX:jordan}.(\ref{ex:Spampinato})) and it  is clearly the  composition of the homogenization of $\mu_{\mathbb J}$ with the (restriction to $X_\mathbb J$ of the) linear projection 
 $\pi_{0_\mathbb J}$ from 
 $T_{0_\mathbb J}X_{\mathbb J}$,
 that is $\psi=\pi_{\infty_\mathbb J}\circ\pi_{0_\mathbb J}^{-1}$ as rational maps.
 We shall return on this  in \S \ref{S:XtoJ}.

 \subsection{Starting from the $C$-world}
 \label{S:Cremona}
   Let $f\in  {\bf Bir}_{2,2}(\mathbb P^{n-1})$, let  $g\in  {\bf Bir}_{2,2}(\mathbb P^{n-1})$ be  its inverse, let  $F,G: \mathbb C^n\rightarrow \mathbb C^n$ be  associated
  quadratic lifts and let $N,M$ be the associated cubic forms,  see Section \ref{S:S-C}. 

  \subsubsection{\bf From the $C$-world to the $X$-world}
  \label{S:CtoX}
  
   Let us consider  the following affine embedding 
  \begin{align}
  \label{E:mu-f}
\mu_{f}: \mathbb C^n & \longrightarrow      \mathbb P\big(\mathbb C\oplus \mathbb C^n\oplus \mathbb C^n\oplus \mathbb C\big)=\mathbb P^{2n+1}\\
x  & \longmapsto    \big[ 1: x : F(x):N(x) \big]. \nonumber
\end{align}

 The  Zariski-closure $X_f$ of its  image  is a non-degenerate irreducible  $n$-dimensional subvariety of $\mathbb P^{2n+1}$   containing $0_f=\mu_f({\boldsymbol{ 0}})=[1:{\boldsymbol{ 0}}:{\boldsymbol{ 0}}:0]$. \smallskip 
 
 In order to prove that $X_f$ is 3-covered by twisted cubics,  we shall use in different ways the following crucial result whose incarnations
 in the three worlds we defined till now will be  the starting points of the bridges  connecting these apparently different universes.
 \smallskip
  \begin{lemma} 
 \label{L:crucial} Let notation be as above.
 There exists a bilinear form $B_F: \mathbb C^n\times \mathbb C^n\rightarrow \mathbb C$  such that
   $$
 dN_x=B_F(F(x),dx),
 $$
for every $x\in \mathbb C^n$.
 \end{lemma}

 \begin{proof} 
 In coordinates, the relation $G(F(x))=N(x)x$ translates into
\begin{equation}
\label{R*} \qquad 
   g_i\big(f_1(x),\ldots,f_n(x)\big)=x_iN(x)\, ,  \qquad  i=1,\ldots,n.
   \end{equation}
   
Let $I=\langle f_1,\ldots, f_n\rangle \subset \mathbb C[x_1,\ldots,x_n]=S=\oplus_{d\geq 0} S_d$
and let $I=\oplus_{d\geq 0} I_d$.
Let us recall that the biggest homogeneous ideal of $S$ defining the scheme $\mathcal B=V(I)$ is the saturated ideal  
$$I^{\sat}=\oplus_{d\geq 0} I^{\sat}_d= \oplus_{d\geq 0} H^0\big(\mathcal I_{\mathcal B}(d)\big).$$

It follows from (\ref{R*}) that 
\begin{equation}
\label{R-l}
\qquad 
x_iN(x)\in (I_2)^2\subseteq I_4  \qquad \mbox{for every } i=1,\ldots,n.
 \end{equation}
   By derivation of   (\ref{R*})  with respect to $x_j$ for $j$ distinct from $i$,     we deduce  that $ x_i({\partial N}/{\partial x_j}) \in I_3 $ yielding 
\begin{equation}
\label{R**}
   \qquad  \qquad 
  x_i^2\frac{\partial N}{\partial x_j}\in I_4
  \qquad \mbox{for every } i,j=1,\ldots,n, \; i\neq j.
   \end{equation}

By derivation of  (\ref{R*}) with respect to $x_i$ we obtain  
$ N(x)+x_i({\partial N}/{\partial x_i})\in I_3$ for every $i=1,\ldots,n$. Multiplying by $x_i$  and using (\ref{R-l}) we deduce 
 $x_i^2({\partial N}/{\partial x_i})\in I_4 $ for every $i$.   Combined with (\ref{R**}), this shows that  
$ x_i^2({\partial N}/{\partial x_j})\in I_4$ for every $i,j=1,\ldots,n$. Then by definition 
\begin{equation}
\label{R-L}
  \frac{\partial N}{\partial x_i} \in I_2^{\sat}
   \end{equation}
 for  every $i=1,\ldots,n$.  Since $
 I_2^{\sat}=H^0(\mathcal I_{\mathcal B}(2))=\mbox{span}\{  f_1,\ldots,f_n \}$,
the last equality being an immediate consequence of the birationality of $f$,  
  there exist constants  $b_{ij}\in \mathbb C$ such that 
 ${\partial N}/{\partial x_i}= \sum_{j=1}^n b_{ij} f_j 
 $ for every $i$. Then letting  $ B_F(x,y)=\sum_{i,j=1}^n b_{ij} x_iy_j$, we have  $  dN_x=B_F(F(x ),dx)
 $ for every~$x\in\mathbb C^n$. 
 \end{proof}
 \smallskip

  We now  provide below two different proofs that 
$X_f\in \boldsymbol{X}^n(3,3)$.  Both are interesting in our opinion: the first one is  more elementary but computational;  the second one
is algebro-geometric and the computations are hidden in some elementary well known facts.
\medskip
   
\begin{prop}\label{P:Xn33f} Let notation be as above. The
variety  $X_f$ belongs to the class $X_f\in \boldsymbol{X}^n(3,3)$:
$X_f$ in non--degenerate in $\p^{2n+1}$, 
is 3--RC by twisted cubics  and is different from  a rational normal scroll.
\end{prop}
\begin{proof}[First proof]
For  $a,b\in \mathbb C^n$ with  $M(b)\neq 0$, $\gamma_{a,b}: t\longmapsto    \frac{G(a+tb)}{M(a+tb)}$
 is a well-defined rational map and it follows from  \eqref{E:bir22} and \eqref{E:bir22-2}  that  for $t$ generic, one has
   \begin{align*}
\mu_f\big(\gamma_{a,b}(t)\big)=
 \, \bigg[ 1 :   \frac{G(a+t\,b)}{M(a+t\,b)} :  \frac{F(G(a+t\,b))}{M(a+t\,b)^2} :    \frac{N( G(a+t\,b))}{M(a+t\, b)^3}        \bigg]   
= \, \big[ {M(a+t\,b)}  :   G(a+t\,b) : a+t\,b :    1    \big] .
\end{align*}
Thus $ \mu_f\circ \gamma_{a,b}(\mathbb P^1)$ is a twisted cubic curve passing through $0_f=\mu_f\big(\gamma_{a,b}(\infty)\big)$ and moreover $X_f$ is 2-covered by twisted cubics passing through $0_f$: for  $(p,p')\in (X_f)^2$ general, there exists a twisted cubic  included in $X_f$ and containing   the points $p,p'$ and $0_f$.

Now let $x_\star\in \mathbb C^n$  be such that $N(x_\star)\neq 0$, let   $\tau_{x_\star}$ be the translation by $x_\star$ in $\mathbb C^n$ 
 and consider the linear automorphism of $\mathbb P^{2n+1}$ defined by 
$$
\ell_{x_\star}(\omega)=\big[s: x+s\, x_\star : y+dF_{x_\star}(x)+sF(x_\star):   t+ B_F(y,x_\star)+dN_{x_\star}(x)+sN(x_\star)     \big]
$$
for $\omega=[s:x:y:t    ]\in \mathbb P(\mathbb C\oplus \mathbb C^n\oplus \mathbb C^n\oplus \mathbb C)=\mathbb P^{2n+1}$, where $B_F$ stands for the bilinear form given by Lemma \ref{L:crucial}. One verifies immediately that
$$
\ell_{x_\star} \circ  \mu_f=\mu_f\circ \tau_{x_\star}.
$$

 This shows that the pair $(X_f,\mu_f(x_\star))$ is projectively equivalent to $(X_f,0_f)$ hence $X_f$ is also 2-covered by twisted cubics passing through $\mu_f(x_\star)$. Since this holds for any $x_\star \in \mathbb C^n$ such that $N(x_\star)\neq 0$, this implies that $X_f=X^n(3,3)$. 
The variety $X_f$ is not a rational normal scroll since the linear system of quadrics defining the so called second fundamental
form at a general point has no fixed component since it is naturally identified with the linear system defining $f$, see    \cite[\S 5]{PR} for definitions and details. \end{proof}

\begin{proof}[Second proof of Proposition \ref{P:Xn33f}] 
Let notation be as above. Consider $\mathbb C^n$ as the hyperplane $\p^n\setminus V(x_0)$    so
that $[x_0:x_1\ldots:x_n]$ are projective coordinates on $\p^n$ and $\mu_f:\p^n\map X_f$ is a rational map defined on $\mathbb C^n$. Since $f$ is not fake we can suppose $n\geq 3$. Consider
three general points $p=(p_1, p_2, p_3)\in\mathbb (\mathbb P^n)^3$ and let $\Pi_p\subset\p^n$ be their linear span. We claim that the line $L_p=\Pi_p\cap V(x_0)$
determines a plane $\Pi'_p$ cutting the base locus scheme of $f$ in length three subscheme $\mathcal P$ spanning $\Pi_p'$.
Indeed $D_p=f(L_p)\subset\p^{n-1}$ is a conic cutting the base locus scheme of $g$ in length  three subscheme $\mathcal P'$ spanning a plane
$\overline \Pi_p$ because $g(D_p)=L_p$. Then taking  $\Pi'_p=g(\overline\Pi_p)$ the claim is proved. The length six scheme $\{p_1,p_2,p_3,\mathcal P\}$ spans
the 3-dimensional space $\langle\Pi_p,\Pi'_p\rangle$ so that it determines a unique twisted cubic $C_p\subset\p^{n-1}$ containing it. 
By Lemma \ref{L:crucial} the birational map $\mu_f:\p^n\map X_f$ is given by a linear system of cubic hypersurfaces having points of multiplicity at least 2 along its
base locus scheme $V(x_0,N(x))\subset\p^n$. Then $\mu_f(C_p)\subset X_f$
is a twisted cubic passing through the three general points $\mu_f(p_i)$, $i=1, 2, 3$.   This shows that $X_f=X^n(3,3)$ while to verify that
$X_f$ is not a rational normal scroll one can argue as in  the end of the previous proof.    \end{proof}

One immediately verifies
 that the projective equivalence class of $X_f$ does not depend on $f$ 
but only on its linear equivalence class. Hence there exists a well-defined map 
 \begin{align*}
 {\bf Bir}_{2,2}(\mathbb P^{n-1})_{\!\big/ \!\!\!
{\tiny{\begin{tabular}{l}
$linear$\\
$equivalence$
\end{tabular}}
}}
& \longrightarrow   \;\,   \boldsymbol { X}^n(3,3)_{\!\big/ \!\!\!
{\tiny{\begin{tabular}{l}
$projective$\\
$equivalence$
\end{tabular}}
}} \\
\big[   f  \big] \quad & \longmapsto \quad \big[  X_{f}  \big]. 
\end{align*}

 \subsubsection{\bf From the $C$-world to the $J$-world}
  \label{S:CtoJ}
Let us now explain how to associate in a direct and algebraic way a  rank 3 Jordan algebra to $f\in  {\bf Bir}_{2,2}(\mathbb P^{n-1})$. 
Assume that $\mathbb P^{n-1}=\mathbb P(V)$ for a $n$-dimensional vector space $V$.  
On the open set defined by $N(x)\neq 0$ we define 
$$j_f(x)=\frac{ F(x)}{N(x)}.$$
Then $j_f:V\map V$ is a birational map which is 
 homogeneous of degree $-1$. Following  \cite{McCrimmon-axioms},  we say that the map 
$j_f:V\map V$ is an {\it inversion} and that the elements $x\in V$ with $N(x)\neq 0$ are {\it invertible}.  For $x\in V$ invertible,
one sets
$$
P_f(x)=-d ({j_f})_{x}^{-1}.
$$
We  defined in this way a rational map  $P_f: V\dashrightarrow {\rm End}(V)$ 
which is homogeneous of degree 2.  Similarly one defines
 $j_g: V\dashrightarrow V$ and $P_g: V\dashrightarrow {\rm End}(V)$.

\begin{thm}\label{P:jordanXf} Let notation be as above. For  every linear equivalence class $[f]$,  $f\in  {\bf Bir}_{2,2}(\mathbb P^{n-1})$,
there exists a  Jordan algebra $\mathbb J_f$ of rank 3 such that $[\#_{\mathbb J_f}]=[f]$.

In particular every quadro-quadric Cremona transformation is linearly equivalent to an involution
which is the adjoint of a rank 3 Jordan algebra.
\end{thm}
\begin{proof}  
 Replacing  $F$ by $P_f(e)\circ F$ if necessary, 
we can assume that there exists an invertible element $e\in V$ such that $ P_f(e)={\rm Id}_V$. Euler's Formula
and the homogeneity of $j_f$ imply  $P_f(x)(j_f(x))=x$ for every invertible $x$ 
so that, without loss of generality,  we can also assume $j_f(e)=e$. 
 Similarly, one sets $j_g(y)=M(y)^{-1} G(y)$ for $y$ such that $M(y)\neq 0$.
 Taking the exterior derivative of the relation $j_g\circ j_f(x)=x$, 
 we deduce $P_f(x)=-d ({j_g})_{j_f(x)}$ for any invertible $x$. The differentials  $dG_y$ and $dM_y$  are  homogeneous of degree 1, respectively of degree 2, in $y$. Hence the substitution   $y=j_f(x)=N(x)^{-1}F(x)$ in
 $$d(j_g)_y=M(y)^{-1} dG_y- M(y)^{-2} G(y)dM_y$$ yields
\begin{align*}
 P_f(x)=
     &\, -
dG_{F(x)}   
 +    {x }{N(x)}^{-1} dM_{F(x)}    && \mbox{(by } (\ref{E:bir22})  \mbox{ and }  (\ref{E:bir22-2}))
   \\
       = &\, -
dG_{F(x)}   
 + {x }{N(x)}^{-1}  B_G\big( G(F(x)),dx\big)     && \mbox{(by Lemma  }  \ref{L:crucial})
   \\
    = &\, -
dG_{F(x)}   
 +    x\, B_G\big( x,dx).   && 
    \end{align*}
Thus the rational map  $P_f:V\dashrightarrow   \End(V)$ extends to a polynomial quadratic affine morphism $P_f:V\rightarrow \End(V)$. Therefore
 $$
P_f(x,y)=   P_f(x+y)-P_f(x)-P_f(y) \in {\rm End}(V) $$
 is bilinear in $x$ and $y$ and the results of \cite{McCrimmon-axioms} (in particular Theorem 4.4 and Remark 4.5 therein)
  assure  that the product $\bullet_f$ on $V$ defined by 
 $$
 x\bullet_f y=\frac{1}{2}P_f(x,y)(e)
 $$
 satisfies the Jordan identity (\ref{E:jordan}), 
admits $e$ as a unital element and induces on  
$V$ a structure of Jordan algebra noted by $\mathbb J_f$.  
For a $x\in V$ invertible element, the inverse for this product is given 
by $x^{-1}=j_f(x)$ hence the adjoint of $x$ is  $x^\#=F(x)$, 
 yielding  $\rk(\mathbb J_f)=3$, see Example \ref{EX:bir}.(\ref{ex:rank}).
\end{proof}

  It can be verified that the isotopy equivalence class of $\mathbb J_f$ depends only on the linear equivalence class of $f$ hence one obtains a well-defined map
  \begin{align*}
 {{\bf Bir}_{2,2}(\mathbb P
 ^{n-1})}_{\!\big/ \!\!\!
{\tiny{\begin{tabular}{l}
$linear$\\
$equivalence$
\end{tabular}}
}}
& \longrightarrow   \;\,   
{\boldsymbol {\mathcal J \!ordan}_{3}^n}_{\! \!\big/  isotopy} 
 \\
\big[   f  \big] \quad & \longmapsto \quad \big[  {\mathbb J}_{f}  \big]. 
\end{align*}

{\begin{rem}
\label{R:CtoJPair}
{\rm  
In the previous proof we choose  a point $e$ 
such that $P_f(e)={\rm Id}_V$, which is  not natural from an intrinsic point of view.
More generally one can consider the source space 
and the target space of $f$ as distinct $n$-dimensional projective spaces
associated to  two  vector spaces $V_F$ and $V_G$ of dimension $n$.
We can consider $f$ as a birational map $f:\mathbb P(V_F)\dashrightarrow \mathbb P(V_G)$ 
with inverse  $g:\mathbb P(V_G)\dashrightarrow \mathbb P(V_F)$. 
Reasoning as in  the proof of Theorem \ref{P:jordanXf},  one  proves
 that $-d(j_f)_x^{-1}$ (resp. $-d(j_g)_y^{-1}$) depends
 quadratically on $x\in V_F$ (resp. on $y\in V_G$),
defining a quadratic map $P_f:V_F\rightarrow {\rm End}(V_G,V_F)$ 
(resp. $P_g:V_G\rightarrow {\rm End}(V_F,V_G)$). 
 Then $(V_F,V_G)$ together with the pair of quadratic operators $(P_f,P_g)$ is a Jordan pair.}
\end{rem}}

 \subsection{Starting from the $X$-world}
 \label{S:fromXworld}
 
In this section, we describe how to associate  to  a $X\in \boldsymbol{X}^n(3,3)$ an equivalence class of quadro-quadric Cremona transformations of $\mathbb P^{n-1}$ and also  how to produce directly (and not through the previous construction!) a  rank 3 Jordan algebra $\mathbb J$ of dimension $n$, defined modulo  isotopy, such that $X$ is projectively equivalent to $X_{\mathbb J}$.
\medskip

Let $X\subset \mathbb P^{2n+1} $ be an element of $ \boldsymbol{X}^n(3,3)$. Let $x\in X_{\rm reg}$ be a general point 
such that $X$ is 2-RC by a family $\Sigma_x\subset {\rm Hilb}^{3t+1}(X,x)$
 of twisted cubics included in $X$ and passing through $x$.
Denote  by  $\pi_x: X\dashrightarrow \mathbb P^n$ the restriction to $X$ of the tangential projection with center $T_x X \subset \mathbb P^{2n+1}$. It is known  that $\pi_x$ is birational, see \cite{PT,PR}. Let $\beta_x: \widetilde X\rightarrow X$ be the blow-up of $X$ at $x$ and  let $E_x=\beta_x^{-1}(x)$  be the associated  exceptional divisor. Let $\varphi_{X,x}$ be the restriction to $E_x$ of the lift of $\pi_x$ to $\widetilde X$: 
 \begin{equation}
 \label{E:varpiXx}
 \varphi_{X,x}= (\pi_x\circ \beta_x)\big|_{E_x} : E_x\dashrightarrow \mathbb P^n.
 \end{equation}
  
In  \cite[\S 5]{PR}, we proved that: 
\begin{enumerate}
\item[{\rm (a)}]  $\varphi_{X,x}$ is  birational onto its image that is a hyperplane $H_x\subset \mathbb P^n$; 
 \item[{\rm (b)}]  $\varphi_{X,x}$ is induced by the second fundamental form $| II_{X,x}| \subset |\mathcal O_{E_x}(2)|$;
\item[{\rm (c)}] as a scheme, the base locus scheme of $\varphi_{X,x}$ coincides  with the Hilbert scheme of lines passing through
$x$ and contained in $X$; 
  \item[{\rm (d)}] $(\varphi_{X,x})^{-1}:H_x\dashrightarrow E_x$ is also induced by a linear system of hyperquadrics in $H_x$. 
 \item[{\rm (e)}] $\varphi_{X,x}$  is a fake quadro-quadric transformation if an only if $X\subset\p^{2n+1}$ is a rational normal scroll. 
  \smallskip \end{enumerate}
  
From (c) and (d) one could deduce another proof of Lemma \ref{L:crucial}, see {\it loc. cit.}, or equivalently one can
say that (d)
 is the incarnation in the $X$-world of the result proved 
in Lemma \ref{L:crucial}, see \cite[Theorem 5.2]{PR} for details.

\begin{rem}
 {\rm   The map $v\mapsto \tilde{v}$ considered in the proof of Theorem \ref{T:XJ} below 
is an alternative geometrical definition of (a quadratic lift of) $\varphi_{X,x}$ 
which is more intrinsic than the preceding one since it does not depend on
 the embedding  of $X$ in  the projective space $\mathbb P^{2n+1}$.}
  \end{rem}

  \subsubsection{\bf From the $X$-world to the $C$-world}
  \label{S:XtoC}

 From now on we shall assume that $X\in {\boldsymbol X}^n(3,3)$ so that   $X$ is not a rational normal scroll.
The results  listed above imply that, after identifying $E_x$ and $H_x$ with $\mathbb P^{n-1}$,  the map  $\varphi_{X,x}$ is  a Cremona transformation of bidegree $(2,2)$ of $\mathbb P^{n-1}$. Moreover, in \cite[Theorem 5.2]{PR} it is proved that $X$ is projectively equivalent to the variety $X_{\varphi_{X,x}}$ associated to $\varphi_{X,x}$ via the construction in section \ref{S:Cremona}.  We leave to the  reader to verify that  the linear equivalence class of $ \varphi_{X,x}$ does not depend on $x$ but only  on the projective equivalence class of $X$.
 Therefore we have  a well-defined application 
  \begin{align*}
   \boldsymbol { X}^n(3,3)_{\!\big/ \!\!\!
{\tiny{\begin{tabular}{l}
$projective$\\
$equivalence$
\end{tabular}}
}}
& \longrightarrow  \;\, 
 {\bf Bir}_{2,2}(\mathbb P^{n-1})_{\!\big/ \!\!\!
{\tiny{\begin{tabular}{l}
$linear$\\
$equivalence$
\end{tabular}}
}}
\\
\big[   X  \big] \quad & \longmapsto \quad \big[ \varphi_{X,x}   \big]. 
\end{align*}
The results of the previous sections show that this map is a bijection.
 
  \subsubsection{\bf From the $X$-world to the $J$-world}
  \label{S:XtoJ}

  The results of  \cite[\S 5]{PR} recalled above and Theorem \ref{P:jordanXf}
 immediately imply that
any  $X\in \boldsymbol{X}^n(3,3)$ is of {\it Jordan type}, 
that is  there exists a rank three Jordan algebra $\mathbb J$ 
such that $X$ is projectively equivalent to $X_{\mathbb J}$. 

There is also  a  direct way  to recover geometrically
 the underlying  structure of Jordan algebra from $X$.  
Since there is no real difficulty here, we will leave to the interested readers 
to fill up some details of the proof of the next result,
which was conjectured firstly in \cite[\S 5]{PR}.

\begin{thm}\label{T:XJ}  If $X=X^n(3,3)\subset\p^{2n+1}$ is not a rational normal scroll, then there exists a rank three Jordan algebra $\mathbb J_X$ such that
$X$ is projectively equivalent to $X_{\mathbb J_X}$.
\end{thm}

\begin{proof}
Let $x^+,x^-$ denote two general points of $X_{\rm reg}$ 
such that $X$ is 1-RC by the family $\Sigma_{x^+x^-}$ of twisted cubics 
included in $X$ passing through $x^+$ and $x^-$.  
One has $\mathbb P^{2n+1}=T_{x^+}{X}\oplus T_{x^-}{X}$ and for $\sigma=\pm$, let
$\pi^\sigma=\pi_{x^\sigma}$ be the restriction to $X$ of the tangential projection
 with center $T_{x^\sigma} X$ onto the projective tangent space $T_{x^{-\sigma}} X$
 at the other point. This map is defined at $x^{-\sigma}$ 
and by definition  ${x^{-\sigma}}=\pi^\sigma({x}^{-\sigma})$.

 Define  $V^\sigma$ as the complex tangent space $T_{X,x^\sigma}$.  For $v\in V^\sigma$ generic, 
there exists a unique twisted cubic curve $C_v$ included in $ X$,  joining 
 $x^{\sigma}$ to $x^{-\sigma}$ and having   $[v]$ as tangent direction  at ${x}^{\sigma}$.
  More precisely, there exists a unique isomorphism 
$\alpha_v:\mathbb P^1\rightarrow C_v$ such that 
 $ \alpha_v(0:1)=x^\sigma$, $  \alpha_v(1:0)=x^{-\sigma}$  and ${d\alpha_v} (s:1)/{ds}|_{s=0}=v$.  
 The map $$v\mapsto \exp(v):=\alpha_v(1:1)$$ can be  extended to  
the whole $V^\sigma$ since, after some natural identifications, 
 it is nothing but the affine embedding $\mu_{\psi_{\sigma}}$ defined in   (\ref{E:mu-f}), 
where $\psi_\sigma$ is the inverse of the quadro-quadratic birational map $\varphi_{X,x^{-\sigma}}$  associated to $\pi^{-\sigma}$ through formula  (\ref{E:varpiXx}) above. 
We thus defined geometrically an  {\it exponential map} 
 \begin{align*}
 \exp: V^\sigma& \longrightarrow X 
  \end{align*}
  whose image is denoted by $X^\sigma$. Being an affine embedding,   
   its  differential $$d\exp_v: T_{V^\sigma,v}\rightarrow T_{X,\exp(v)}$$ is an isomorphism
   for every $v\in V^\sigma$.  
Using the linear structure of $V^\sigma$, one can (canonically) 
identify $T_{V^\sigma,v}$ with $V^\sigma$ itself 
obtaining a linear isomorphism $\delta_v^\sigma: V^\sigma\rightarrow T_{X,\exp(v)}$. 
 For $v$ general,  $\exp(v)\in X^{-\sigma}$ so 
that there exists a unique $\tilde v\in V^{-\sigma}$ such that 
$\exp(v)=\exp(\tilde v)$ (moreover $\tilde{v}= {d\alpha_v} (1:t)/{dt}|_{t=0}$). 
Thus we can define a linear isomorphism by setting
$$
Q^\sigma_v=    -  \big(   \delta_v^\sigma\big)^{-1}   \circ   \delta_{\tilde v}^{-\sigma} : V^{-\sigma}\longrightarrow V^\sigma. 
$$

The linear map $Q^\sigma_v$ depends quadratically on $v\in V^\sigma$
and this association extends 
to the whole $V^\sigma$ yielding a quadratic polynomial map
\begin{equation*}
\label{E:Qsigma}
 Q^\sigma:  V^\sigma\rightarrow {\rm Hom}( V^{-\sigma}, V^\sigma).
  \end{equation*}
 
  The quadratic maps $Q^\pm$ thus defined   induce a structure of Jordan pair on 
  $(V^+,V^-)$ admitting invertible elements. For $u\in V^{-}$ invertible, the Jordan algebra $V^+_u$ has  rank 3 and the  initially considered variety $X$ is projectively equivalent to  $X_{V^+_u}$. 
 \end{proof}

 \begin{rem} {\rm Let $X\in \boldsymbol{X}^n(3,3)$. 
  The previous proof shows that the Jordan avatar of 
  the  geometrical data formed by $X$ together with two general points $x^+,x^-$ on it  is the Jordan pair $(V^+,V^-)$. Similarly, the geometrical  object corresponding
 precisely  to a rank three Jordan algebra  $\mathbb J$ is  not really $X_{\mathbb J}$ but rather the geometrical data formed by $X_{\mathbb J}$ together with three  points on it.   
    These two remarks lead to the following  heuristic question: what are the Jordan-theoretic counterparts of the data of $X$ alone, or of a pair $(X,x)$ where $x$ is a general point on $X$?
Some of the notions introduced in \cite{bertram} seem to be relevant to study this question.}
  \end{rem}\label{R:HeuristicRemark}  
  
 We shall now briefly outline another geometrical way of recovering the algebra $\mathbb J_X$ naturally associated to $X\in \boldsymbol{X}^n(3,3)$. Let notation be as in \S \ref{S:fromXworld}, let $x_1,x_2\in X$ be general points and let $\psi:\p^n\map\p^n$ be the birational map $\pi_{x_1}\circ\pi_{x_2}^{-1}$, see also end of \S \ref{S:fromJ}. From the results in \cite[\S 5]{PR} recalled above, it is not difficult to see that the birational map $\psi$, which is clearly of bidegree $(3,3)$, is of Spampinato type.
 Indeed, arguing as in the proof of Theorem \ref{P:jordanXf} and based on that analysis, one proves that $\psi$ is linearly equivalent
  to the involution of a rank 4 Jordan algebra $\widetilde J_X$, which is clearly isomorphic to $\mathbb J_X\times \mathbb C$. We leave to the reader the details of the proof of this claim. In conclusion
  from a geometrical point of view the passage from the $X$-world to the $C$ and the $J$ worlds is completely determined
  by the general tangential projections.

   \section{Statement of the main theorem and of a general principle}
   \label{S:XJC-correspondance}
 The constructions of the previous  sections are all represented in the  diagram below, which we will call the `{\it $XJC$-diagram}': 
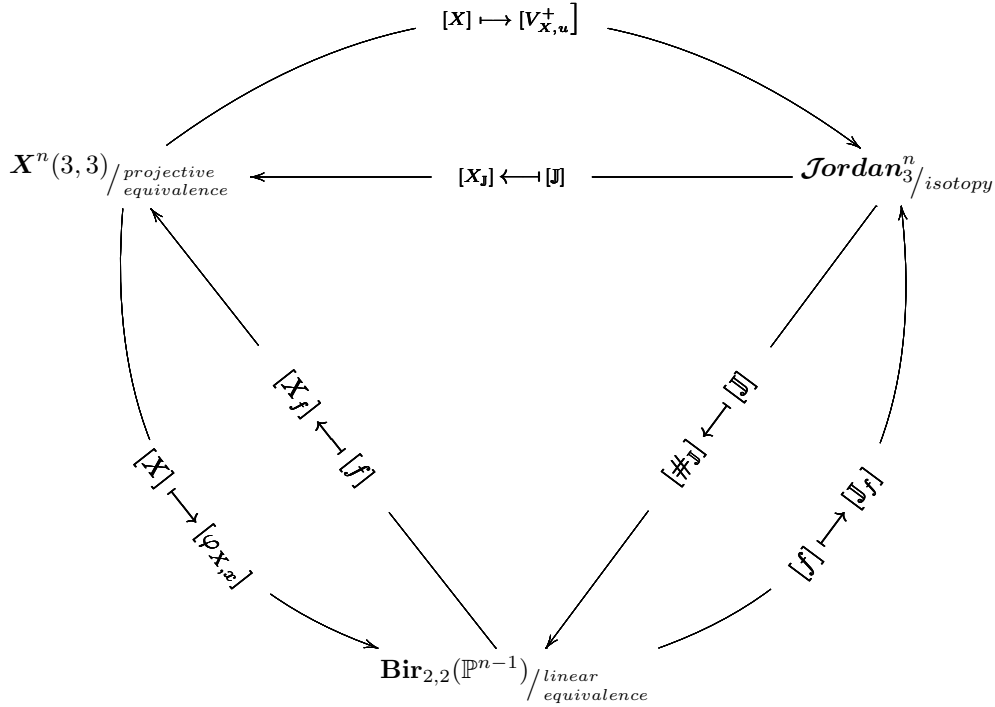
      \begin{equation*}
  \xymatrix@R=1cm@C=0.7cm{    
    \boldsymbol { X}^n(3,3)_{\!\big/ \!\!\!
{\tiny{\begin{tabular}{l}
$projective$\\
$equivalence$
\end{tabular}}
}}
 { \boldsymbol{ \ar@/^5pc/[rrrr]
 | {\quad [X]\;   \longmapsto   \;  [V_{X,u}^+\big]  \quad }}}
    { \boldsymbol{  \ar@/_5pc/[dddddrr]   |*[@]
     {\; {{    [X] \longmapsto [   \varphi_{X,x} ]   \quad }}}}}
     { \boldsymbol{  \ar@{<-}[dddddrr]    | *[@] { \quad {{ [X_f] \stackrel{ \rotatebox{180}{$\longmapsto$}}{    }  [f]  }} \quad }}}
    & & &   &      
   {\boldsymbol {\mathcal J \!ordan}_{3}^n}_{\! \!\big/  isotopy} 
 { \boldsymbol{ \ar@/_0pc/[llll]| {\quad [X_{\mathbb J}]  \,     \stackrel{ \rotatebox{180}{$\longmapsto$}}{    }   \,   [\mathbb J]  \quad }}} 
    \\
 &    &                 &       
 & \\    & & &   & \\
    &    &          &   &  \\ & & &   & \\
     &  &
  {\bf Bir}_{2,2}(\mathbb P^{n-1})_{\!\big/ \!\!\!
{\tiny{\begin{tabular}{l}
$linear$\\
$equivalence$
\end{tabular}}
}}
  { \boldsymbol{  \ar@/_5pc/[uuuuurr] |*[@] {\quad   [f] \longmapsto [\mathbb J_f]  \quad  }
  }}
{ \boldsymbol{  \ar@{<-}[uuuuurr]   |*[@] {\quad   [\#_{\mathbb J}]    \stackrel{ \rotatebox{180}{$\longmapsto$}}{    }          [\mathbb J]  \quad  }   }}
      & & 
 }
\end{equation*}

Then the main results of this paper can be formulated in concise terms by  making reference to this diagram:  \begin{thm} 
\label{T:main}
The above diagram is commutative, all maps  appearing in it  are bijections and 
 the composition of two of these maps, when possible, is  the identity. 
  \end{thm}  
    
 Once the maps in the $XJC$-diagram have been introduced, 
the proof of the preceding theorem reduces to  straightforward verifications left to the reader. 
    
    Theorem \ref{T:main} says in some sense that (up to certain well-understood  equivalence relations) there are correspondences between the objects of  these three distinct worlds. The $X$-world is a world of particular projective algebraic varieties sharing deep geometrical properties and it can be considered as  a `geometrical world'. The $J$-world is a world of particular algebras so that  it is an `algebraic world' while we consider the $C$-world of another nature, which  we will call `cremonian'.  
    \smallskip

  A consequence of the preceding main theorem is the following general principle:

  \begin{XJC}
  Any notion, construction or result concerning one of the $X$,  $J$ or $C$-world  admits a counterpart in the other two worlds.    
  \end{XJC}

 \begin{rem}
    \label{R:categ}
    {\rm 
   The $XJC$--Principle is not a mathematical result in the classical sense  and it  has to be considered as a kind of meta-theorem. 
Theorem \ref{T:main}  and the $XJC$-Principle are  manifestations  
of a deeper phenomena that could be formulated in terms of equivalences of categories. 
We plan to come back to this point of view in the near future and  
we will not deal with this here, although it is very interesting and natural. 
In the sequel we prefer to present some different applications regarding classification results
for particular classes of objects in the different worlds. Other applications will be
obtained in \cite{PR3}.  }      
    \end{rem}

Maybe the better way to realize that such a principle holds consists 
 in presenting some archetypal examples. 
 \subsection{\bf A first occurence of the $XJC$-principle } 
 Assume that $X, \mathbb J$ and $f$ are corresponding objects.

 \begin{prop}\label{P:product} The following assertions are equivalent:
 \begin{enumerate}
 \item[{\rm (i)}] the variety $X$ is  a cartesian product;
 \item[{\rm (ii)}] the algebra $\mathbb J$ is  a direct product;
  \item[{\rm (iii)}] the Cremona map $f$ is  an elementary quadratic transformation, see Example \ref{EX:bir}.
  
\end{enumerate} 
Moreover,  the objects satisfying these properties are respectively: 
the Segre embeddings ${\rm Seg}(\mathbb P^1\times Q^{n-1})$, 
the direct products $\mathbb C\times \mathbb J'$ where $\mathbb J'$
 is a Jordan algebra  of rank 2, the elementary quadratic 
transformations.
\end{prop}
\begin{proof} Clearly (iii) implies (ii) and (i). If (i) holds and if $X=X_1\times X_2\subset\p^{2n+1}$, then we can suppose that through three general points
of $X_1\subset\p^{2n+1}$ there passes a line and that through three general points of $X_2\subset\p^{2n+1}$ there passes a conic.
Then $X_1$ is a line and $X_2$ is a quadric hypersurface in its linear span. Thus $X$ is projectively equivalent to the Segre embedding of
$\p^1\times Q^{n-1}$. The other implications/conclusions easily follows.
\end{proof}   
   
 \subsection{\bf A second occurrence of the $XJC$-principle }
In this subsection, we  relate the smoothness property in   the $X$-world 
to an algebraic 
 one in  the $J$-world and to another one in the $C$-world.  
We introduce these properties.\smallskip 
 
  By definition,  the {\it radical} $R$ of  a Jordan algebra $\mathbb J$, indicated by $\Rad(\mathbb J)$,  is defined as the biggest solvable ideal of $\mathbb J$ (see also Property \ref{P:radical} below for a characterization of the radical when $\mathbb J$ has rank 3).  Then $\mathbb J$ is said to be  {\it semi-simple}  if $\Rad(\mathbb J)=0$. In this case, a classical result 
  of the theory asserts that $\mathbb J$ is isomorphic 
to a  finite direct product $   \mathbb J_1\times  \cdots \times    \mathbb J_m$ 
where the $\mathbb J_k$'s are {\it simple} Jordan algebras, that is 
 Jordan algebras without any non-trivial ideal. 
 \medskip 
 
Following \cite{ST} and \cite{EinShB}, a Cremona transformation  $f:\p^{n-1}\map\p^{n-1}$ is called {\it semi--special} if the base locus scheme of
$f$ is smooth. A Cremona transformation is said to be {\it special} if the base locus scheme is smooth and irreducible.
Thus special Cremona transformations $f:\p^{n-1}\map\p^{n-1}$ can be solved, as rational maps, by a single blow-up along an irreducible smooth variety while
semi--special Cremona transformations can be solved by blowing--up smooth irreducible subvarieties of $\p^{n-1}$, that is there are no ``infinitely near
base points".  In conclusion the 
 semi--special Cremona transformations are the simplest objects
from the point of view of Hironaka's resolutions of rational maps.
 \medskip

 Assume that $X, \mathbb J$ and $f$ are corresponding objects.
 
 \begin{thm}\label{T:simple}  The following assertions are equivalent:
 \begin{enumerate}
 \item[{\rm (i)}] the variety $X$ is smooth;
 \item[{\rm (ii)}] the algebra $\mathbb J$ is semi-simple;
 \item[{\rm (iii)}] the Cremona transformation $f$ is semi-special. 
\end{enumerate}
Moreover, the classification of the objects satisfying these properties is given in  the table
below and $f$ is semi--special
but not special if and only if it is an elementary quadratic transformation associated to a smooth quadric.
  \end{thm}

  \begin{table}[H]\label{tableJordan}
  \centering
  \begin{tabular}{|c|c|c|}
 
  \hline
\begin{tabular}{c}  {\bf Semi-simple rank 3 }  \\
{\bf Jordan algebra} 
\end{tabular}
   &  
   \begin{tabular}{c} 
   {\bf Smooth variety $X^n\subset \mathbb P^{2n+1}$,  3-RC} \\ 
 {\bf   by cubics, not of Castelnuovo type}
   \end{tabular}
   & 
   \begin{tabular}{c}    {\bf Special Cremona}\\ 
  {\bf 
 transformation}
\end{tabular}
 \\
 \hline     
   \begin{tabular}{l}
direct product $ \mathbb C\times \mathbb J$ with \\ 
 $\mathbb J$ rank 2   Jordan algebra
\end{tabular}
 &     
\begin{tabular}{l} Segre embedding 
${\rm Seg}(\mathbb P^1\times Q^{n-1})$
  \\
with $Q^{n-1}$ smooth hyperquadric
\end{tabular} & 
\begin{tabular}{c}
elementary \\
quadratic
\end{tabular}
  \begin{tabular}{c} 
  \\ 
    \end{tabular}
\\
\hline     
$ {\rm Herm}_3(\mathbb R_{\mathbb C})\simeq {\rm Sym}_3(\mathbb C)$  
 &     
\begin{tabular}{l} 6-dimensional Lagrangian  \\
grassmannian  $LG_3(\mathbb C^6)\subset \mathbb P^{13}$
\end{tabular} &$[x]\map [x^{\#}]$\\

  \hline
${\rm Herm}_3(\mathbb C_{\mathbb C})  \simeq M_3(\mathbb C)$  
 &     
\begin{tabular}{l} 9-dimensional Grassmannian \\
manifold $G_3(\mathbb C^6)\subset \mathbb P^{19}$
\end{tabular} &$[x]\map [x^{\#}]$\\
  \hline
$ {\rm Herm}_3(\mathbb H_{\mathbb C})\simeq{\rm Alt}_6(\mathbb C)$ 
 &     
\begin{tabular}{l} 15-dimensional orthogonal  \\
Grassmannian  $OG_6(\mathbb C^{12})\subset \mathbb P^{31}$
\end{tabular}& $[x]\map [x^{\#}]$\\
  \hline %\hspace{0.85cm}
$ 
 {\rm Herm}_3(\mathbb O_{\mathbb C})
$   & 
\begin{tabular}{l} 27-dimensional $E_7$-variety  
in   $ \mathbb P^{55}$
\end{tabular} &$[x]\map [x^{\#}]$
\\
  \hline
\end{tabular}
\vspace{0.15cm}
\end{table}

\begin{proof} Semi--special Cremona transformations are classified and they correspond to semi--simple Jordan algebras of rank 3, see
for example \cite[Proposition 5.6]{PR}, showing the equivalence between (ii) and (iii). It is known that the twisted cubics associated
to semi--simple Jordan algebras are smooth and they are described in the table 
 above.
{We proved the remaining implications in \cite[Theorem 5.7]{PR}}. 
 \end{proof}

    \subsection{\bf A  generalization of the $XJC$-equivalence covering some degenerate cases}
    \label{S:generalizedXJC}
   In order to formulate our main result we did not consider some extremal varieties 3-RC by cubics as well as  quadro-quadric Cremona transformations equivalent (as rational maps) to linear projective automorphisms.  In fact, the $XJC$-equivalence
    can be extended in order to cover these ``degenerated objects"
  as we shall see
   briefly in this subsection.

 \subsubsection{}   Let $\overline {\boldsymbol X}^n\!(3,3)$ be the set of projective equivalence classes of extremal $n$-dimensional irreducible varieties $X\subset \mathbb P^{2n+1}$ that are $3$-RC by twisted cubics.  It is just the union of $ {\boldsymbol X}^n(3,3)$ with the projective equivalence classes of the  scrolls    $S_{1,\ldots,1,3}$  (with  $n\geq 1$) and $S_{1,\ldots,1,2,2}$  (with $n>1$). 

\subsubsection{} By definition, a {\it norm} on a Jordan algebra $\mathbb J$ is
 a homogeneous form $\eta\in {\rm Sym}(\mathbb J^*)$ verifying $\eta(e)=1$ 
 and which  decomposes as a product of powers of the irreducible components 
of the generic norm $N$ of $\mathbb J$, see \cite{bk}. 
Then one defines   $\widetilde{\boldsymbol {\mathcal J \!ordan}}_{3}^{\,n}$ 
as the set of {\it  Jordan algebras with a cubic norm}, which by definition is 
 the set of pairs $(\mathbb J,\eta)$ where $\eta$ is a cubic norm on the Jordan 
algebra $\mathbb J$.   Since a rank 3 Jordan algebra admits a single cubic norm (the generic one),  
 ${\boldsymbol {\mathcal J \!ordan}}_{3}^n$ can be considered 
as a subset  of $\widetilde{\boldsymbol {\mathcal J \!ordan}}_{3}^{\,n}$.
   A Jordan algebra with a cubic norm is necessarily of rank less than or equal to $ 3$ and if the rank is less than 3, then it is isomorphic to 
one of the  following Jordan algebras:  
\begin{itemize}
\item  the rank 1 Jordan algebra $\mathbb C$, denoted by $\mathbb J_0^1$;\smallskip 
\item  the rank 2 Jordan algebra of Example  \ref{EX:jordan} (\ref{Ex:rank2}) with $q=0$, denoted by $\mathbb J_0^n$; \smallskip 
\item  the  rank 2 Jordan algebra of Example  \ref{EX:jordan} (\ref{Ex:rank2})  with $q$ of rank 1, denoted by $\mathbb J_1^n$.
\end{itemize}
The notation $\mathbb J_0^n$ is consistent since the `Jordan algebra' $\mathbb C$ can be described as in Example  \ref{EX:jordan}  (\ref{Ex:rank2}) by taking $W$ of dimension 0.\smallskip

For any $n\geq 1$,  the algebra $\mathbb J_0^n$ admits  a unique cubic norm, namely $\eta(\lambda,w)=\lambda^3$. For any $n>1$, the algebra $\mathbb J_1^n$ admits a 1-dimensional family of cubic norms.  Indeed, the generic norm on $\mathbb J_1^n$ is given by $N(\lambda,w)=\lambda^2+q(w)$. Since $q$ has rank 1, there exits a linear form $\ell\in W^*$ such that $q=\ell^2$ so that $N=\ell_+\, \ell_-$ with $\ell_\pm(\lambda,w)=\lambda\pm i \ell(w)$ for $(\lambda,w)\in \mathbb J_1^n$.  Then for every nonzero $(a,b)\in \mathbb C^2$, $\eta_{a,b}=(a \ell_++b \ell_-)  \ell_+ \ell_-$ is a cubic norm on  $\mathbb J_1^n$. \smallskip 

 One verifies that modulo isomorphism, 
 $\widetilde{{\boldsymbol {\mathcal J \!ordan}}}_{3}^n\setminus  {\boldsymbol {\mathcal J \!ordan}}_{3}^n$ consists of the two pairs $(\mathbb J_0^n,\lambda^3)$ and $(\mathbb J_1^n,{\eta_{a,b}})$
 when $n>1$, and reduces to $(\mathbb J_0^1,\lambda^3)$ when $n=1$.
 Let us define
$\overline{\boldsymbol {\mathcal J \!ordan}}_{3}^{\,n}$ as the set of pairs $(\mathbb J,\eta)$ verifying the compatibility relation
in Lemma \ref{L:crucial}, that is the partial derivatives of $\eta$ belong to the ideal generated by the quadratic forms defining the adjoint map. Thus when $n>1$, $\overline{\boldsymbol {\mathcal J \!ordan}}_{3}^{\,n}$ is the union of 
${\boldsymbol {\mathcal J \!ordan}}_{3}^{\,n}$ with the isomorphism classes of 
$(\mathbb J_0^n,\lambda^3)$ and of $(\mathbb J_1^n,\eta_{1,0})$. We shall indicate by $\widetilde{\mathbb J}^n_1$ the Jordan algebra
$\mathbb J_1^n$ with cubic norm $\eta_{1,0}$.

\subsubsection{}
\label{S:extCworld}
 Finally, let us return  to the corresponding Cremona transformations
to be considered in order to complete the picture. 
Consider   the set of  {\it normed quadro-quadratic Cremona transformation} of $\mathbb P^{n-1}$, that is of pairs $(f,[\eta])$ where $f=[f_1:\cdots:f_n]$  is  a birational map of $\mathbb P^{n-1}$ defined by quadratic forms $f_i$ and $[\eta]=\mathbb C^*\eta
 $ is the class of a  non-trivial  cubic form $\eta $  such that there exists a quadratic map  $G$  satisfying 
$G\big( f_1(x),\ldots, f_n (x)  \big)=\eta (x)x$ for every $x$. 
 Clearly, given   $f  \in  {{{\bf Bir}_{2,2}}}(\mathbb P^{n-1})$, there exists a unique $[\eta ]$ as above such that $(f,[\eta ])$ is a normed quadro-quadratic Cremona transformation and such that $(f,[\eta])$ satisfies the condition of Lemma \ref{L:crucial}. Therefore ${{{\bf Bir}_{2,2}}}(\mathbb P^{n-1})$ can be considered as a subset of 
${\overline{{\bf Bir}_{2,2}}}(\mathbb P^{n-1})$, which by definition it is the set of pairs $(f,[\eta])$ satisfying the compatibility relation in Lemma \ref{L:crucial}.  One verifies easily that if $(f,[\eta])\in {\overline{{\bf Bir}_{2,2}}}(\mathbb P^{n-1})$ but with $f$ not of bidegree $(2,2)$, then $(f,[\eta])$ is linearly equivalent to one of the following:
\begin{itemize}
\item  $({\rm Id}_{\mathbb P^{n-1}}, [\ell^3]) $ where  $\ell$ is a nonzero linear form, $n\geq 1$; \smallskip    
\item  $({\rm Id}_{\mathbb P^{n-1}}, [\ell^2\ell']) $ where  $\ell$  and $\ell'$ are  linearly independent linear forms,  $n> 1$.
\end{itemize}
\medskip 

Then the $XJC$-correspondence extends: there are bijection extending the ones in the $XJC$-diagram such that one has a commutative triangle  of equivalences  between the sets  introduced above: 
  \begin{equation*}
    \xymatrix@R=0.75cm@C=0.5cm{ 
  {\overline{ \boldsymbol { X}}}^n(3,3)_{\!\big/ \!\!\!
{\tiny{\begin{tabular}{l}
$projective$\\
$equivalence$
\end{tabular}}
}}
  \ar@{<->}[rr]  \ar@{<->}[rd]
  &        &   
     {\overline{\boldsymbol {\mathcal J \!ordan}}_{3}^{\, n}}_{\! \!\big/  isotopy} 
   \\ 
& 
{\overline{{\bf Bir}_{2,2}}}(\mathbb P^{n-1})_{\!\big/ \!\!\!
{\tiny{\begin{tabular}{l}
$linear$\\
$equivalence$
\end{tabular}}
}}
.   \ar@{<->}[ru] &  }
 \end{equation*}

For instance, let us explain how to associate an extremal variety 3-RC by cubics in $\mathbb P^{2n+1}$ 
to  a degenerate element  $( f,[\eta])$ of ${\overline{{\bf Bir}_{2,2}}}(\mathbb P^{n-1})$.  Let $F$ be a quadratic affine lift of $f$. 
Then as in  section \ref{S:CtoX}, one defines 
$ X_{f,[\eta]}$ as the Zariski closure of the image of 
the affine map $x\mapsto [1:x:F(x):\eta(x)]$.  
This variety  belongs to  ${\overline{ \boldsymbol { X}}}^n(3,3)$:  we let the reader 
verify that  the proofs of section \ref{S:CtoX} 
apply if one takes for $G$ the unique affine quadratic map 
such that $G(F(x))=\eta(x)x$ for every $x$. 

By the way let us remark that for $(\alpha,\beta)\in \mathbb C^2$
 such that $\alpha\beta\neq 0$, setting $\ell_{\alpha,\beta}=\alpha\ell_++\beta\ell_-$, one associates a  non-degenerate $n$-dimensional  variety in $\p^{2n+1}$ to the   pair $(\mathbb J_1^n,\eta_{\alpha,\beta})$  by defining it as the closure of the image of the 
  affine map 
$(\lambda, w)\mapsto [1:\lambda:w:\ell_{\alpha,\beta}(\lambda,w)\lambda:-\ell_{\alpha,\beta}(\lambda,w)w:\ell_{a,b}(\lambda,w) (\lambda^2+q(w))]$. 
 However, this variety is not 3-covered by twisted cubics 
since the compatibility relation in Lemma \ref{L:crucial}
 is not satisfied by $\eta_{\alpha,\beta}$.
\medskip 

In fact, this generalization of the $XJC$-correspondence does 
not present a very deep interest since it covers 
only two new cases  when  $n>1$, namely the ones described in the following two tryptics: 
\begin{equation*}
\begin{tabular}{ccc}
$    \xymatrix@R=0.75cm@C=0.5cm{ 
\big[   S_{1,\ldots,1,3} \big]   \ar@{<->}[rr]  \ar@{<->}[rd]
  &        &    \big[ \mathbb   J_0^n  \big]    \\ 
%&   & \\
&  \big[   ({\rm Id},[\ell^3]) \big]   \ar@{<->}[ru] &  }
$ &\qquad   &
 $ \xymatrix@R=0.75cm@C=0.5cm{ 
\big[   S_{1,\ldots,1,2,2} \big]   \ar@{<->}[rr]  \ar@{<->}[rd]
  &        &    \big[ \widetilde{\mathbb  J}_1^n  \big]    \\ 
%&   & \\
&  \big[   \big({\rm Id},[\ell^2\ell']\big) \big].   \ar@{<->}[ru] &   }$
 \end{tabular}
\end{equation*}\smallskip

In dimension $n=1$, one has 
$${\overline{\boldsymbol X}}^1\!(3,3) /{\! \! \!
{\small{\begin{tabular}{l}
$proj.$
\end{tabular}}}} \!\!\!=\big[   v_3(\mathbb P^1)  \big] 
\, , \quad 
 {\overline{\boldsymbol {\mathcal J \!ordan}}}_{3}^1/{\! \!\!
{\small{\begin{tabular}{l}
$isot.$
\end{tabular}}}} \!\!\!= \big[   \mathbb C \big] 
\quad \mbox{and}\quad
{\overline{{\bf Bir}_{2,2}}}(\mathbb P^{0}) /{\! \!\!
{\small{\begin{tabular}{l}
$lin.$
\end{tabular}}}} \!\!\!
= \big[  ( {\rm Id},x^3) \big]$$
hence the  generalized $XJC$-equivalence reduces in this case  to the following trivial tryptic
\begin{equation}
\label{E:v3P3}
    \xymatrix@R=0.75cm@C=0.5cm{ 
\big[   v_3(\mathbb P^1)  \big]   \ar@{<->}[rr]  \ar@{<->}[rd]
  &        &    \big[   \mathbb C  \big]    \\ 
%&   & \\
&  \big[   ({\rm Id}_{\mathbb P^0} ,x^3)
 \big].   \ar@{<->}[ru] &  }
 \end{equation}

Despite the very small number of new cases covered
 by the generalized $XJC$-correspondence, we inserted this extension into the discussion 
in order to show that  the elementary case  (\ref{E:v3P3}) 
can be included in the whole picture.  
Moreover,  the notion of  {\it normed quadro-quadric birational map}
 introduced in \ref{S:extCworld} will be used also
 to describe the general structure of Cremona transformations
 of bidegree $(2,2)$ in the next section.

\section{Further applications}

The theory of Jordan algebras is now  well established.  We recall some  general results on  the structure of Jordan algebras,  focusing 
especially  on rank 3 algebras.

\subsection{Some results on the structure of Jordan algebras} 
\label{S:Jordan-II}
Let  $\mathbb J$ be a fixed Jordan algebra of arbitrary rank $r\geq 1$. 
 For any subset $A\subset \mathbb J$,  one defines inductively the subsets $A^{(n)}\subset \mathbb J$ for any integer $n>0$ by setting 
$A^{(1)}=A$  and  $A^{(k+1)}=(A^{(k)})^2$ for every $k>0$.  If $A$ is a subalgebra of $\mathbb J$, the $A^{(n)}$ form a decreasing sequence of subalgebras $A=A^{(1)} \supset 
A^{(2)} \supset A^{(3)}\supset \cdots$.  By definition, $A$ is 
{\it solvable} if $A^{(t)}=0$ for a positive integer $t$.

 If $I_1,I_2$ are two solvable ideals of $\mathbb J$, it can be verified that $I_1+I_2$ is solvable too. Since $\mathbb J$ is finite dimensional, the union of all the solvable ideals of $\mathbb J$ is a  solvable ideal of $\mathbb J$, which is maximal for inclusion and which is 
 called  the {\it radical} of $\mathbb J$ and  denoted by ${\rm Rad}(\mathbb J)$,  or just by $R$ if there is no risk of confusion. \smallskip 
 
 The notion of solvability  introduced  above is not the most useful when working with Jordan algebras. Indeed,
 it can occur that for an ideal  $I\subset \mathbb J$, the subsets $I^{(k)}$ are not ideals  for some $k>2$.    Hence  it is not possible to construct   inductively a solvable ideal $I$  from its derived series $I=I^{(1)} \supset  I^{(2)} \supset \cdots  \supset I^{(r-1)} \supset  I^{(r)}=0$.
  To bypass this technical difficulty, Penico introduced in \cite{penico} the nowadays called {\it Penico's series} of an ideal $I$ as the family $I^{[k]}$, $k\geq 0$ defined inductively by 
 $$I^{[0]}=\mathbb J,\quad  
  I^{[1]}=I\quad \mbox{ and }  \quad
     I^{[k+1]}=\big(I^{[k]}\big)^2+ \big(I^{[k]}\big)^2 \mathbb J   \quad \mbox{ for } \, k\geq 1.$$

 The interest of this notion is twofold. First of all, it can be proved that $I$ is solvable if and only if it is {\it Penico-solvable}, that is if $I^{[s]}=0$ for a positive integer $s$. Moreover,  $I^{[k]}$ is an ideal for any $k\geq 1$, see \cite{penico}. 
 
 The notions introduced by Penico are more relevant than the classical ones to describe the structure of Jordan algebras. 
 Since $R$  is solvable, there exists a positive integer $t\geq 1$ such that $R^{[t]}=0$ and $R^{[t-1]}\neq 0$.  Since  the $R^{[k]}$'s are ideals in $\mathbb J$, the quotients  $\mathbb J^{[k]}=\mathbb J/ R^{[k]}$  are Jordan algebras for every $k\geq 1$, yielding, for 
 $\ell=2,\ldots,t$, the exact sequences:  
 \begin{equation}
 \label{E:nullradicalEXT}
  0\rightarrow {R^{[\ell-1]}}/{R^{[\ell]}}\longrightarrow \mathbb J^{[\ell]}
 \longrightarrow\mathbb J^{[\ell-1]}
 \rightarrow  0.
 \end{equation}
 
 Remark that the left hand side in these exact sequences  is an  ideal with trivial product because  $({R^{[\ell-1]}}/{R^{[\ell]}})^2=0$ for every $\ell$.  In the terminology of Jordan algebras, one says that $\mathbb J^{[\ell]}$ is a {\it null extension} of $\mathbb J^{[\ell-1]}$. 
 We can now recall the following important result:

 \begin{thm}[Albert \cite{albert}, Penico \cite{penico}] 
 \label{T:AlbertPenico}${}^{}$
 
 \begin{enumerate}
 \item[(1)] The quotient $\mathbb J_{\ss}=\mathbb J^{[1]}=\mathbb J/R$ is semi-simple, that is  ${\rm Rad}(\mathbb J_{\ss})=0$ or equivalently $\mathbb J_{\ss}$ is isomorphic to a direct product of simple Jordan algebras. \smallskip

   \item[(2)] The exact sequence of (non-unital) Jordan algebras  
 $
 0\rightarrow R \rightarrow \mathbb J\rightarrow \mathbb J_{\ss} \rightarrow 0 
 $ splits:    there exists an embedding of Jordan algebras $\sigma: \mathbb J_{\ss} \hookrightarrow \mathbb J$ such that 
 $\mathbb J=   \sigma( \mathbb J_{\ss}) \rtimes  R$. Moreover the embedding $\sigma$ is unique up to composition to the left by an automorphism of $\mathbb J$.\smallskip

  \item[(3)] The Jordan algebra $\mathbb J$ is obtained from its semi-simple part $\mathbb J_{\ss}$ by  the series of successive  null radical extensions  (\ref{E:nullradicalEXT}).
 \end{enumerate}
 \end{thm}

Simple Jordan algebras are completely classified so that  the first part of the previous result ensures that the semi-simple parts of an arbitrary Jordan algebra can be completely described.  The second part says that the structure of a general Jordan algebra  $\mathbb J$ is given by   its radical $R$ and by the structure of  $ \mathbb J_{\ss}$-module on it.  Finally, it comes from (3)  that the Jordan product on $R$ as well as its structure of $ \mathbb J_{\ss}$-module 
can be constructed inductively starting from $ \mathbb J_{\ss}$, by successive extensions of a very  simple kind.  \smallskip

\begin{ex}
\label{Ex:Ce/e3}
{\rm Being  associative and commutative, the algebra $A=\mathbb C[\varepsilon]/(\varepsilon^3)
$ can also be viewed as a 3-dimensional  Jordan algebra. One has $R_A=\Rad(A)=\langle \epsilon,  \epsilon^2\rangle $, $R_A^{[2]}=\langle \epsilon^2\rangle $ and 
$R_A^{[3]}=0$.  Thus the semi-simple part  $A_{\ss}=A/\Rad(A)$  has rank 1 and    is isomorphic to $\mathbb C$. }
\end{ex}

In the next section, using the $XJC$-correspondence, we shall state a version of  Theorem 
 \ref{T:AlbertPenico}
for quadro-quadric Cremona transformations and for twisted cubics over Jordan algebras.  We will use the following facts showing that the radical can be determined from the generic norm. 
 \begin{prop}{\rm ( \cite[0.15 and 9.10]{springer})}
\label{P:radical} 
For any Jordan algebra $\mathbb J$, one has 
$$ \Rad(\mathbb J)=\big\{     x\in \mathbb J \,   |  \, N(x+\mathbb J)=N(x) \, \big\}.$$

\end{prop}

We finish these reminders on Jordan algebras by stating some remarks on  the rank 3 case.
  Assume in what follows that $\mathbb J$ has rank 3 and for $x,y\in \mathbb J$, set
$$T(x,y)=T(xy)\qquad \mbox{ and } \qquad x\#y=(x+y)^\#-x^\#-y^\#.$$

By Proposition \ref{P:radical} (see also \cite{PeterssonRacine}), for a rank 3 Jordan algebra  one has
\begin{equation}
\label{E:RadicalRank3}
\Rad(\mathbb J)= \big\{     x\in \mathbb J \,|\, N(x)=T(x,\mathbb J) =T(x^\#,\mathbb J)= 0\big \}=\big\{     x\in \mathbb J \,|\,d^2\!N_x= 0\big \}.
\end{equation} 
Moreover, it can be verified that for $x\in \mathbb J$, the quadratic operator $U_x=2\,L_x^2-L_{x^2}$ is given by  
\begin{equation}
\label{E:UxRank3}
U_x(y)=T(x,y)x-x^\# \# y.
\end{equation}

It follows from  (\ref{E:RadicalRank3})  that  $T(r,\mathbb J)=0$ for every $r\in R$. Furthermore, one has $x\#y=d\#_x(y)=d\#_y(x)$ for every $x,y\in \mathbb J$. Since $I^{[2]}=U_I(\mathbb J)$ for any ideal $I$, the Penico series 
can also be defined inductively by 
$$
R^{[k+1]}=  d\#\big({R^{[k]}}\big) =\mathbb J  \# {R^{[k]}}=\big\langle
d\#_x({R^{[k]}}  ) \, \big|\,  x\in  \mathbb J \, 
\big\rangle\, .
%\, , \qquad k\geq 1.
$$

Finally, if $u\in \mathbb J$ is invertible, then the quadratic operator $U^{(u)}_x$ in the isotope $\mathbb J^{(u)}$ is given by $U_x^{(u)}=U_x U_u$ for every element $x$. Using this, one verifies easily that the Penico series  depends only on the isotopy class of $\mathbb J$.

\subsection{The general structure of quadro-quadric Cremona transformations}
\label{S:genSTRUCUTRE-Bir22}
A  consequence of the equivalence   between $\overline{{\bf Bir}_{2,2}}(\mathbb P^{n-1})/ \!\!\!{
{\tiny{\begin{tabular}{l}
$lin. $\vspace{-0.05cm}\\
$equiv.$
\end{tabular}
 }}}$ and $  \overline{{\boldsymbol {\mathcal J \!ordan}}}_{3}^{\, n}/{\! \! \!{\tiny{
 \begin{tabular}{l}
 $ isot.$
  \end{tabular}
  } }}$
 is a general structure theorem for quadro-quadric Cremona transformations,
obtained by translating  in the $C$-world the structure results for Jordan algebras presented above.
\medskip

The assertions below can be verified without difficulty and their proofs are left to the reader.

\subsubsection{\bf The radical}  
  \label{Ss:radical}
 Let $f$ be a quadro-quadric Cremona transformation of 
$ \mathbb P^{n-1}=\mathbb P(V)$ with baselocus scheme $\mathcal B_f\subset \mathbb P^{n-1}$.  
The secant scheme $\Sec(\mathcal B_f)$ of $\mathcal B_f$  is  the cubic hypersurface 
$V(N(x))\subset\p^{n-1}$, where $N(x)$ is the cubic form appearing in (\ref{E:bir22}).  This scheme
can be also considered as the {\it ramification scheme of $f$}, the name being justified by the fact that the locus
of points where the differential of the birational map $f$ is not of maximal rank is exactly  $V(N(x))_{\red}$, see also
\cite[\S 1.3]{CR}.
The {\it radical} of $f$ is  the set $R_f$ of points of multiplicity 3 of $\Sec(\mathcal B_f)$ 
and it has a natural scheme structure given
by  $R_f= V(d^2\!N_x)
\subset \mathbb P^{n-1} . $ 
The  support  of $R_f$, if not empty,   
is clearly a linear subspace of $\mathbb P^{n-1}$, contained in $\Sec(\mathcal B_f)$, and it is the {\it vertex of the cone} 
$\Sec(\mathcal B_f)=V(N(x))\subset\p^{n-1}$.  
We remark that $R_f$ can have  any dimension between $-1$ and $n-2$ (with the usual convention that the empty set is a subspace of dimension $-1$). 
The case  when $R_f$ is empty  corresponds to the semi-simple case and,  at the opposite side, 
$R_f$ is a hyperplane if and only if  $N(x)=L(x)^3$ with $L(x)$  linear form. 

\subsubsection{\bf The $JC$-correspondence in action}
\label{S:JCinAction}
Let $g$ be the quadratic inverse  of $f$. Let $F,G$ be some quadratic lifts of $f$, respectively $g$ and 
let  $R_F$ and $R_G$ be the affine cones over $R_f$ and $R_g$ respectively.  According to the $XJC$-equivalence (Theorem \ref{T:main}), there exist two linear maps  $L_1,L_2\in GL(V)$ such that $F=L_1^{-1}\circ \#_{\mathbb J}\circ L_2$  where 
$\#_{\mathbb J}$ denotes  the adjoint map of a rank 3 Jordan algebra $\mathbb J$. Then $G=L_2^{-1}\circ \#_{\mathbb J}\circ L_1$ and  $\Rad(\mathbb J)=L_2(R_F)=L_1(R_G)$. It is well known that   $(x+r)^\#-x^\#\in \Rad(\mathbb J)$  for every $x\in \mathbb J$ and every $r\in \Rad(\mathbb J)$, see \cite{springer} for instance.
In this setting, this gives us the following result.
\begin{lemma} If $x\in V$ and if $r\in R_F$,  then $F(x+r)-F(x)\in R_G$.
\end{lemma}

From the previous  Lemma, it follows that $F$ and $G$ pass to the quotient by $R_F$, respectively by $R_G$, inducing  quadratic affine morphisms
$\overline{F}: V/{R_F} \rightarrow V/{R_G}$, respectively $\overline{G}: V/{R_G} \rightarrow V/{R_F}$.  To understand what  these maps are, 
let $\mathbb J_{\ss}=\mathbb J/\Rad(\mathbb J)$ 
be the semi-simple part of $\mathbb J$, consider $L_1$ and $L_2$ as isomorphisms between 
$V$ and $\mathbb J$ inducing quotient maps  $\overline{L}_1: V/R_G\simeq \mathbb J_{\ss}$ 
and $\overline{L}_2: V/R_F\simeq \mathbb J_{\ss}$. 
The following diagram, in which all the  vertical arrows 
are the natural quotient maps, is commutative:
\begin{equation}
    \xymatrix@R=1.2cm@C=1.5cm{  
V   \ar@{->}[r]_{L_1}\ar@{->}[d] \ar@{->}@/^2pc/[rrr]^{F}&  \mathbb J \ar@{->}[d]   \ar@{->}[r]_{\#_{\mathbb J}} & \mathbb J \ar@{->}[r]_{L_2^{-1}}  \ar@{->}[d] & V   \ar@{->}[d] \\
{V}/{R_F} \ar@{->}[r]^{\overline{L}_1} 
\ar@{->}@/_2pc/[rrr]_{\overline{F}}
& \mathbb J_{\ss}    \ar@{->}[r]^{\#_{\mathbb J_{\ss}}} &\mathbb J_{\ss} \ar@{->}[r]^{\overline{L}_2^{-1}}  & {V}/{R_G}.}
 \end{equation}

\subsubsection{\bf The semi-simple part} 

Of course, $\overline{F}$ and $\overline{G}$  are  quadratic maps, 
each one being the inverse of the other in the sense used
till now. 
Indeed, if $N$ is the cubic form such that (\ref{E:bir22}) holds, 
 it passes to the quotient and induces a well-defined cubic form $\overline{N}$ on $V/R_F$ 
defined by 
${\overline{N}}({\overline{x}})=N(x)$   for  ${x}\in V$  (where  $\overline{x}$ stands 
for the class of $x$ modulo $R_F$). Moreover, 
$\overline{G}( \overline{F}(\overline{x})   )=\overline{N}(\overline{x})\, \overline{x}$
 for every ${x}\in V$ so that the pair $(\overline{F},[\overline{N}])$ is an element of 
 ${\overline{{\bf Bir}_{2,2}}}(\mathbb P(V/R_F))$ 
(the pair $(\overline{F},[\overline{N}])$ satisfies the statement of Lemma  \ref{L:crucial}). 
By definition, it  is  the  {\it semi-simple part} of $F$ and it is  denoted by $F_{\ss}$.  
In practice, one  identifies  $F_{\ss}$ with $\overline{F}$, which  is not a big deal 
since the cubic form $\overline{N}$ is always (essentially) determined.

\begin{ex}[{\bf{continuation of Example \ref{Ex:Ce/e3}}}]
\label{Ex:fss}
{\rm
The adjoint and the generic norm in $A=\mathbb C[\varepsilon]/(\varepsilon^3)$ are given by 
$(a,b,c)^\#=(a^2,-ab,b^2-ac)$ and $N(a,b,c)=a^3$    if $(a,b,c)$ stands for the coordinates of an element of $A$ relatively to the basis $(1,\varepsilon ,\varepsilon ^2)$.
The semi-simple part of $\#_A$ is the quadratic map  $a\mapsto a^2$,
which is a lift of the normed quadro-quadratic map $(a^2,[N])$ where $N$ is the cubic norm on 
$ {A}_{\ss}$ induced by the generic norm of $A$, i.e. $N(a)=N(a,0,0)=a^3$ 
for every $a\in {A}_{\ss}\simeq \mathbb C$.}
\end{ex}

We define the {\it semi-simple rank} $r_{\ss}(\mathbb J)$ of a Jordan algebra 
$\mathbb J$ as the rank of its semi-simple part 
$\mathbb J_{\ss}=\mathbb J/\Rad(\mathbb J)$ and the {\it semi--simple dimension} $\dim_{\ss}(\mathbb J)=\dim(\mathbb J_{\ss})$.
These notions are invariant up to isotopies
so that we can define the {\it semi-simple rank} $r_{\ss}(f)=r_{\ss}(F)$ 
of $f\in {{{\bf Bir}_{2,2}}}(\mathbb P(V))$, 
respectively the {\it semi--simple dimension} $\dim_{\ss}(f)$ 
(or of any affine lift $F\in {\rm Sym}^2(V^*)\otimes V$ of $f$), 
as the semi-simple rank of the associated isotopy class $[\mathbb J_f]$ 
of Jordan algebras, respectively as the semi-simple dimension of $[\mathbb J_f]$. 
In this way two new invariants (relatively to linear equivalence) 
of quadro-quadric Cremona transformations naturally appear. Let us see
how these definitions work in the simplest cases.
\medskip

\begin{ex}\label{ex:cremonap2}{\rm 
Modulo linear equivalence there are exactly three equivalence classes of quadro-quadric Cremona transformations
on $\p^2$, corresponding to the three isotopy classes of
cubic Jordan algebra of dimension three. Let us summarize the $JC$-correspondence and the semisimple
parts of these three classes in the following table:
\begin{table}[H]

%\hspace{-0.8cm}
  \centering
  \hspace{-0.4cm}
  \begin{tabular}{|c|c|c|c|c|c|} 
   \hline
   {\bf  Semi-simple }  
     &   {\bf Algebra}   & {\bf Cremona transformation}& {\bf Semi-simple}&{\bf Semi-simple} &{\bf Norm}  \\
%  \cline{3-4} 
   %{\bf Semi-simple dimension}
    {\bf  rank}  
    ${\boldsymbol{r_{\ss}}}$
     & ${\mathbb J_f}$  &   $f:\p^2\map\p^2$     &   {\bf part of $\mathbb J_f$}& {\bf part of $f$}& ${\boldsymbol{\overline{ N}}}$ \\
    \hline 
      1&$\frac{\mathbb C[\varepsilon]}{(\varepsilon^3)}$    &   $(x_1^2,-x_1x_2,x^2_2-x_1x_3)$     &$\mathbb C$&  $x_1^2$ & $x_1^3$  
     \\  \hline
  2&$\mathbb C\times \frac{\mathbb C[\varepsilon]}{(\varepsilon^2)}$ & $(x_2^2,x_1x_2,-x_1x_3)$ & $\mathbb C\times\mathbb C$& $(x_2^2,x_1x_2)$&$x_1x_2^2$
\\ \hline 
 3 &$\mathbb C\times\mathbb C\times\mathbb C$ & $(x_2x_3,x_1x_3,x_1x_2)$ & $\mathbb C\times\mathbb C\times\mathbb C$& $(x_2x_3,x_1x_3,x_1x_2)$&$x_1x_2x_3$\\
 \hline 
\end{tabular} \vspace{0.25cm}
\label{tablefp2}
\end{table}}
\end{ex}

\subsubsection{\bf Classification of the semi-simple part}
Since  a semisimple Jordan algebra is a direct product of simple ones 
and since  the classification of all simple Jordan algebras was obtained by Jacobson, see \cite{jacobson},  
the $JC$-correspondence provides  the complete classification 
of the semisimple parts of quadro-quadric Cremona transformations, yielding the following result.

\begin{prop}\label{P:semi} Let $F:V\to V$ be a  lift of a normed quadro-quadric 
Cremona transformation $f:\p(V)\map\p(V)$.  Then 
\begin{enumerate}
\item   $F$ is semi-simple ({\it ie.} $F=F_{\ss}$) if and only if $f$ is semispecial;
\item if $F$ is semi-simple, then $f$ is linearly equivalent to one of the 
quadro-quadric Cremona transformations listed in the third column 
of the table  below. 
\end{enumerate}
\end{prop}

 \begin{table}[H]
 \label{tableFss}
%\hspace{-0.8cm}
  \centering
  \hspace{-0.4cm}
  \begin{tabular}{|c|c|c|c|c|} 
   \hline
    {\bf  Semi-simple}   &     {\bf Semi-simple} &{\bf Ambient}   & \multicolumn{2}{|c|}{\qquad \qquad \qquad  {\bf Semi-simple part}      }  \\
%  \cline{3-4} 
  {\bf rank} ${\boldsymbol{r_{ss}}}$ &  {\bf dim.} ${\boldsymbol{\dim_{\ss}}}$&  {\bf space} ${\boldsymbol{\overline{V}}}$  &   ${\boldsymbol{\overline{F}}}$     &    ${\boldsymbol{\overline{N}}}$       \\
    \hline 
    $1$ &  1& $\mathbb C$    &   $\lambda \longmapsto \lambda^2$     &    $\lambda^3$  
     \\  \hline
 $2$  & $1+\dim(W)$& $\mathbb C\oplus W$ & $(\lambda,w)\longmapsto (\lambda^2, -\lambda w)$ & $\lambda(\lambda^2+q(w))$
\\ \hline 
 $2$  & 2& $\mathbb C\times\mathbb C$ & $(\rho,\lambda)\longmapsto (\lambda^2, \rho\lambda)$ & $\rho\lambda^2$
\\ \hline 
$3$  & $2+\dim(W)$& $\mathbb C\times (\mathbb C\oplus W)$ & $(\rho,\lambda,w)\longmapsto \big(  \lambda^2+q(w),       \rho \lambda, -\rho w\big) $ & $\rho(\lambda^2+q(w))$
\\ \hline 
$3$  & 6& ${\rm Sym}_3(\mathbb C)$ & $M \longmapsto {\rm Adj}(M)$   & $\det(M)$
\\ \hline 
$3$  & 9&${\rm M}_3(\mathbb C)$ & $M \longmapsto {\rm Adj}(M)$   & $\det(M)$
\\ \hline 
$3$  & 15&${\rm Alt}_6(\mathbb C)$ & $M \longmapsto M^{\#}$\quad \quad  
   & ${\rm Pf}(M)$
\\ \hline 
$3$  & 27&${\rm Herm}_3(\mathbb O\otimes \mathbb C)$ & $M \longmapsto M^\# $ \quad \quad    &  {\it cf.} (\ref{E:CubicNorm})
\\ 
\hline 
\end{tabular} \vspace{0.25cm}

\label{tableFss2}
 \caption{Explicit classification  of semi-simple parts of Cremona transformations of bidegree $(2,2)$. In this table,  $q$ stands for a nondegenerate quadratic form on a non-trivial vector space $W$, ${\rm Adj}(M)$ is the usual adjoint matrix while $M^\#$ is the adjoint
 in the corresponding algebra; $\det(M)$  is the usual determinant while 
 ${\rm Pf}(M)$ is the Pfaffian of an antisymmetric matrix;  
 the last line  is expressed using the general  formalism of the theory of Jordan algebras (see section \ref{S:Jworld} and also Table \ref{tableJordan}). }
\end{table}
  
 As  an application of the previous Proposition, 
we deduce two  classification results.
Let us recall that a homogeneous polynomial $P\in \mathbb C[x_1,\ldots, x_n]$ 
is called {\it homaloidal} if the associated polar map
$$%\phi_P=
P'=\bigg[\frac{\partial P}{\partial x_1}:\cdots:\frac{\partial P}{\partial x_n}\bigg]:\p^{n-1}\map\p^{n-1}$$ is birational. 
Let notation be as in \S \ref{Ss:radical}, and set 
 $\mathcal B_P=\mathcal B_{P'}=V(\frac{\partial P}{\partial x_1},\ldots,\frac{\partial P}{\partial x_n})\subset\p^{n-1}$. 
Assuming that $P'$ has bidegree $(2,2)$, we know that there exists a cubic form $N$ 
such that  $V(N)\subset\p^{n-1}$ is  the secant scheme of $\mathcal B_P$. 
Let us remark that since   $P'$ is birational, the partial derivates of $P$ are
 linearly independent so that $V(P)\subset\p^{n-1}$ is not a cone. In particular if $P$
has degree three, then it is necessarily 
  a reduced polynomial. We first classify reducible homaloidal
polynomials of degree three defining quadro-quadric Cremona transformations.
%\medskip

\begin{coro}\label{C:redhomaloidal} 
 Let  $P$ be a cubic homaloidal polynomial  in $n\geq 3$ variables 
such that $P'$ is a quadro-quadric Cremona transformation. If $P$ is reducible, then
one of the following holds:
\begin{itemize}
\item   $V(P)={Sec}(\mathcal B_P)$ and 
  $P$ is linearly equivalent to the norm of the   semi--simple (but not simple) rank 3   complex Jordan algebra of  the fourth line in Table 1 above;

\smallskip   
\item  $V(P)\neq {Sec}(\mathcal B_P)$ and   $V(P)$  is the union of a  smooth hyperquadric in $\mathbb P^{n-1}$ with   a tangent hyperplane; in some coordinates, one has 
$P(x)=x_1(x_2^2+\cdots+x_{n-1}^2-x_1x_n)$ and  $N(x)=x_1^3$. 
\end{itemize}
\end{coro}

\begin{proof}  If $P$ is the product of three  distinct linear forms, 
then necessarily $n=3$ and we are
in the first case, see also the third case in Example \ref{ex:cremonap2}. 

Suppose that $P$  is the product of a linear form $\ell$ with a quadratic form $q$. Without loss of generality we
can assume $\ell=x_1$. Let $Q=V(q)\subset\p^{n-1}$. There is an inclusion
of schemes $V(x_1)\cap Q\subset\mathcal B_P$ from which it follows that $V(x_1)$
  is contained in the secant locus scheme of  $\mathcal B_P$
so that it is contracted  by $P'$
and  $x_1$ is an irreducible factor of $N(x)$.
 
If the hyperquadric $Q$ is also contracted by $P'$, then  it is necessarily a cone 
with vertex a point
and we can suppose, modulo constants, 
$N(x)=x_1  q(x)=P(x)$. Since $V(N)=V(P)\subset\p^{n-1}$ is not a cone, the rank 3 Jordan algebra  $\mathbb J_{P'}$ has trivial radical by Proposition \ref{P:radical} hence is semi-simple.  Since $N=P$ is reducible, $\mathbb J_{P'}$ is not simple hence 
 we are in   the case corresponding to the fourth line  of
Table 2.                   

Finaly, if $Q$ is not contracted by $P'$, then it is a smooth hyperquadric. In this case the hyperplane $V(x_1)\subset\p^{n-1}$
is necessarily tangent to $Q$ at a point  (otherwise
$\mathcal B_P$ would be a degenerated smooth quadric $Q^{n-3}\subset\p^{n-2}$,
which is impossible) and we are in the second case.
 In the isotopy class $[\mathbb J_{P'}]$ we can choose
a representative such that $x^{\#}=(x_1^2,-x_1x_2,\ldots,-x_1x_{n-1},x^2_2+\cdots+x^2_{n-1}-x_1x_n)$ and $N(x)=x_1^3$.
\end{proof}

Following \cite{Dolgachev},  we will say that a homogeneous polynomial  $P\in\mathbb C[x_1,\ldots,x_n]$  such that $\det (\Hess(\ln P))\neq 0$ is  {\it EKP--homaloidal} if its multiplicative Legendre transform $P_*$ is again polynomial. In this case $P_*$ is a homogeneous polynomial function too and $\deg(P) =\deg(P_*)$, see also \cite{EKP} where this condition was investigated and studied. By the preliminary results of \cite{EKP}, a $EKP$--homaloidal polynomial is homaloidal and, after having identified  $\mathbb C^n$ with its bidual, we have 
\begin{equation}\label{keyhom}
\frac{P_*'}{P_*}\circ\frac{P'}{P}=I\!d_{\mathbb C^n}. 
\end{equation}
 
Therefore such a $EKP$--homaloidal polynomial  of degree $d$ 
defines a Cremona transformation of type $(d,d)$. If moreover $d=3$, 
it  follows from \eqref{keyhom} (combined with (\ref{E:bir22-2})) that we have $V(N)=V(P)$, that is $V(P)$ is the ramification locus scheme of $P'$. On the contrary as we shall see in  the proof of Corollary \ref{C:chaputsab}
below, if $P$ is a homaloidal cubic polynomial such that $V(N)=V(P)$, then it is  $EKP$--homaloidal.
Note that the $EKP$ condition defined above is not satisfied by the reducible polynomials of the type
described in the first case of Corollary \ref{C:redhomaloidal} where $V(N)$ is a cubic hypersurface supported on the tangent hyperplane.

\begin{coro}\label{C:chaputsab}
 Let  $P$ be a  homogeneous polynomial  in $n\geq 3$ variables. The following assertions are equivalent: 
 \begin{itemize}
 \item[$(1)$] $P$ is a  cubic $EKP$--homaloidal polynomial;\smallskip 
  \item[$(2)$]  $P$ is homaloidal, $P'$ has  bidegree $(2,2)$ and $V(P)=V(N)$;\smallskip 
 \item[$(3)$]  $P$ is  the norm of a semi--simple rank 3 complex Jordan algebra. \smallskip
\end{itemize}
When these assertions are verified, $P$ is linearly equivalent to one of the 
norms  in the last five lines of Table \ref{tableFss}.
\end{coro}

\begin{proof} We have seen before that $(1)$ implies $(2)$.   Assume that the latter is satisfied by $P$.   Since $P'$ has bidegree $(2,2)$, the $JC$-correspondence ensures  that (modulo composition by  linear automorphisms), one can assume that $P'$ is nothing but the adjoint map of a  rank three  complex Jordan algebra noted by $\mathbb J_P$.   Since $P$ is homaloidal, 
 $V(N)=V(P)\subset\p^{n-1}$ is not a cone.  By Proposition \ref{P:radical}, this implies  that the radical of $\mathbb J_P$ is trivial. Thus 
 $\mathbb J_P$  is semi--simple and the conclusion follows from the classification  recalled in Table 2.
\end{proof}

\begin{rem}\label{R:Verrahom}
{\rm For $n=3$ the two examples described in Corollary \ref{C:redhomaloidal},
 modulo linear equivalence, are the unique homaloidal polynomials by a result 
of Dolgachev without  any assumption on $\deg(P)$ and/or on  
$P'$, see \cite[Theorem 4]{Dolgachev}. For $n=4$ there exists
irreducible homaloidal polynomials of degree 3 whose associated Cremona transformation is of type $(2,3)$. One such example
is given by the equation of a special projection of the cubic scroll in $\p^4$ from a point lying in a plane generated by the directrix line and one
of the lines of the ruling, see \cite{CRS} for details and generalizations of this construction.
For $n\geq 4$ there exist irreducible homaloidal polynomials
of any degree $d\geq 2n-5$, see \cite{CRS}.

More related to the above results is a very interesting series of irreducible  cubic homaloidal polynomials communicated to us by A. Verra.
The associated polar map is an involution and hence of type $(2,2)$ but the ramification locus of these maps is different
from the associated cubic hypersurface. 
The construction of these polynomials is described in \cite{BV} but the details about the geometry
 of their
polar maps will be probably treated elsewhere. 
}
\end{rem}

\subsubsection{\bf The general structure of quadro-quadric Cremona transformations}   \label{S:GenStruQqCremona}
  In this section, we  translate Theorem \ref{T:AlbertPenico}  into the $C$-world as explicitely as possible. We continue to use the notation introduced in  \ref{S:JCinAction}. 
  \smallskip
  
   It will be useful  to denote by $V_F$  (respectively  $V_G$) the space $V$ considered as the source space of the map $F$ (respectively $G$).   If  $A_f$ is  a subset of $\mathbb P(V_F)$, we will denote by $A_F$ the affine cone over $A_f$ in $V_F$
   and we shall use the analogue notation for subsets in $\mathbb P(V_G)$ and in $V_G$.  One denotes  by $\overline{V}_{\!\! F}$,
   respectively $\overline{V}_{\!\!G}$, the quotient space $V_F/R_F$, respectively $V_G/R_G$,
   and by $\pi_f: \mathbb P(V_F)\dashrightarrow \mathbb P(\overline{V}_{\!\! F})$, respectively  $\pi_g: \mathbb P(V_G)\dashrightarrow \mathbb P(\overline{V}_{\!\! G})$, the rational map induced by the canonical linear projection. 
\smallskip

The interpretation of part  (1) in  Theorem \ref{T:AlbertPenico}  has  been explained in Section \ref{S:JCinAction}, see also part (1) in Theorem \ref{T:RadCrem} below for a precise statement.  In order to reinterpret part (2) and (3) of Theorem \ref{T:AlbertPenico} it is necessary  to  introduce some notions and to recall some definitions. 
\smallskip 

Let ${\rm Str}(f)$ be the {\it structure group} of $f$  (or rather of $F$) defined as  in 
\cite[\S 1.1]{springer}:  by definition, ${\rm Str}(f)$  is the 
set of linear automorphisms $\theta\in GL(V_F)$ such that 
$ F\circ \theta=\theta^\# \circ F  $ 
 for a certain  $\theta^\#\in GL(V_F)$. Of course, 
 ${\rm Str}(f)$ depends only on $f$ and 
 one verifies that it is  an algebraic subgroup of $GL(V_F)$. Moreover,  since $f$  is  invertible,  $\theta^\#$ is uniquely determined by $\theta$ and one verifies easily that the map $\theta\mapsto \theta^\#$ is an  isomorphism of algebraic groups from ${\rm Str}(f)$  onto  ${\rm Str}(g)$. \smallskip 
 
 \begin{ex}[{\bf{continuation of Example \ref{Ex:Ce/e3}}}]
 {\rm The structure group of  $F:(a,b,c)\rightarrow (a^2,-ab,b^2-ac)$ is  the subgroup of 
 invertible  triangular inferior complex  matrices $(m_{ij})_{i,j=1}^3$, 
 whose diagonal entries satisfy $m_{11}m_{33}=(m_{22})^2$.
 Since  $F\circ F(a,b,c)=a^3(a,b,c)$, the map $\theta\mapsto \theta^\#$ is an automorphism of ${\rm Str}(F)$.  It is given by 
$$
\begin{pmatrix}
\frac{{m_{22}}^2}{m_{33}}& 0 & 0\\
m_{21} & m_{22} & 0 \\
m_{31} &  m_{32}    &    m_{33} \\
 \end{pmatrix}\longmapsto 
\begin {pmatrix} {\frac {{m_{{22}}}^{4}}{{m_{{33}}}^{2}}}&0&0\\ 
\noalign{\medskip}-{\frac {{m_{{22}}}^{2}m_{{21}}}{m_{{33}}}}&{\frac {{m_{{22}}}^{3}}{m_{{33}}}}&0\\
 \noalign{\medskip}{\frac {-{m_{{22}}}^{2}m_{{31}}+{m_{{21}}}^{2}m_{{33}}}{m_{{33}}}}&-{\frac {m_{{22}} \left( -m_{{22}}m_{{32}}+2\,m_{{21}}m_{{33}} \right) }{m_{{33}}}}&{m_{{22}}}^{2}
\end {pmatrix}.
$$}\smallskip 
 \end{ex}

 Inspired by \cite[\S9]{springer},  we define an {\it ideal of  $f$}  as a pair $(I_f,I_g)$  of  projective subspaces $I_f\subset \mathbb P(V_F)$  and $I_{g}\subset \mathbb P(V_G)$ (necessarily of the same dimension) such that  
 \begin{equation}
 j_f(x+I_F)-j_f(x)\subset I_G \qquad \mbox{ and }\qquad j_g(y+I_G)- j_g(y)\subset I_F
 \end{equation}
 for $x\in V_F$ and $y\in  V_G$ generic, 
where $j_f:V_F\dashrightarrow V_G$ and $j_g:V_G\dashrightarrow V_F$ are the 
rational maps  considered in Section \ref{S:CtoJ}. 
 In this case, $j_f$ and $j_g$ factor through $I_F$ and $I_G$  and their associated projectivization $\tilde{f}$ and $\tilde{g}$ are (normed) quadro-quadric Cremona transformations such that the following diagram commutes: 
  \begin{equation}
 \label{E:C-extension}
    \xymatrix@R=1.2cm@C=2cm{  
\mathbb P(V_{\!F})  \ar@{-->}[d]   \ar@{-->}@/^1pc/[r]^{f}&  \mathbb P(V_{\!G}) \ar@{-->}@/^1pc/[l]_{g}\ar@{-->}[d] \\
\mathbb P\big({V}_{\!F}/I_F\big)    \ar@{-->}@/^1pc/[r]^{\tilde{f}}&  \mathbb P\big({V}_{\!G}/I_G\big) . \ar@{-->}@/^1pc/[l]_{\tilde{g}} }
 \end{equation}
 
If $(I_f,I_g)$ is an ideal then $ f\big(  \langle x,I_f\rangle   \big)= \langle f(x),I_g\rangle$  for  generic $x$  in $\mathbb P(V_f)$. This implies that $I_g$ is completely determined  by $I_f$ and vice-versa. Thus we can say that $I_f\subset \mathbb P(V_F)$ is an ideal of $f$ and that $I_g\subset \mathbb P(V_G)$ is an ideal of $g$.  We will say that $I_f$ and $I_g$ are {\it corresponding} ideals.

An ideal $I_f\subset \mathbb P(V_F)$ is {\it radical} if $I_f\subset R_f$. It is equivalent to the fact that   $I_g\subset R_g$.  
Although  the  definition of ideal of a $f$ as above was formulated  in the  affine setting,  there is a projective characterization of radical ideals. If $E_f$ and $E_g$ are two projective subspaces of $\mathbb P(V_F)$ and $\mathbb P(V_G)$ respectively,   one defines 
 $df(E_f)\subset \mathbb P(V_G)$ as the projectivization of $dF_{V_F}({E}_F)=F(V_F, {E}_F)\subset V_G$\footnote{When $dF_{V_F}(E_F)=0$, one sets $df(E_f)=\emptyset$.} and in the analogue way one defines  $dg(E_g)\subset \mathbb P(V_F)$.
 \begin{prop} 
     \label{P:Radical-C}
     Assume that $E_f\subset R_f$ and $E_g\subset R_g$. 
  Then the  following assertions are equivalent:
  \begin{enumerate}
  \item   $E_f$ and $E_g$ are corresponding radical ideals for $f$ and $g$ respectively;
  \smallskip 
   \item one has $df(E_f)\subset E_g$ and $dg(E_g)\subset E_f$.
  \end{enumerate}
  \end{prop}
  \begin{proof} 
  Let $P_f: V_F\rightarrow {\rm End}(V_G,V_F)$ and $P_g:V_G\rightarrow {\rm End}(V_F,V_G)$ be  the quadratic maps considered in Remark 
   \ref{R:CtoJPair}.  For any subset $A_G\subset  V_G$, one defines $P_f(A_G)\subset V_F$ as the span of the images of $A_G$ by the maps $P_f(x)$ for $x$ varying in $V_F$. We use the corresponding notation  for $P_g(A_F)$ with $A_F\subset V_F$.

    Adapting (\ref{E:UxRank3}) to our setting and using  (\ref{E:RadicalRank3}), we deduce
    that  $P_f(E_F)\subset E_G$ (resp.  $P_g(E_G)\subset E_F$) if and only if $dF(E_F)\subset E_G$ (resp. $dG(E_G)\subset E_F$).
Proposition 9.6 of \cite{springer}, translated in our setting, ensures that $E_f$ and $E_g$ are corresponding ideals if and only if $P_f(E_F)\subset E_G$ and $P_g(E_G)\subset E_F$, proving the equivalence of conditions  {\it (1)} and {\it (2)}.
\end{proof}
 \smallskip 
  
Given a (normed) quadro-quadric Cremona transformation $\tilde{f}$, any   quadro-quadric Cremona  map $f$ inducing $\tilde{f}$ on
the quotient  by the corresponding radical ideal $I_f\subset R_f$ and $I_g\subset R_g$ will be called a {\it radical extension} of $\tilde{f}$ by $I_f$ (or by $(I_f,I_g)$).  This (radical) extension is {\it null} if the restriction of $f$ to a generic fiber of the canonical projection $\mathbb P(V_F)\dashrightarrow \mathbb P(V_F/I_F)$ is equivalent (as a rational map) to a linear map.

\begin{ex}[{\bf{continuation of Example \ref{Ex:Ce/e3}}}]
\label{Ex:ftilde}
{\rm 
 Let $f$ be the  projectivization of the map $F:\mathbb C^3\rightarrow \mathbb C^3$ defined by $F(a,b,c)=(a^2,-ab,b^2-ac)$, that is nothing but the adjoint map of the algebra $ \mathbb C[\varepsilon]/(\varepsilon^3)$ expressed in the basis $(1, \varepsilon, \varepsilon^2)$.
Then $f$ is a null  radical extension of the normed quadro-quadric Cremona transformation 
$\tilde{f}: \mathbb P^1\dashrightarrow \mathbb P^1$ defined as the projectivization of 
the quadratic map $\tilde{F}:  
(a,b)\mapsto (a^2,-ab)$ with associated cubic norm $\tilde N(a,b)=a^3$.}
\end{ex}

Let us define inductively a family of projective subspaces  of $\mathbb P(V_F)$ and $\mathbb P(V_G)$ by setting  $R_f^{[1]}=R_f$, $R_g^{[1]}=R_g$  and 
 $$\qquad 
 R_f^{[k+1]}=dg\big({R_g^{[k]}}\big) \subset \mathbb P\big(V_F \big)\; ,  \quad R_g^{[k+1]}=df\big({R_f^{[k]}}\big)\subset \mathbb P\big(V_G  \big) \qquad \mbox{for } k\geq 1.\smallskip  
 $$

By definition, 
$(R_f^{[k]})_{k\geq 0}$ is the {\it `Penico series'} of $f$.  
  It follows  from Proposition \ref{P:Radical-C} that  this is a  decreasing
  \vspace{-0.1cm} \\
   series of  radical ideals of 
 $\mathbb P(V_f)$. Moreover, $R_f^{[k]}$ and $R_g^{[k]}$ are corresponding ideals for every $k\geq 1$.  Then, passing to the quotients, 
 the map $f $ induces  a  normed quadro-quadric Cremona transformation
  \begin{equation*}
    \xymatrix@R=1.2cm@C=2cm{  
f^{[k]} : \, \mathbb P\big(\overline{V}_{\!\!F}^{[k]}\big) 
   \ar@{-->}[r]&  \mathbb P\big(\overline{V}_{\!\!G}^{[k]}\big) }
 \end{equation*}
for every $k\geq 1$,  where $\overline{V}_{\!\!F}^{[k]}$ and $\overline{V}_{\!\!G}^{[k]}$ stand for the quotients spaces $V_F/R_F^{[k]}$ and $V_G/R_G^{[k]}$ respectively.
 \bigskip 

We can now state the `$C$-version'  of Theorem \ref{T:AlbertPenico}:   
\smallskip

   \begin{thm} 
 \label{T:RadCrem}${}^{}$
  \begin{enumerate}
 \item[(1)] The Cremona transformations $f$ and $g$ factor through $R_f$ and $R_g$:  there are semi-simple (normed) quadro-quadric Cremona transformations
 $\overline{f}$ and $\overline{g}$ such that the following diagram commutes  
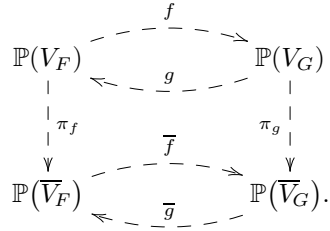
\begin{equation*}
    \xymatrix@R=1.2cm@C=2cm{  
\mathbb P(V_F)  \ar@{-->}[d]^{\pi_f}   \ar@{-->}@/^1pc/[r]^{f}&  \mathbb P(V_G) \ar@{-->}@/^1pc/[l]_{g}\ar@{-->}[d]_{\pi_g} \\
\mathbb P\big(\overline{V}_{\!\!F}\big)    \ar@{-->}@/^1pc/[r]^{\overline{f}}&  \mathbb P\big(\overline{V}_{\!\!G}\big).  \ar@{-->}@/^1pc/[l]_{\overline{g}} }
 \end{equation*}
 \medskip

 \item[(2)]  There exist linear embeddings $\sigma_f: \mathbb P(\overline{V}_{\!\!F})\hookrightarrow \mathbb P({V}_F)$ and $\sigma_g: \mathbb P(\overline{V}_{\!\!G})\hookrightarrow \mathbb P({V}_G)$ whose images are  linear spaces supplementary to $R_f$ and $R_g$ respectively, such that the diagram below commutes:   
 
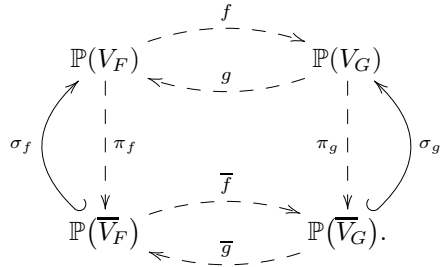
\begin{equation*}
    \xymatrix@R=1.7cm@C=2cm{  
\mathbb P(V_F)  \ar@{-->}[d]^{\pi_f}   \ar@{-->}@/^1pc/[r]^{f}&  \mathbb P(V_G) \ar@{-->}@/^1pc/[l]_{g}\ar@{-->}[d]_{\pi_g} \\
\mathbb P\big(\overline{V}_{\!\!F}\big)    \ar@{-->}@/^1pc/[r]^{\overline{f}}
 \ar@{_{(}->}@/^2pc/[u]^{\sigma_f}
&  
\ar@{^{(}->}@/_2pc/[u]_{\sigma_g}
\mathbb P\big(\overline{V}_{\!\!G}\big) \ar@{-->}@/^1pc/[l]_{\overline{g}}. 
}
 \end{equation*}
 Moreover, the pair $(\sigma_f,\sigma_g)$ is unique modulo the action of the structure group
 given by 
 $$\qquad \gamma\cdot (\sigma_f,\sigma_g)=(\gamma\circ \sigma_f,\gamma^\#\circ \sigma_g)
  \quad  \mbox{for }\, \gamma \in {\rm Str}(f). \smallskip
 $$
  \item[(3)]   The radical $R_f$  of $f$  is solvable: there exists $t>0$ such that $R_f^{[t]}$ is empty.  Moreover,  $f^{[\ell]}$ is a radical 
  \vspace{-0.1cm}\\
  null extension of $f^{[\ell-1]}$ for $\ell=2,\ldots,t$ so that   $f$ can be obtained from its semi-simple part $\overline{f}$ by the successive series of non-trivial null radical extensions represented by the commutative diagram
     \begin{equation*}
    \xymatrix@R=1cm@C=0.7cm{ 
 \mathbb P\big(V_F\big)   \ar@{-->}[r]  \ar@{-->}[d]_{f=f^{[t]}}  &   \cdots \ar@{-->}[r] &   \mathbb P\big(\overline{V}_{\!\!F}^{[\ell]}\big)  \ar@{-->}[d]_{f^{[\ell]}} \ar@{--}[r] &   \ar@{-->}[r]&   \mathbb P\big(\overline{V}_{\!\!F}^{[\ell-1]}\big)    \ar@{-->}[d]^{f^{[\ell-1]}}      \ar@{-->}[r]& \cdots \ar@{-->}[r]&  \mathbb P\big(\overline{V}_{\!\!F}\big) 
    \ar@{-->}[d]^{f^{[1]}=\overline{f}} 
  \\
 \mathbb P\big(V_G\big)   \ar@{-->}[r] & \cdots \ar@{-->}[r]&
  \mathbb P(\overline{V}_{\!\!G}^{[\ell]}) \ar@{--}[r] &  \ar@{-->}[r]  &
  \mathbb P\big(\overline{V}_{\!\!G}^{[\ell-1]}\big) 
    \ar@{-->}[r]& \cdots \ar@{-->}[r]&  \mathbb P\big(\overline{V}_{\!\!G}\big) \, ,      
 }
 \end{equation*}
where all the horizontal maps are the ones induced by the canonical linear projections 
$\overline{V}_{\!\!F}^{[\ell]}\rightarrow \overline{V}_{\!\!F}^{[\ell-1]}$. 
   \end{enumerate}
 \end{thm}

\begin{ex}[{\bf{continuation of Example \ref{Ex:Ce/e3}}}]
{\rm We apply the third part of the previous result to the projectivization $f$ of the adjoint map $F(a,b,c)=(a^2,-ab,b^2-ac)$ of  $A=\mathbb C[\epsilon]/(\epsilon^3)$. The associated  Penico series  is $\emptyset =R_f^{[3]}\subset R_f^{[2]}=\mathbb P\langle \varepsilon^2\rangle \subset R_f^{[1]}=\mathbb P R_A=
\mathbb P \langle \epsilon,\epsilon^2\rangle
$.  Hence one has $f^{[3]}=f$, $f^{[1]}=\overline{f}$ (see Example \ref{Ex:fss}) and  $f^{[2]}: \mathbb P^1\dashrightarrow \mathbb P^1$ is nothing but the normed quadro-quadric Cremona transformation $\tilde{f}$ of Example \ref{Ex:ftilde}.}
\end{ex}\medskip

We also point out an immediate consequence of the previous result in the affine setting.

 \begin{coro} 
  \label{T:structureBir22}
  Let $F$ be an affine   lift of a quadro-quadric Cremona transformation $f$.  
Set $R=R_F$ and $\overline{V}=V/R$.  Then 
  $F$ is linearly equivalent to a quadratic  map  of the form 
 \begin{align*}
 \overline{V}\oplus R  \; & \longrightarrow \quad \overline{V}\oplus R\\ 
 (x,r) & \longmapsto \big(  \overline{F}(x), \mathcal F(x,r)+\mathscr F(r)   \big)
  \end{align*}
  where   \begin{itemize}
  \item[--] $\overline{F}$ is the semi-simple part of $F$  (so is equivalent to one of the  quadratic maps in Table \ref{tableFss});\smallskip 
  \item[--] $\mathcal F:  \overline{V}\times R\rightarrow R$ is a bilinear map; \smallskip
  \item[--] $\mathscr F: R\rightarrow R$ is  a quadratic map such 
that $\mathscr G \circ \mathscr F\equiv 0$ for another nontrivial quadratic map 
$\mathscr G:R\rightarrow R$.\smallskip 
  \end{itemize}
  
  Moreover, $f$ is a null extension of its semi-simple part if and only if the quadratic map $\mathscr F$ vanishes identically.
   \end{coro}

\begin{ex}[{\bf{continuation of Example \ref{Ex:Ce/e3}}}]
\label{Ex:Penico-f}
  {\rm 
 Let us  consider again the quadratic map $F(a,b,c)=(a^2,-ab,b^2-ac)$.  With the notation  of  the previous  corollary,  one has  $V=A=\mathbb C[\epsilon]/(\epsilon^3)$, $\overline{V}=\mathbb C \, 1$, $R=\epsilon \, A\simeq \mathbb C^2$, $\overline{F}(a)=a^2$, $ \mathcal F(a,(b,c))=(-ab,-ac)$ and $  \mathscr F(b,c)=(0,b^2)$ for every $(a,b,c)\in V=\mathbb C^3$. }
  \end{ex}

\subsection{The general structure of twisted cubics over Jordan algebras}
  In this section, $X\subset \mathbb P^{2n+1}$ stands for a fixed element of $\boldsymbol{X}^n(3,3)$ with $n\geq 3$. According to the `$XJ$-correspondence', one can assume that  there exists a rank 3 Jordan algebra $\mathbb J$ of dimension $n$ such that $X=X_\mathbb J$. 
  In what follows, we will write 
     $Z_2(\mathbb J)= \mathbb C\oplus \mathbb J\oplus \mathbb J\oplus \mathbb C$ for simplicity and 
     $(\alpha,x,y,\beta)$ will stand for linear coordinates on $Z_2(\mathbb J)$ corresponding to the decomposition in direct sum of the complex vector space $Z_2(\mathbb J)$.   According to our hypothesis, $X$ is the closure of the image of the affine embedding 
     $\mu=\mu_{\mathbb J}: \mathbb J\hookrightarrow \mathbb PZ_2(\mathbb J): x\mapsto [1:x:x^\# : N(x)]$ considered  in Section \ref{S:fromJ}. Moreover, there exists a  family $\Sigma_X$ of twisted cubics included in $X$, which is  3-covering and unique.\smallskip 
  
  In order to state the `$X$-version' of Theorem  \ref{T:AlbertPenico} we shall introduce some terminology and recall
   some preliminary results.

    \subsubsection{\bf The conformal group} 
     
     The {\it structure group} ${\rm Str}(\mathbb J)$ of $\mathbb J$ is the algebraic subgroup ${\rm Str}(\#_{\mathbb J})$ of $GL(\mathbb J)$ associated to the adjoint map 
       of $\mathbb J$ defined in  Section \ref{S:GenStruQqCremona} above. 
    It can be verified that there exists a non-trivial character $  \gamma\mapsto \eta_\gamma$ on ${\rm Str}(\mathbb J)$ such that $N(\gamma(x))=\eta_\gamma N(x)$ for every $\mathbb J$ and every $\gamma\in {\rm Str}(\mathbb J)$.  
Then one defines  
the {\it conformal group} of $\mathbb J$, denoted by ${\rm Conf}(\mathbb J)$, as the  subgroup  of the group
        of affine birational transformations of $\mathbb J$ generated by ${\rm Str}(\#_{\mathbb J})$, by the inversion $j:x\dashrightarrow x^{-1}$ and by the translations $t_w:x\mapsto x+w$ (with $w\in \mathbb J$). 
     \smallskip

The projective representation $\rho: {\rm Conf}(\mathbb J)\rightarrow PGL(Z_2(\mathbb J))$ is defined in the following way:
\begin{align*}
     \rho( j) \cdot \theta= & \, [\beta: y: x:\alpha]  \\
           \rho(\gamma)\cdot \theta= & \, [\alpha: \gamma(x): \gamma^\#(y):\eta_\gamma \beta]\\
  \mbox{and }\;   \rho(t_w) \cdot \theta= & \, [\alpha: x+\alpha w: y+ x\#w+\alpha w^\#: \beta+T(y,w)+T(x,w^\#)+\alpha N(w) ]      \end{align*}
     for every $\theta=[\alpha:x:y:\beta]\in \mathbb P Z_2(\mathbb J)$, $ \gamma\in {\rm Str}(\mathbb J)$ and $ w\in \mathbb J$.  It can be verified that $\rho(j)\cdot \mu=\mu\circ j$, $\rho(\gamma)\cdot \mu=\mu\circ \gamma$ and $\rho(t_w)\cdot \mu=\mu \circ t_w$ for every structural transformation $\gamma$ and every translation $t_w$.  This implies that the image of $\rho$ in $PGL(Z_2(\mathbb J))$ 
     is contained in the  group ${\rm Aut}(X)$ of projective automorphisms of $X$. 
%         \begin{prop}
  %   \label{P:ConfX=AutX} The projective representation $\rho: {\rm Conf}(\mathbb J)\rightarrow PGL(Z_2(\mathbb J))$
   %  is faithful and  $\rho({\rm Conf}(\mathbb J))={\rm Aut}(X)$.
      %   \end{prop}
          \begin{prop}
     \label{P:ConfX=AutX} The  representation $\rho: {\rm Conf}(\mathbb J)\rightarrow PGL(Z_2(\mathbb J))$
     is faithful and  $\rho({\rm Conf}(\mathbb J))={\rm Aut}(X)$.
         \end{prop}
     \begin{proof}
     Let $\varphi\in {\rm Aut}(X)$.  Being projective, it induces an automorphism of $\Sigma_X$, again denoted by $\varphi$.   
     The twisted cubic  $C_0=\{[s^3:s^2 t e: s t^2 \, e: t^2] \in \mathbb PZ_2(\mathbb J)\, \big\lvert \, [s:t]\in \mathbb P^1  \}$ is included in $X$ and passes through the three points $0_X=\mu(0)=[1:0:0:0]$,  $e_X=\mu(e)=[1:e:e:1]$ and  $\infty_X=\mu(\infty)=[0:0:0:1]$ of $X$.   Since ${\rm Conf}(\mathbb J)$ acts transitively on generic 3-uples of points in $X$ (see \cite[Proposition 4.7]{PR} for instance),  it comes that the orbit  $\Gamma={\rm Conf}(\mathbb J)\cdot C_0$  of $C_0$ is dense in the 3-covering family $\Sigma_X$ of cubics included in $X$. Thus there exists $C\in \Gamma$ such that $\varphi(C)\in \Gamma$
     and, modulo compositions on the left and on the right by conformal automorphisms of $X$, we can assume  $\varphi(C_0)=C_0$. Moreover, since the subgroup of conformal transformations of $X$ fixing $C_0$ acts  as ${\rm Aut}(C_0)\simeq PGL_2(\mathbb C)$ on $C_0\simeq \mathbb P^1$, we can also suppose that $\varphi$ fixes the points $0_X$ and $\infty_X$ of $C_0$.  It follows that $\varphi$ induces an automorphism of the subfamily  $\Sigma^*_X$ of $ \Sigma_X$ formed by the twisted cubics included in $X$  and passing through $0_X$ and $\infty_X$. 
  \smallskip 
  
     Let us denote by   $V^+$ and $V^-$  the abstract tangent spaces of $X$ at $0_X$ and  $\infty_X$ respectively.   
     The differential $\varphi_{0}=d\varphi_{0_X}$ (resp. $\varphi_{\infty}=d\varphi_{\infty_X}$) of $\varphi$ at $0_X$ (resp. at $\infty_X$) is a  linear automorphism of $V^+$ (resp. of $V^-$).  Then reasoning as in  the proof of Theorem
     \ref{T:XJ} we deduce that $(\varphi_0,\varphi_\infty)$ is an automorphism of  the Jordan pair  $(V^+,V^-)$  constructed  there.
      Taking 
     $u=\exp^{-1}(e_X)\in V^-$ 
     as  invertible element (see the notation at the end of the proof of Theorem
     \ref{T:XJ}), one obtains that $V^+_u=\mathbb J$ as Jordan algebras.  It follows then from \cite[Proposition 1.8]{loos}  that $\varphi_0\in {\rm Str}(\mathbb J)$ and  $\varphi_{\infty}=(\varphi_0^\#)^{-1}$.      
          \smallskip 
          
          Let us now prove that the action of $\psi_0=\rho(\varphi_0)$ on $X$ coincides with that of $\varphi$.  Let $x\in X$ be a general point. 
          By hypothesis, there exists a twisted cubic  $C_x \in \Sigma_X^*$ passing through $x$ that is unique according to \cite{PT} or \cite[Theorem 2.4 (1)]{PR}. This implies that the tangent map $\Sigma_X^*\rightarrow \mathbb P(V^+)$ that associates to $C_x$ its projective tangent line at $0_X$ is 1-1 onto its image.   In particular, this gives us that $\varphi(C_x)=\psi_0(C_x)$. Since $\varphi_x=\varphi\lvert_{C_x}: C_x\rightarrow \varphi(C_x)$ is a projective isomorphism that lets $0_X$ and $\infty_X$ fixed, it is completely determined by its differential at one point, for instance at $0_X$.  Since $d(\varphi)_{0_X}=d (\psi_0)_{0_X}=\varphi_0$, this shows that  $\varphi_x$ and $\psi_0\lvert_{C_x}$ coincides so that  $\psi_0(x)=\varphi_x(x)=\varphi(x)$. From the generality of $x\in X$, we deduce  that $\varphi=\psi_0$.
            Since  $\psi_0=\rho(\varphi_0)$ with  $\varphi_0\in {\rm Str}(\mathbb J)$,  this gives us that   $\rho({\rm Conf}(\mathbb J))={\rm Aut}(X)$.
          \smallskip 
          
            Finally,  let $\nu\in {\rm Conf}(\mathbb J)$ such that $\rho(\nu)={\rm Id}_X$.  From $\mu=\rho(\nu)\cdot \mu=\mu\circ \nu$ one gets that $x=\nu(x)$ for $x\in \mathbb J$ generic, that is $\nu={\rm Id}_{\mathbb J}$. Thus $\rho$ is  faithful and  the  result
            is proved.
     \end{proof}

     A different proof of the previous result  is given in \cite{Freudenthal} for the  case when $\mathbb J={\rm Herm}_3(\mathbb O_{\mathbb C})$. The proof therein clearly applies to all twisted cubics over  semi-simple rank 3 Jordan algebras but it is not clear wheter  it can be applied to the general case.  We have included the proof above because we were unaware of any proof of Proposition \ref{P:ConfX=AutX} in the literature, despite this result is certainly well-known to  the experts of this field.

     \begin{rem}   {\rm The proof of Proposition \ref{P:ConfX=AutX} also shows that the subgroup  of projective automorphisms of $X$ fixing two (resp. three) general points of $X$ is isomorphic to the structure group  (resp. to the automorphism group) of the Jordan algebra $\mathbb J$.  This is related to the considerations in   Remark \ref{R:HeuristicRemark}.}
     \end{rem}

    \subsubsection{\bf The radical and the semi-simple part} 
    \label{S:Xradical}
    We  use again  here some notation and construction introduced in the proof of Theorem  \ref{T:XJ}:   $x^+$ and $x^-$ are two general points on $X$  such that 
$X$ is 1-RC by the family $\Sigma_{x^+x^-}$   of twisted cubics included in $X$ passing through $x^+$ and $x^-$.  For $\sigma=\pm$, one defines a rational  map $F^\sigma:V^\sigma \dashrightarrow V^{-\sigma}$  by setting 
$$ F^\sigma(v)=d\alpha_v^\sigma(1:t)/dt \big\lvert_{t=0}
$$
for $v\in V^\sigma$, where $\alpha_v^\sigma: \mathbb P^1\rightarrow X $ is the projective parametrization of  a twisted cubic belonging to  $\Sigma_{x^+x^-}$ such that  $\alpha_v^\sigma(0:1)=x^\sigma$, $\alpha_v^\sigma(1:0)=x^{-\sigma}$ and $d\alpha_v^\sigma(s:1)/ds \lvert_{s=0}=v$.   The
map  $F^\sigma$ is homogeneous  of degree -1 and   $F^\sigma\circ F^{-
\sigma}={\rm Id}_{V^\sigma}$ for every $\sigma=\pm$. Thus the associated projectivization 
$f^\sigma: \mathbb P(V^\sigma)\dashrightarrow \mathbb P(V^{-\sigma})$ of $F^\sigma$  is a Cremona transformation with inverse $f^{-\sigma}: \mathbb P(V^{-\sigma})\dashrightarrow \mathbb P(V^{\sigma})$.  It can be verified that $f^\sigma$ has bidegree $(2,2)$ and that it is nothing but the map $\varphi_{X,x^\sigma}$ defined in  Section  \ref{S:fromXworld} (up to linear equivalence). \smallskip 

Let $R_{f^\sigma}\subset \mathbb P(V^\sigma)$ be the radical of $f^\sigma$ as defined in Section \ref{Ss:radical}. Since $\mathbb P(V^\sigma)$ identifies canonically with 
the projective quotient $T_{x^\sigma } X/\langle x^\sigma\rangle$, one can define the cone $R_{x^\sigma}\subset T_{x^\sigma } X$ over $R_{f^\sigma}$ with vertex $x^\sigma$ (alternatively, $R_{x^\sigma}$ can be defined as the closure of the radical $R_{F^\sigma}\subset V^\sigma$ of $F^\sigma$ in   the natural affine embedding $V^\sigma\subset T_{x^\sigma } X$).   
%Since $\mathbb P^{2n+1}=\langle X\rangle =T_{x^+} X\oplus  T_{x^-} X$, it comes that $R^\sigma$ embeds naturally in the ambiant space $\mathbb P^{2n+1}$ for $\sigma=\pm$.  
By definition, the radical $R_{x^+x^-}$ of $X$ relatively to the pair $(x^+,x^-)$ is the direct sum of $R_{x^+}$ and $R_{x^-}$ in $\mathbb P^{2n+1}$: 
 $$R_{x^+x^-}=R_{x^+}\oplus R_{x^-}\subset T_{x^+} X\oplus  T_{x^-} X=\langle X\rangle=\mathbb P^{2n+1}\, .$$ 

  A straightforward verification proves the following result.
   
\begin{lemma}
The radical $R_{x^+x^-}$ does not depend on the pair $(x^+,x^-)$ but only on $X$.
\end{lemma}

We can thus define the {\it radical of $X$} as the projective subspace $R_X=R_{x^+x^-}\subset \mathbb P^{2n+1}$ for any generic pair $(x^+,x^-)$ of elements of $X$.  If  $r$ is the dimension of the radical of $\mathbb J$ then $R_X$ is a projective subspace of dimension $2r-1 $ in $\mathbb P^{2n+1}$ that is projectively attached to $X$, that is one has  $\varphi(R_X)=R_X$ for every $\varphi\in {\rm Aut}(X)$.    
    \smallskip 
  
   The preceding definition of the radical of $X$ makes quite explicit the link with the corresponding notion in the $C$-world (hence in the $J$-world). Notwithstanding  we think it is interesting to provide a purely projective definition of $R_X$.  To this end, we
   first remark that since two generic projective tangent spaces of $X$ are in direct sum, $X\subset \p^{2n+1}$ has the secant variety $\sigma(X)$ filling the whole space by Terracini Lemma, {\it i.e.} $\sigma(X)=\p^{2n+1}$. Moreover $X$  is also tangentially non-degenerate, {\it i.e.}  the tangent variety $\tau(X)\subset \mathbb P^{2n+1}$ of $X$,  defined as the closure   of the union of the lines tangent to the smooth locus of $X$, is a hypersurface in $\mathbb P^{2n+1}$.

    \begin{lemma} The tangent variety $\tau(X)$ is the hypersurface in $\mathbb P^{2n+1}$ cut out by the irreducible quartic form 
        $$Q\big(\alpha,x,y,\beta\big)=T(x^\#,y^\#)-\beta N(x)-\alpha N(y)-\frac{1}{4}\big( T(x,y)-\alpha\beta \big)^2.
    $$
    \end{lemma}
    \begin{proof}  Since $\tau(X)$ is irreducible and singular along $X$ and since $\sigma(X)=\p^{2n+1}$, we have 
    $\deg(\tau(X))\geq 4$.  Indeed, if $\deg(\tau(X))=2$, then $X\subset\p^{2n+1}$ would be degenerated  being contained in $\Sing(\tau(X))$. If $\deg(\tau(X))=3$, then the secant variety of $X$ would be contained in $\tau(X)$
    because $\tau(X)$ is singular along $X$.
    Since $V(Q)\subset \mathbb P^{2n+1}$ is a quartic hypersurface to prove that $Q$ is irreducible and that $\tau(X)=V(Q)$ it will be sufficient  to show that $\tau(X)\subset V(Q)$. 
    
   The quartic form  $Q$ is invariant for the action of the conformal group of $X$  on $\mathbb P^{2n+1}$ (the proofs given in \cite{Freudenthal} or in  \cite[\S7]{clerc} concern a priori only the semi-simple cases but can be applied in full generality).  Since the orbit of $0_{X}\in X$ under the action of 
    ${\rm Aut}(X)$  is Zariski-open in $X$, it is  enough to prove that a line in $\mathbb P^{2n+1}$ tangent to $X$ at $0_{X}$ is included in $V(Q)$.  A point of such a line has  homogeneous coordinates $p_v=(1,e+v, e+e\# v,1+T(v))$ with $v\in \mathbb J$. 
A straightforward (but a bit lenght) computation implies  that $Q(p_v)=0$ for every $v$, proving the result.
\end{proof}

    \begin{prop}
    \label{P:DefGeomRadX}
    The tangent hypersurface $\tau(X)$ is a quartic cone of vertex $R_X$.
    \end{prop}
    \begin{proof}
    It follows from (\ref{E:RadicalRank3}) that for any $r_x,r_y\in R={\rm Rad}(\mathbb J)$, one has
   $Q(\alpha,x+r_x,y+r_y,\beta)
    =Q(\omega)$ for every $\omega=(\alpha,x,y,\beta)\in Z_2(\mathbb J)$. This proves that $R_X$ is included in the vertex $V(d^3 Q)$ of $V(Q)=\tau(X)$. 
    
    Conversely,  let $(\partial_{x_1},\ldots,\partial_{x_n})$ (resp. $(\partial_{y_1},\ldots,\partial_{y_n})$) be the  system of partial derivatives naturally associated to  a system of linear coordinates on the first  (resp. on the second) 
    $\mathbb J$-summand of $Z_2(\mathbb J)$. The relations 
    $\partial_\alpha^2\partial_\beta Q(\omega)=\partial_\alpha\partial^2_\beta Q(\omega)=0$ imply that $\alpha=\beta=0$.  The set of relations   $\partial_{\alpha}\partial^2_{x_ix_j}Q(\omega)=-\partial_{x_i}\partial_{x_j}N(x)=0$, $i,j=1,\ldots,n$, can be summarized by $d^2N_x=0$, that is $x\in R= {\rm Rad}(\mathbb J)$ according to (\ref{E:RadicalRank3}). 
    Arguing similarly for $y$, one obtains that  $d^3Q_{\omega}=0$ implies that $
    \omega=( 0,x,y,0)$ with $x,y\in R$. This proves that the vertex $V(d^3Q)$ of $\tau(X)$ is included in $R_X$ and finishes the proof.      
    \end{proof}

In what follows, let $\mathbb J=\mathbb J_{\ss}\oplus R$ be the decomposition 
given in point {\it (2)} of Theorem  \ref{T:AlbertPenico}, where the embedding 
of  $\mathbb J_{\ss}=\mathbb J/R \hookrightarrow \mathbb J$  has been fixed once for all (hence is not indicated to simplify).  A straightforward verification gives that $R_X$ is nothing but the projectivization $\mathbb P(0\oplus R\oplus R\oplus 0)\subset \mathbb PZ_2(\mathbb J)$.  Hence setting $\mu_{\ss}=\mu_{\mathbb J_{\ss}}$  and    $X_{\ss}=X_{\mathbb J_{\ss}}$, we obtain the commutative
diagram
\begin{equation*}
    \xymatrix@R=0.85cm@C=1.7cm{  
\mathbb J  \,  \ar@{^{(}->}[r]^{\mu=\mu_{\mathbb J} } 
\ar@{->}[d]_{\pi_R}
   &\ar@{-->}[d]^{\pi_{R_X}}  X  &  \hspace{-2.4cm}\subset\mathbb PZ_2(\mathbb J)
     \\ 
 \mathbb J_{\ss}   \ar@{^{(}->}[r]^{\mu_{\ss}} 
   & X_{\ss}  &  \hspace{-1.9cm} \subset \mathbb PZ_2(\mathbb J_{\ss}),  
}
 \end{equation*}
   where  $\pi_R$   stands for the canonical linear projection $\mathbb J\rightarrow \mathbb J_{\ss}=\mathbb J/R$ and where $\pi_{R_X}$ denotes the restriction to $X$ of the linear projection $\mathbb PZ_2(\mathbb J)\dashrightarrow \mathbb PZ_2(\mathbb J_{\ss})$ from  the radical $R_X$ of $X$.  Since $\pi_R$ is surjective, this shows that 
$\pi_{R_X}(X)=X_{\ss}$.   By definition, $X_{\ss}$ is the {\it semi-simple  part} of $X$.  
\smallskip 

We have the following result, based on the classification of smooth varieties $X\in \overline{\boldsymbol{X}}(3,3)$.

\begin{prop} The semi-simple part  $X_{\ss}$ of $X$ belongs to $\overline{\boldsymbol{X}}(3,3)$ and is smooth. Moreover
\begin{enumerate}
\item[-] $r_{\ss}(\mathbb J)=3$ if and only if $X_{{\ss}}\in \boldsymbol{X}(3,3)$ hence 
is one of the varieties of
 the table of Theorem \ref{T:simple}; \smallskip
\item[-] $r_{\ss}(\mathbb J)=2$ if and only if $X_{{\ss}}$ is a scroll
$S_{1\ldots 122}$ of $S_{1\ldots 13}$ (in particular, $\dim(X)>1$); \smallskip  
\item[-] $r_{\ss}(\mathbb J)=1$ if and only if $X_{{\ss}}$ is the twisted cubic $v_3(\mathbb P^1)\subset \mathbb P^3$. \smallskip  
\end{enumerate}
\end{prop}

\begin{ex}[{\bf{continuation of Example \ref{Ex:Ce/e3}}}]
\label{Ex:RXA}
  {\rm  The radical  of the cubic curve $X_A\subset \mathbb P Z_2(A)=\mathbb P^7$ over $A=
  \mathbb C[\epsilon]/(\epsilon^3)$ is the  3-dimensional linear subspace 
  $R_{X_A}=\{ [0:r:r':0] \, \lvert \, r,r'\in R_A   \}\subset \mathbb P^7$.  The projection from $R_{X_A}$ induces a dominant rational map from $X_A$ onto  the twisted cubic curve $v_3(\mathbb P^1)\subset \mathbb P^3$ which is thus the semi-simple part of $X_A$.}
  \end{ex}

  \begin{ex} 
  \label{Ex:XH3sextonions}
  {\rm Let $\mathbb H_{\mathbb C}$ be  the complexification  of the  (real) algebra $\mathbb H$ of quaternions. Then $\mathbb H_{\mathbb C}$ can (and will) be identified with the complex algebra $M_{{2}}(\mathbb C)$ of $2\times 2$ complex matrices. For $M\in \mathbb H_{\mathbb C}$, let $\overline{M}$ stands for the adjoint $M-T(M){\bf Id}$. By definition, the  algebra of (complex) {\it sextonions} is the vector space $\mathbb S_{\mathbb C}=\mathbb H_{\mathbb C}\oplus M_{2\times 1}(\mathbb C)$ together with the product defined by   $(M,u)\cdot (N,v)=(MN,\overline{M}v+Nu)$ for $M,N\in \mathbb H_{\mathbb C}$ and $u,v\in M_{2\times 1}(\mathbb C)$, see \cite{landsbergmanivel2,westbury}.
This algebra can be embedded in the complexification $\mathbb O_{\mathbb C}$ of the octonions so that it  is alternative and it has an involution given  by $\overline{(M,u)}=(\overline{M},-u)$ for $(M,u)\in \mathbb S_{\mathbb C}$.  One  then defines the algebra ${\rm Herm}_3(\mathbb S_{\mathbb C})$ (see Example \ref{EX:jordan}.(\ref{L:Herm3})), which is  a rank 3 Jordan algebra of dimension 21. 

The cubic curve  $X=X_{{\rm Herm}_3(\mathbb S_{\mathbb C})}\subset \mathbb P^{43}$ over 
${\rm Herm}_3(\mathbb S_{\mathbb C})$ is singular along a smooth 10-dimensional quadric hypersurface, which spans the radical $R_{X}$ of $X$. Thus $\dim(R_X)=11$ (see  \cite[Corollary 8.14]{landsbergmanivel2}).  The semi-simple part of $X_{{\rm Herm}_3(\mathbb S)}$ is the orthogonal grassmannian variety $X_{{\rm Herm}_3(\mathbb H_{\mathbb C})}=OG_6(\mathbb C^{12})\subset \mathbb P^{31}$.}
\end{ex}

  \subsubsection{\bf  Radical ideals and extensions}  
   Let  $\boldsymbol{ I}$ be a proper projective subspace of $\mathbb P^{2n+1}$ and denote by $\pi_{\boldsymbol{ I}}: \mathbb P^{2n+1}\dashrightarrow \mathbb P^m$  the linear projection from ${\boldsymbol{ I}}$.  By definition, ${\boldsymbol{ I}}$ is a {\it radical ideal} for (or of) $X$ if $\boldsymbol{I}$ is included in  the radical $R_X$ of $X$ and if     the restriction of $\pi_{\boldsymbol{ I}}$ to $X$, again denoted by $\pi_{\boldsymbol{ I}}$, is such that $\pi_{\boldsymbol{ I}}(X)=\tilde{X}$ is still  3-RC by twisted cubics and extremal. 
     Since $X_{\ss}\in \overline{\boldsymbol{X}}(3,3)$, 
     $R_X$ itself is a radical ideal for $X$.      
   \smallskip 
   
  Let $I\subset \mathbb J$ be a radical ideal of the Jordan algebra $\mathbb J$.  Then $\overline{\mathbb J}=\mathbb J/I$ is a cubic Jordan algebra hence one can define the associated twisted cubic $X_{\overline{\mathbb J}}\subset \mathbb PZ_2(\overline{\mathbb J})\in\overline{\boldsymbol{X}}(3,3)$.   Then 
 $I\!\! I=\mathbb P(0\oplus I\oplus I\oplus 0)\subset \mathbb PZ_2(\mathbb J)$
  is a radical ideal of $X_{\overline{\mathbb J}}$.  More generally, the image of 
  such a
  $I\!\! I$ by any element of ${\rm Conf}(X)$  is again a radical ideal for $X$.

\begin{prop} 
\label{P:RadicalIdealX}
Any radical ideal  of $X$  comes  from a radical ideal of $\mathbb J$ by the construction presented above.
\end{prop}
\begin{proof}
 We continue to use the notation introduced above for $X$ and the tilded versions of these will stand for the corresponding notation for $\tilde{X}$. 

Let $I\!\! I\subset \mathbb P^{2n+1}$ be a non-trivial radical ideal for $X$.  Let $x^+$ and $x^-$ be two general points on $X$. 
For $\sigma=\pm$, one denotes by  $\pi^\sigma$  the differential of $\pi_{I\!\! I}$  at $x^\sigma$: since $x^\sigma$ is general, it is a well-defined surjective linear map from $V^\sigma$ onto $\tilde{V}^\sigma$ whose kernel $I^\sigma$ has dimension $i$.  We want to prove that $I=(I^+,I^-)$ is an ideal of the Jordan pair $V=(V^+,V^-)$ such that $V/I=(V^+/I^+,V^-/I^-)$ is isomorphic to $\tilde{V}=(\tilde{V}^+,\tilde{V}^-)$ (as Jordan pairs).

Let $\Sigma$ (resp. $\tilde \Sigma$) be the 3-covering family of twisted cubics on $X$ (resp. on $\tilde X$).   Then for $C\in \Sigma$ general, its image $\pi_{I\!\! I}(C)$ by the linear projection $\pi_{I\!\! I}$ is an irreducible rational  curve of degree $\leq 3$ included in $\tilde{X}$. Since the family  $\pi_{I\!\! I}(\Sigma)=\{\pi_{I\!\! I}(C)\}_{C\in \Sigma}$ is also 3-covering, it comes from \cite[Lemme 2.1]{PT} that $\pi_{I\!\! I}(\Sigma)=\tilde{\Sigma}$.  Since $x^+$ and $x^-$ are general points of $X$, 
we deduce  $\pi_{I\!\! I}(\Sigma_{x^+x^-})=\{\pi_{I\!\! I}(C)\, \lvert \, C\in \Sigma_{x^+x^-}\}=\tilde{\Sigma}_{\tilde{x}^+\tilde{x}^-}$.  Let $v\in V^\sigma$ be such that $F^\sigma$ and $\tilde{F}^\sigma$ 
are defined at $v$ and $\tilde{v}=\pi^\sigma(v)$ respectively. Let $\alpha_v: \mathbb P^1\rightarrow X$ be the projective parametrization of the twisted cubic element of 
$\Sigma_{x^+x^-}$ such that $d\alpha_v(s:1)/ds\lvert_{s=0}=v$ (see Section   \ref{S:Xradical}).  Then $\tilde{\alpha}_v=\pi_{I\!\! I}\circ \alpha_v: \mathbb P^1\rightarrow \tilde{X}$ is a projective parametrization of a twisted cubic in $\tilde{X}$ such that $\tilde{\alpha}_v(0)=\tilde{x}^\sigma$, $\tilde{\alpha}_v(\infty)=\tilde{x}^{-\sigma}$
 and $d\tilde{\alpha}_v(s:1)/ds\lvert_{s=0}=\pi^{\sigma}(v)=\tilde{v}$: with the notation of Section  \ref{S:Xradical}, one has $\tilde{\alpha}_v=\alpha_{\tilde{v}}$.  This implies that 
 $\tilde{F}^\sigma\circ \pi^\sigma=\pi^{-\sigma}\circ F^\sigma$ for $\sigma=\pm$.  Taking total derivatives, one obtains  
 $d(\tilde{F}^\sigma)_{\pi^\sigma(\cdot)} \circ \pi^\sigma=\pi^{-\sigma} \circ dF^\sigma$.
 Combined with the fact that $F^{-\sigma}\circ F^\sigma={\rm Id}_{V^\sigma}$ and 
 $\tilde{F}^{-\sigma}\circ \tilde{F}^\sigma={\rm Id}_{\tilde{V}^\sigma}$, this gives that 
 for  $v\in V^\sigma$ general: 
 \begin{align*} 
\big( d(\tilde{F}^\sigma)_{\pi^{\sigma}(v)}\big)^{-1}\circ \pi^{-\sigma}= & \, 
d( \tilde{F}^{-\sigma})_{\tilde{F}^{\sigma}\circ \pi^{\sigma}(v)} \circ \pi^{-\sigma}  \\
= & \, 
d( \tilde{F}^{-\sigma})_{\pi^{-\sigma}\circ F^\sigma(v)} \circ \pi^{-\sigma} \\
= & \, 
d( \tilde{F}^{-\sigma}\circ \pi^{-\sigma})_{ F^\sigma(v)}   = d(\pi^\sigma\circ F^{-\sigma}) _{ F^\sigma(v)}=\pi^\sigma\circ \big(d(F^\sigma)_v\big)^{-1}.
 \end{align*}

By density, this series of equalities implies that for $\sigma=\pm$, one has
$$
\tilde{P}^\sigma_{\pi^{\sigma}(v )}\circ \pi^{-\sigma}=\pi^\sigma \circ P^\sigma_v
$$ 
for every $v\in V^\sigma$, 
 where $P^\sigma$ and $\tilde{P}^\sigma$ stand for the quadratic operators of the Jordan pairs $V$ and $\tilde{V}$ respectively (see Remark \ref{R:CtoJPair}).  According to  Definition 1.3 in \cite{loos}, this means that $\pi=(\pi^+,\pi^-):V\rightarrow \tilde{V}$ is a surjective morphism  of Jordan pairs.  Consequently, ${\ker}(\pi)=(I^+,I^-)$ is an ideal of $V$ and $V/I\simeq \tilde{V}$.  For $\sigma=\pm$, let $\boldsymbol{I}^\sigma$ be the  closure  of $I^\sigma$ in $V^\sigma\subset T_{x^\sigma} X\subset \mathbb P^{2n+1}$.  We let  the reader verify that the radical ideal $I\!\! I$ from the beginning is  nothing but the direct sum $\boldsymbol{I}^+\oplus \boldsymbol{I}^-\subset T_{x^+} X \oplus T_{x^-} X=\mathbb P^{2n+1}$,  concluding the proof. 
 \end{proof}

If $I\!\! I\subset \mathbb P^{2n+1}$ is a radical ideal for $X$, we will say that $X$ is a  {\it (radical) extension} of $\tilde{X}=\pi_{I\!\! I}(X)\in \overline{\boldsymbol{X}}(3,3)$ by $I\!\! I$. 
 This extension is {\it null} if the generic fiber of $\pi_{I\!\! I}: X\dashrightarrow \tilde{X}$ is a linear subspace in $\mathbb P^{2n+1}$.  It is {\it split} if there exists a linear embedding $\boldsymbol{\iota}: \langle \tilde{X}\rangle =\mathbb P^{2\tilde{n}+1}\hookrightarrow \mathbb P^{2n+1}$ the image of which  is supplementary to $I\!\! I$ and is such that $\pi_{I\!\! I}\circ \sigma$ induces the identity when restricted to $\tilde{X}$.

\begin{prop} The $XJ$-correspondence induces  correspondences between 
null  (respectively  split) extensions in the $J$-world and in the $X$-world.
\end{prop}
\begin{proof} Let $\pi_{I\!\!I}:X\dashrightarrow \tilde X$ be a (radical) extension in the $X$-world corresponding to a radical extension
$0\rightarrow I \rightarrow {\mathbb J}\rightarrow {\tilde{ \mathbb J}}=\mathbb J/I \rightarrow 0$ in the  $J$-world (we use here Proposition \ref{P:RadicalIdealX}).   
 One can assume that $X=X_{\mathbb J}$  and  $\tilde{X}=X_{\tilde {\mathbb J}}$ 
and that the projection from ${I\!\!I}$ is induced by the linear map $Z_2(\mathbb J)\rightarrow Z_2(\tilde{\mathbb J})$ coming from the canonical projection $\mathbb J\rightarrow \mathbb J/I$. Let us denote by  $\tilde{x}$ the  class in $\tilde{ \mathbb J}$ of an element $x\in 
\mathbb J$. Since $I$ is assumed to be a radical ideal, $x\#i$ and $i^\#$ belong to $I$ for every $x\in \mathbb J$ and $i\in I$. In particular 
the class of $x^\#$ in $\mathbb J/I$ is $\tilde{x}^\#$. 
From this it follows that  $\pi_{I\!\!I}:X_{\mathbb J}\dashrightarrow X_{\mathbb J/I}$ is given  by 
$[1:x:x^\#:N(x)]  \longmapsto [  1:\tilde{x}:\tilde{x}^\#:N(x)   ]$ hence for any $x\in \mathbb J$, the fiber of $\pi_{{I\!\!I}}:X\dashrightarrow \tilde{X}$ over $\tilde{p}_{x}=
[  1:\tilde{x}:\tilde{x}^\#:N(x)   ]\in \tilde{X}$ is   
 $$\pi_{{I\!\!I}}^{-1}(\tilde{p}_x)=\overline{\big\{ [1:x+i: x^\#+x\#i+i^\#:N(x)]\, \big\lvert \, i\in I\, \big\}}\subset \mathbb PZ_2(\mathbb J).$$
   This fiber  is a linear subspace in $\mathbb PZ_2(\mathbb J)$ if and only if  $i^\#=0$ for every $i\in I$. Since  $I$ is radical,  $i^\#=i^2$ for every $i\in I$ so that $i^\#=0$ for every $i\in I$ if and only if $I^2=0$ (remember that any product in $I$ can be expressed as a linear combination of squares). This proves the proposition for the case of null extensions. \smallskip 

We now consider the case of split extensions. Clearly a split extension in the $J$-world yields
 a split extension in the $X$-world. On the contrary, assume that $\boldsymbol{\iota}: \langle \tilde{X}\rangle \hookrightarrow \mathbb P^{2n+1}$ is a splitting of an extension $\pi_{{I\!\!I}}:X\dashrightarrow \tilde{X}$.  If    $\tilde{x}^+$ and $\tilde{x}^-$ are two general points of $\tilde{X}$, then $x^+=\boldsymbol{\iota}(\tilde{x}^+)$ and 
$x^-=\boldsymbol{\iota}(\tilde{x}^-)$ are two points of $X$ for which the construction of the Jordan pair $(V^+,V^-)$ described in the proof of Proposition \ref{P:RadicalIdealX} can be performed.
For $\sigma=\pm$, let ${\iota}^\sigma: \tilde{V}^\sigma\hookrightarrow V^\sigma$ be the differential  at $\tilde{x}^\sigma$ of the restriction of $\boldsymbol{\iota}$ to $\tilde{X}$.  Since $\boldsymbol{\iota}$ is a linear embedding, it sends any twisted cubic included in $\tilde{X}$ onto a twisted cubic in $X$. From this,  one deduces that 
${\iota}=({\iota}^+,{\iota}^-): \tilde{V}\rightarrow V$ is an injective morphism of Jordan pairs. Finally,  from the fact that  ${\rm Im}(\boldsymbol{\iota})$ and $I\!\!I$ are supplementary in $\mathbb PZ_2(\mathbb J)$, one deduces that  $ {\iota}: \tilde{V}\rightarrow V$ gives a splitting of the extension of Jordan pairs $I\hookrightarrow V\stackrel{\pi}{\twoheadrightarrow} \tilde{V}$, concluding
the proof of all the assertions.
\end{proof}

\begin{ex} [{\bf{continuation of Example \ref {Ex:XH3sextonions}}}]
 {\rm  The decomposition $\mathbb S_{\mathbb C}=\mathbb H_{\mathbb C}\oplus \mathbb U_{\mathbb C}$ (with  $\mathbb U_{\mathbb C}=M_{2\times 1}(\mathbb C)$)  induces a decomposition  in direct sum 
 ${\rm Herm}_3(\mathbb S_{\mathbb C})={\rm Herm}_3(\mathbb H_{\mathbb C})\oplus 
 {\rm Alt}_3(\mathbb U_{\mathbb C})$ where ${\rm Alt}_3(\mathbb U_{\mathbb C})$ is the space of antisymmetric $3\times 3$ matrices with coefficients in $\mathbb U_{\mathbb C}$.   One has $\Rad({\rm Herm}_3(\mathbb S_{\mathbb C}))={\rm Alt}_3(\mathbb U_{\mathbb C})$ and ${\rm Herm}_3(\mathbb S_{\mathbb C})$ is a split and  null extension of ${\rm Herm}_3(\mathbb H_{\mathbb C})$ by ${\rm Alt}_3(\mathbb U_{\mathbb C})$.  The geometrical interpretation of this is that 
   $X=X_{{\rm Herm}_3(\mathbb S_{\mathbb C})}\subset \mathbb P^{43}$  is a split and null extension of $X_{{\rm Herm}_3(\mathbb H_{\mathbb C})}=OG_6(\mathbb C^{12})$. In this particular case,  the linear projection $\pi_{R_X}: X\dashrightarrow X_{{\rm Herm}_3(\mathbb H_{\mathbb C})}$   is surjective and any  of its fibers is a linear subspace of dimension $6$ in $\mathbb P^{43}$.}
\end{ex}

 \subsubsection{\bf  The Penico series of $X$}
 We use in this subsection the notation introduced in the proof 
 of Proposition \ref{P:RadicalIdealX}:  $x^+$ and $x^-$ are two general points on $X$, etc.   Let  $I\!\! I\subset R_X$ be a radical ideal associated to the radical ideal $(I^+,I^-)$ of the Jordan pair $V=(V^+,V^-)$.  For $\sigma=\pm$, define $\mathcal P(I^\sigma)$ as $P^\sigma_{I^\sigma}(V^{-\sigma})\subset V^\sigma$. Then $\mathcal P(I\!\! I)=\mathbb P(\mathcal P(I^+)\oplus \mathcal P(I^-))\subset \mathbb P^{2n+1}$ is a radical ideal of $X$.  \smallskip 

We can now define the {\it Penico series} of $X$ as  the decreasing  family  of projective subspaces $\mathbb P^{2n+1}\supset R^{[1]}_X \supset R^{[2]}_X \supset \cdots  \supset  R^{[\ell-1]}_X 
\supset R^{[\ell ]}_X \supset \cdots$ defined inductively by 
$$R_X^{[1]}=R_X \quad \mbox{ and } \quad 
R_X^{[k+1]}=\mathcal P\big(R_X^{[k]}\big) \quad \mbox{ for } k\geq 1.$$

One verifies that the $R^{[k]}_X$'s do not depend on the base points $x^+$ and $x^-$ but only on $X$ and that they are projectively attached to $X$.  It would be interesting to give a purely geometrical  characterization of the Penico series of $X$, in the same spirit of 
the characterization of the radical  of $X$ given in Proposition \ref{P:DefGeomRadX}.

\begin{ex}[{\bf{continuation of Example \ref{Ex:Ce/e3}}}]
\label{Ex:Penico-X}
  {\rm  We have seen  in Example \ref{Ex:RXA} that 
 the radical $R_{X_A}$ of the cubic curve $X_A\subset \mathbb P Z_2(A)=\mathbb P^7$ over $A=
  \mathbb C[\epsilon]/(\epsilon^3)$ is a 3-dimensional projective subspace in $\mathbb P^7$. One verifies easily $R^{[2]}_{X_A}$ is the projective line 
  $\{ [0:\lambda\,   \epsilon^2 :\lambda'\,   \epsilon^2:0]\, \lvert \, [\lambda,\lambda']\in \mathbb P^1 \}\subset R_{X_A}$, whereas  $R^{[\ell]}$ is empty for every $\ell\geq 3$. }
  \end{ex}

 For any $\ell\geq 1$, let $p^{[\ell]}$ be the restriction to $X$ of the linear projection from $R_X^{[\ell]}$ and denote by $X^{[\ell]}$ its  image (note that $X^{[\ell]}=X$ and 
 $p^{[\ell]}={\rm Id}_X$ if $R^{[\ell]}_X$ is empty).    Since  $R_X^{[\ell]}$ is  a radical ideal, 
$X^{[\ell]}$ belongs to $\overline{\boldsymbol{X}}(3,3)$. Moreover, it is not difficult to verify that for $k<\ell$, the linear subspace $p^{[\ell]}(R_X^{[k]})$ is a radical ideal for $X^{[\ell]}$ that is nothing but $R_{X^{[\ell]}}^{[k-\ell]}$.  If one denotes by  $\pi^{[\ell]}$  the restriction to $X^{[\ell]}$ of the linear projection from 
$p^{[\ell]}(R_X^{[\ell-1]})$, there is a commutative diagram 
 \begin{equation*}
    \xymatrix@R=0.7cm@C=1.7cm{  
 X  \ar@{-->}[r]^{p^{[\ell]} }  \ar@{-->}[rd]_{p^{[\ell-1]} }  & X^{[\ell] }
  \ar@{-->}[d]^{\pi^{[\ell]} } 
 \\ 
   &  X^{[\ell-1] }\, , 
 }
 \end{equation*}
where the maps in it are (restrictions of) dominant linear projections 
 sending isomorphically a general twisted cubic curve in a source space
    onto a general twisted cubic curve in the target space. \smallskip 
    
    A good reason  to consider the projections $\pi^{[\ell]}$ is certainly given by the following result,
    whose proof is left to the reader.
    
    \begin{prop}\label{P:linearfiber} For any $\ell\geq 1$, the generic fiber of the rational map 
   $\pi^{[\ell]}:  X^{[\ell]}\dashrightarrow  X^{[\ell-1]}$ is a linear subspace.
       \end{prop}
    
      \subsubsection{\bf  The general structure of twisted cubics over Jordan algebras}  
    
    We are now in position of stating the translation in the $X$-world of Theorem \ref{T:AlbertPenico}.

     \begin{thm} 
 \label{T:structure-X}  Assume that $X\in \boldsymbol{X}(3,3)$ is not semi-simple (or equivalently that $X$ is not smooth).  Then 
 \begin{enumerate}
 \item[(1)] the restriction to $X$ of the linear projection from its radical $R_X$ induces a dominant rational map
 $$
 \pi_{R_X}: X\dashrightarrow {X}_{\ss}
 $$
 over a semi-simple 3-RC variety ${X}_{\ss}\in \boldsymbol{\overline{X}}(3,3)$, the restriction of which to a general twisted cubic $C \subset X$ is an isomorphism onto its image $\pi_{R_X}(C)$, which is then a twisted cubic curve in ${X}_{\ss}$; 
 \medskip 
  \item[(2)] there exists a linear embedding $\sigma: \mathbb PZ_2(\mathbb J_{\ss})\hookrightarrow \mathbb PZ_2({\mathbb J})$ whose image is supplementary to $R_X$  such that  
  $$\sigma\big({X}_{\ss}\big)\subset X
  \qquad \mbox{and}\qquad 
 \pi_{R_X}\circ \sigma\in {\rm Aut}\big(X_{\ss}\big).$$ 
  
  Moreover, $\sigma$ is unique, up to composition to the left by a projective automorphism of $X$;
 \medskip 
  \item[(3)]  the radical $R_X$ is solvable: there exists a positive integer 
  $t$ such that $R_X^{[t]}$ is empty.  Moreover,  $X^{[\ell]}$ is a radical null extension of $X^{[\ell-1]}$ for $\ell=2,\ldots,t$ so that   $X=X^{[t]}$ can be obtained from its semi-simple part $X_{\ss}=X^{[1]}$ by the successive series of null radical extensions represented below
  \begin{equation*}
    \xymatrix@R=0.85cm@C=1.17cm{  
 X  \ar@{-->}[r]^{\pi^{[t]} }  & X^{[t-1]}  \ar@{-->}[r]^{\pi^{[t-1]} } & \cdots   \ar@{-->}[r]^{\pi^{[\ell+1]} } &  X^{[\ell]}  
 \ar@{-->}[r]^{\pi^{[\ell]} } &  X^{[\ell-1]} \ar@{-->}[r]^{\pi^{[\ell-1]} } &  \cdots 
 \ar@{-->}[r]^{\pi^{[3]} }
  & X^{[2]}  \ar@{-->}[r]^{\pi^{[2]} }
& X_{\ss}.
 }
 \end{equation*}
 \end{enumerate}
 \end{thm}
    \smallskip 

  \subsubsection{\bf  Null extension and Verra construction}

  It follows from part  {\it (3)}  of  Theorem \ref{T:structure-X}  that the notion of radical null extension  is particularly relevant when dealing with varieties in the class $ {\boldsymbol{X}}(3,3)$. 
 Notwithstanding  the considerations and results of the preceding sections are not fully satisfying from the intrinsic point of view. For instance the  construction of all  radical null extensions of a given $X\in {\boldsymbol{X}}(3,3)$ shows immediately that it is desirable to have 
  an intrinsic  geometric characterization of such objects, the term `intrinsic' meaning here `in term of $X$ alone'.   This section is dedicated to this purpose. 
  \smallskip 
    
    Let us recall that if $X\in {{\boldsymbol{X}}}(3,3)$, then $X\subset\p^{2n+1}$ is a {\it variety with one apparent double point}, briefly an {\it OADP--variety},
meaning that through a general point of $\p^{2n+1}$ there passes a unique secant line to $X$, 
see for example \cite[Corollary 5.4]{PR} for a proof\footnote{The name  `OADP--variety'  comes from the fact that the projection
of $X$ from a general point acquires only one double point as (further) singularities (see also \cite{CR} for relations between twisted cubics over Jordan algebras and OADP-varieties).}. If $\pi:X'\dashrightarrow X$  is a radical null extension in the $X$-world then 
 the general fiber of $\pi$ is a linear subspace.  This shows that  $X'$ is an OADP-variety obtained from $X$ by the so called {\it Verra construction} of new OADP-varieties from a given one. We recall brievely this geometric construction below, 
 referring to \cite[\S 3]{CR} for more details and proofs. \medskip 

Let  $Y\subset \p^{2(n+r)+1}$ be a degenerate OADP--variety of dimension $n$,
which spans a linear space $V$ of dimension $2n+1$. 
Let $\mathcal C_W(Y)$ be the cone over $Y$ with vertex a linear space $W\subset \p^ {2(n+r)+1}$ 
of dimension $2r-1$ in direct sum with $V$.   Assume that $Y' \subset \mathcal C_W(Y)$ is an irreducible non--degenerate  variety of dimension $n+r$, that is not secant defective and 
which intersects the general ruling $\Pi\simeq \p^ {2r}$ of $\mathcal C_W(Y)$ 
along a linear subspace of dimension $r$.
Then the linear projection of $\p^ {2(n+r)+1}$ 
from $W$ onto $V$ restricts to $Y'$ to a dominant map $\pi: Y'\map Y$ having  linear fibers of dimension $r$ that are generically disjoint.
 We shall say that $Y'$ {\it is obtained from $Y$ via Verra's construction} or also that $Y'$ is a {\it Verra variety}. It is not difficult to prove that Verra varieties are OADP--varieties.\smallskip 

Let $Y'$ be a Verra variety as above. Then $\pi^{-1}(y)$ is a linear subspace of dimension $r-1$ of $W$ for $y\in Y$ general. Therefore $y\mapsto \pi^{-1}(y)$ defines a rational map 
$ \gamma_{Y'}: Y\dashrightarrow \mathbb G(r-1,W)=G_{r}(\mathbb C^{2r}) $.  Moreover,  $ \gamma_{Y'}(y_1)$ and $ \gamma_{Y'}(y_2)$ are skew subspaces of $W$ when $y_1$ and $y_2$ are two  general points of $Y$.
Conversely, let  $\mathcal V_Y^r$ be the set of rational maps
$\gamma: 
Y\dashrightarrow \mathbb G(r-1,\mathbb P^{2r-1}) $ satisfying the condition that  
$ \gamma(y_1)$ and $ \gamma(y_2)$ are skew if $y_1$ and $y_2$ are general in $Y$.   
It can be verified that  for any $\gamma\in \mathcal V_Y^r$,  if  $Y_0$ stands for the open subset of $ Y$ on which  $\gamma$ is defined, then 
$$
Y_{\!\gamma}=\overline { \bigcup_{y\in Y_0}\big\langle y,\gamma(y)\big\rangle }\subset \mathbb P^{2(n+r)+1}
$$
is an OADP--variety that is  obtained from $Y$ by   Verra's construction.  
This gives an identification between the set of $(n+r)$-dimensional Verra varieties constructed from $Y$ up to projective equivalence and the quotient of the set $\mathcal V_Y^r$ by a  certain  relation of equivalence that the interested reader 
could make explicit without difficulty. \smallskip 

Since a radical null extension $\pi:X'\dashrightarrow X$ is obtained by Verra's construction from $X$, there exists $\gamma_{X'}\in \mathcal V_X^r$ such that $X'=X_{\gamma_{X'}}$.  Nevertheless one verifies easily that not every $\gamma\in  \mathcal V_X^r$ is such that $X_\gamma$ is a radical null extension of $X$, see also Example \ref{Ex:simplext} below. A necessary and sufficient 
assuring that $X'\in {{\boldsymbol{X}}}(3,3)$ is given by the following result:

\begin{thm} Let $X\in  {{\boldsymbol{X}}}(3,3)$ and let $\gamma\in  \mathcal V_X^r$. The following conditions are equivalent: 
\begin{enumerate}
\item  the Verra variety $X_{\gamma}$ belongs to ${\boldsymbol{X}}(3,3)$ so that in particular  is a radical null extension of $X$\smallskip 
\item  the restriction  of $\gamma$ to a general cubic curve $C\subset X$ is an embedding and  $\gamma(C)$ is a line in  $G_r(\mathbb C^{2r})$.
\end{enumerate}
\end{thm}
\begin{proof}
Clearly  {\it (1)} implies {\it (2)} and we now  prove the converse. 
Let $X'=X_{\gamma}$ where $\gamma\in \mathcal V_X^r$ is such that {\it (2)} holds
and denote by $\pi:X'\dashrightarrow X$  the (restriction of the) linear projection defining $X'$ as a Verra variety over $X$. 
If $x_1',x_2'$ and $x_3'$ are three general points on $X'$ then the $x_i=
\pi(x_i')$'s are three general points on $X$.  Let $C$ be the twisted cubic included in $X$ passing trough the $x_i$'s.  Clearly, $\pi^{-1}(C)$ is nothing but the Verra variety $C_{\gamma_C}$ over $C$,  where $\gamma_C$ stands for the restriction of $\gamma$ to $C$.  Since $C$ is a general cubic in $X$, it follows from {\it (2)} that $C_{\gamma_C}$ is the $(r+1)$-dimensional rational normal scroll $S_{1\ldots 13}$ in $ \mathbb P^{2r+3}$. The later being 
an element of the class $\overline{\boldsymbol{X}}(3,3)$, there exists a twisted cubic $C'\subset C_{\gamma_C}$ passing trough $x_1',x_2'$ and $x_3'$ and such $\pi(C')=C$.  This shows that $X'=X_\gamma$ is 3-RC by cubic curves.  Then $X_\gamma$ is a radical extension of $X$. Since   the general fiber of $\pi:X_\gamma\dashrightarrow X$ is linear,  this radical extension is null by the definition, concluding the proof.
\end{proof}

The following examples show that there exist $\gamma\in  \mathcal V_X^r$ such that $X_{\gamma}\not\in\overline{\boldsymbol{X}}(3,3)$.

\begin{ex}
\label{Ex:simplext}{\rm
Let $\mathbb J$ be a rank 3 Jordan algebra of dimension $n\geq 1$ with generic norm $N(x)$. Let $\mathbb J'$ be a  split radical extension of $\mathbb J$ by a Jordan bimodule $R$ of dimension 1.

First of all, since $R^2\subset R$ and $r^3=0$ for every $r\in R$ (and since $R\subset {\rm Rad}(\mathbb J')$), it follows that $R^2=0$.  Because the extension $R\hookrightarrow \mathbb J'\twoheadrightarrow \mathbb J$ is split, one can assume that  $\mathbb J'= \mathbb J\oplus R$ with product 
given by 
\begin{equation*}
\label{eq:productJprime}
\qquad 
(x_1,r_1)*(x_2,r_2)=(x_1*x_2,r_1\varphi(x_2)+r_2\varphi(x_1)) \qquad \mbox{for }\, x_1,x_2\in \mathbb J, \, r_1,r_2\in R, 
\end{equation*}
where  $\varphi:\mathbb J\to\mathbb C$ is a certain (fixed) linear form. 
The unity of $\mathbb J'$ is $e'=(e,0)$, yielding   $\varphi(e)=1$. 

Reasoning as in the proof of Theorem \ref{T:structureBir22} and recalling that $r^2=0$ for every $r\in R'$, we deduce that
there exists also a linear form $f:\mathbb J\to\mathbb C$ such that
$(x,r)^\#=(x^\#, rf(x))$ for every $(x,r)\in \mathbb J'$. 
By definition of radical  we have $N(x')=N(x,r)=N(x)$ for $x'=(x,r)\in \mathbb J'$ so that the identity $N(x')x'=({x'}^\#)^\#$ is equivalent to
\begin{equation}\label{eq:structurenormJprime}
N(x)=f(x)f(x^\#)
\end{equation}
for every $x\in\mathbb J$.

Now if $X_{\mathbb J}\subset\p^{2n+1}$ is a twisted cubic over a simple rank three Jordan algebra $\mathbb J$ (so with  $n\in \{6,9,15,27\}$) 
 and if $Z^{n+1}\subset\p^{2n+3}$ is obtained from $X_{\mathbb J}$
via Verra construction, then $Z\not\in\overline{\boldsymbol{X}}^{n+1}(3,3)$.
Indeed, otherwise $Z$ would be projectively equivalent to $X_{\mathbb J'}$ 
for a certain 1-dimensional null radical extension ${\mathbb J'}$  of $\mathbb J$, that is necessarily split (this follows from the fact that $\mathbb J$ is simple).  Then \eqref{eq:structurenormJprime} would imply that the norm $N(x)$ of $\mathbb J$ is a reducible polynomial of degree 3,
which is not the case. 
}
\end{ex}

Part  {\it (3)} of   Theorem \ref{T:structure-X} points out that 
among radical null extensions the split ones  are the most interesting. 
 Accordingly to the general principle of the XJC-correspondence, given $X\in \boldsymbol{X}(3,3)$, it would be interesting to get a characterization in geometric terms of the rational maps $\gamma\in  \mathcal V_X^r$ satisfying condition {\it (2)} of the preceding theorem and such that the radical extension $X_\gamma\dashrightarrow X$ is not only null but also split.  We  intend to return on this and on other related questions in the future.

\end{document}